\documentclass{amsart}
\pdfoutput 1
\usepackage[all,hyperref,numberbysection]{bi-discrete}
\usepackage{amsmath}
\usepackage{upgreek}
\usepackage{microtype}
\numberwithin{equation}{section}
\usepackage[a4paper,margin=24ex]{geometry}
\allowdisplaybreaks
\usepackage[backend=biber,maxbibnames=5,maxalphanames=5,style=alphabetic,bibencoding=utf8,giveninits,url=false,isbn=false]{biblatex}
\AtBeginBibliography{\small}

\begin{document}

\author[Kim]{Kyeongbae Kim}
\address{Department of Mathematical Sciences, Seoul National University, Seoul 08826, Korea}
\email{kkba6611@snu.ac.kr}
\author[Lee]{Ho-Sik Lee}
\address{Fakult\"at f\"ur Mathematik, Universit\"at Bielefeld, Postfach 100131, D-33501 Bielefeld, Germany}
\email{ho-sik.lee@uni-bielefeld.de}
\author[Nowak]{Simon Nowak}
\address{Fakult\"at f\"ur Mathematik, Universit\"at Bielefeld, Postfach 100131, D-33501 Bielefeld, Germany}
\email{simon.nowak@uni-bielefeld.de}

\makeatletter
\@namedef{subjclassname@2020}{\textup{2020} Mathematics Subject Classification}
\makeatother


\thanks{Kyeongbae Kim thanks for funding of the National Research Foundation of Korea (NRF) through IRTG 2235/NRF-2016K2A9A2A13003815 at Seoul National University. Ho-Sik Lee thanks for funding of the Deutsche Forschungsgemeinschaft through GRK 2235/2 2021 - 282638148. Simon Nowak gratefully acknowledges funding by the Deutsche Forschungsgemeinschaft (DFG, German Research Foundation) - SFB 1283/2 2021 - 317210226.
}

\title{Gradient estimates for nonlinear kinetic Fokker-Planck equations}

\begin{abstract}
	In this work, we provide a comprehensive gradient regularity theory for a broad class of nonlinear kinetic Fokker-Planck equations. We achieve this by establishing precise pointwise estimates in terms of the data in the spirit of nonlinear potential theory, leading to fine gradient regularity results under borderline assumptions on the data. Notably, our gradient estimates are novel already in the absence of forcing terms and even for linear kinetic Fokker-Planck equations in divergence form.
\end{abstract}

\maketitle



\section{Introduction}
\subsection{Nonlinear kinetic Fokker-Planck equations}
This aim of this paper is to undertake a thorough investigation of the gradient regularity of weak solutions $f=f(t,x,v):W \times V \subset \bbR^{n+1}\times \bbR^n\to\bbR$ $(n\geq 1)$ to equations of the type
\begin{equation}\label{eq.main}
	\partial_t f+v\cdot \nabla_x f-\divergence_v(a(t,x,v,\nabla_v f))=\mu -\divergence_v G \quad\text{in }W\times V\subset \bbR^{n+1}\times \bbR^n.
\end{equation}
Here $W$ and $V$ are open sets, $\mu=\mu(t,x,v):\bbR^{n+1}\times \bbR^n\to\bbR$ is a given function and $G=G(t,x,v):\bbR^{n+1}\times \bbR^n\to\bbR^n$ is a given vector field. Moreover, throughout the paper we at least assume that the vector field $a=a(t,x,v,\xi)$ satisfies the following.

\begin{assumption} \label{assump}
	We assume that $a(t,x,v,\xi)$ is measurable in $(t,x,v) \in \bbR \times \bbR^{n}\times \bbR^n$ and $C^1$-regular in $\xi\in \bbR^n$ such that for all $t \in \mathbb{R}$ and all $x,v,\xi,\zeta \in \mathbb{R}^n$, we have
	\begin{align}\label{ass.coef}
		\begin{cases}
			|a(t,x,v,\xi)|&\leq \Lambda|\xi|,\\
			|\nabla_\xi a(t,x,v,\xi)|&\leq\Lambda,\\
			\nabla_\xi a(t,x,v,\xi)\zeta\cdot \zeta&\geq\Lambda^{-1}|\zeta|^2
		\end{cases}
	\end{align}
	for some $\Lambda \geq 1$.
\end{assumption} 
In the case when $a(t,x,v,\xi)=\xi$, that is, when $\divergence_v(a(t,x,v,\nabla_v f))$ reduces to the standard Laplacian with respect to the velocity variable $v$, equations of the type \eqref{eq.main} were first studied by Kolmogorov in \cite{Kolmogorov}, who in this particular case provided an explicit fundamental solution. Since the additional transport term leads to a lack of ellipticity with respect to the spatial variable $x$ in comparison to the corresponding parabolic setting, already in the case when the diffusion is modeled by the Laplacian the equation \eqref{eq.main} rather exhibits hypoelliptic features in the sense of H\"ormander (see \cite{Hormander}).
In addition, linear kinetic Fokker-Planck equations are closely linked to the theory of stochastic processes, leading to applications in physics (see e.g.\ \cite{Chandra,LN22,AlbArmMouNov24}). Moreover, they are related to the Landau equation from kinetic theory (see e.g.\ \cite{Villani,HenSne20,Sil23}).

In the case of elliptic and parabolic equations, it is common to model diffusion also by nonlinear operators as the one appearing in \eqref{eq.main}, since nonlinear divergence form operators arise naturally through variational principles. For this reason, in \cite{GarNys23} the authors initiated a systematic study of \emph{nonlinear} kinetic Fokker-Planck equations of the type \eqref{eq.main}, proving in particular H\"older estimates for the solution itself. Although most of our results are new already in the linear setting, the primary goal of the present paper is to establish a fine gradient regularity theory for nonlinear kinetic equations via tools from \emph{nonlinear potential theory}. In particular, this allows us to extend both the zero-order regularity theory for nonlinear kinetic equations from \cite{GarNys23} to the gradient level and the first-order nonlinear potential theory developed in the elliptic and parabolic setting (see \cite{Min11,DuzMin11a,KuuMin12,BCDKS}) to the realm of kinetic equations.

\subsection{Schauder-type velocity gradient regularity for homogeneous kinetic Fokker-Planck equations} \label{hom}

While a main focus of the present work consists of obtaining fine regularity results for the velocity gradient of weak solutions to \eqref{eq.main} in terms of the data $\mu$ and $G$, we want to begin by emphasizing that our results are already completely new in the homogeneous case when $\mu \equiv 0$ and $ G \equiv 0$. Indeed, in this case we obtain the following gradient regularity result that serves as a model case and will follow from more refined results that will be stated below.

\begin{theorem}[Gradient regularity for nonlinear homogeneous kinetic equations] \label{thm.holmodel} 
	Let $f$ be a weak solution of
	\begin{equation}\label{eq.mainhom}
		\partial_t f+v\cdot \nabla_x f-\divergence_v(a(t,x,v,\nabla_v f))=0 \quad\text{in }W\times V\subset \bbR^{n+1}\times \bbR^n,
	\end{equation}
	where $a$ satisfies Assumption \ref{assump}. 
	Then there exists some $\alpha=\alpha(n,\Lambda)\in(0,1)$ such that if $a$ is H\"older continuous with exponent $\beta \in (0,\alpha)$ in $W \times V$ in the sense of Definition \ref{def:DH}, then $\nabla_v f\in C^{\beta}_{\textnormal{kin}}(W \times V)$. 
\end{theorem}

For the precise definition of weak solutions and of the kinetic H\"older spaces $C^{\beta}_{\textnormal{kin}}(W \times V)$ that we use in Theorem \ref{thm.holmodel} and our other results, we refer to Section \ref{setup} below.
\begin{remark}[Autonomous case] \normalfont
	Let us highlight that Theorem \ref{thm.holmodel} is new already in the autonomous case when $a$ does not depend on the coefficient variables $t,x,v$. In this case, Theorem \ref{thm.holmodel} partially answers an open question raised in \cite[Section 6, Problem 3]{GarNys23}. \end{remark} 

Next, we consider the linear case when $a(t,x,v,\xi)=A(t,x,v)\xi$ for some coefficient matrix $A:\mathbb{R} \times \bbR^{n}\times \bbR^n \to \mathbb{R}^{n \times n}$ satisfying
\begin{equation} \label{assumplin}
	|A(t,x,v)| \leq \Lambda, \quad A(t,x,v) \zeta \cdot \zeta \geq \Lambda^{-1} |\zeta|^2 \quad \forall t \in \mathbb{R}, \, x,v,\zeta \in \mathbb{R}^n
\end{equation}
for some $\Lambda \geq 1$, which corresponds to the assumptions \eqref{ass.coef} in the nonlinear case.
In this linear setting, we are able to improve the H\"older exponent from Theorem \ref{thm.holmodel}.

\begin{theorem}[Gradient regularity for linear homogeneous kinetic equations] \label{thm.holmodellin} 
	Assume that $A:\mathbb{R} \times \bbR^n \times \bbR^n \to \mathbb{R}^{n \times n}$ satisfies \eqref{assumplin} and let $f$ be a weak solution to \begin{equation}\label{eq.mainhomlin}
		\partial_t f+v\cdot \nabla_x f-\divergence_v(A(t,x,v) \nabla_v f)=0 \quad\text{in }W\times V\subset \bbR^{n+1}\times \bbR^n.
	\end{equation} If $A$ is H\"older continuous with exponent $\beta \in (0,1)$ in $W \times V$ in the sense of Definition \ref{def:DH}, then $\nabla_v f\in C^{\beta}_{\textnormal{kin}}(W\times V)$. 
\end{theorem}

We note that in \cite{Loh23}, Loher obtained corresponding $C^{k,\beta}$ regularity results for $k \geq 2$. More precisely, in \cite{Loh23} it was in particular proved that in the setting of Theorem \ref{thm.holmodellin}, for any integer $k \geq 2$ and any $\beta \in (0,1)$, $u$ satisfies local $C^{k,\beta}$ estimates whenever the coefficient matrix $A$ satisfies a $C^{k-1,\beta}$ assumption.

Since in Theorem \ref{thm.holmodellin} we obtain a similar result also for $k=1$ at least with respect to the velocity variable, Theorem \ref{thm.holmodellin} therefore complements the higher-order Schauder theory for linear kinetic Fokker-Planck equations in divergence form developed in \cite{Loh23}. 

\subsection{Setting} \label{setup}

Before being able to state our further main results involving general forcing terms, we need to introduce our setting more rigorously. We begin by recalling several function spaces. Indeed, we consider the kinetic Sobolev space
\begin{equation*}
	H^1_{\mathrm{kin}}(W\times V)=\{f\in L^2(W;H^1(V))\,:\,\partial_t f+v\cdot\nabla_x f\in L^2(W;H^{-1}(V))\}
\end{equation*}
and also the subspace $\mathcal{W}(W\times V)$ that is defined as the closure of $C^\infty(\overline{W\times V})$ in the norm of $H^1_{\mathrm{kin}}(W \times V)$,
which are introduced in \cite{AlbArmMouNov24,LitNys21}.
For any $C^{1,1}$-domain $W\subset \bbR^{n+1}$ and any open set $V \subset \mathbb{R}^n$, we denote the Kolmogorov boundary of $W\times V$ by
\begin{equation}\label{defn.kbdry}
	\partial_{\mathcal{K}}(W\times V)\coloneqq ({W}\times \partial V)\cup \{(t,x,v)\in \partial W\times \overline{V}\,:\,(1,v)\cdot N_{t,x}<0\},
\end{equation}
where $N_{t,x}$ is the outer normal unit vector to $W$ at $(t,x)$.
For $r>0$ and $z_0\coloneqq (t_0,x_0,v_0) \in \mathbb{R} \times \mathbb{R}^n \times \mathbb{R}^n$, we consider the \emph{kinetic cylinder} $Q_r(z_0)$ defined by
\begin{equation*}
	Q_r(z_0)\coloneqq\{(t,x,v)\in\bbR\times\bbR^n\times\bbR^n\,:\,t\in I_r(t_0),\,v\in B_r(v_0),\, |x-x_0-(t-t_0)v_0|<r^3\},
\end{equation*}
where
$$ I_r(t_0):=(t_0-r^2,t_0].$$
{If $z_0=0$, we write $Q_r=Q_r(0)$ for simplicity.}
Throughout the paper, we work with the following weak formulation of \eqref{eq.main}.
\begin{definition}[Weak solutions]
	Let $\mu\in L^2(W\times V)$ and $G \in L^2(W \times V,\mathbb{R}^n)$. We say that $f\in H^1_{\mathrm{kin}}(W\times V)$ is a weak solution to \eqref{eq.main}, if for any
	$\oldphi\in L^2(W;H_0^1(V))$,
	\begin{align*}
		&\int_{W}\left<(\partial_t+v\cdot\nabla_x )f,\oldphi\right>\,dx\,dt+\int_{W}\int_{V}a(t,x,v,\nabla_v f)\cdot \nabla_v \oldphi\,dz\\
		&\quad=\int_{W}\int_{V}\mu\oldphi\,dz + \int_{W}\int_{V}G \cdot \nabla_v \oldphi\,dz,
	\end{align*}
	where $\skp{\cdot}{\cdot}\coloneqq\skp{\cdot}{\cdot}_{H^{-1},H^{1}}$.
\end{definition}
\begin{remark}\label{rmk.weak} \normalfont
	We note that for any weak solution $f\in H^1_{\mathrm{kin}}(W\times V)$ to \eqref{eq.main} with $\mu\in L^2(W\times V)$, we actually have that $f\in L^\infty(I;L^2(U_x\times U_v))$ for any $I\times U_x\times U_v\Subset W\times V$, where $I \subset \mathbb{R}$ and $U_x,U_v\subset \bbR^n$, which we prove in Appendix \ref{appen}.
\end{remark}

\begin{definition}[Kinetic H\"older spaces]
	Let $\beta \in (0,1)$. For any kinetic cylinder $Q \subset \bbR\times\bbR^n\times\bbR^n$ and for any function $f:Q \to \mathbb{R}$, we define
	\begin{equation}\label{defn.hol}
		[f]_{C_\textnormal{kin}^\beta(Q)}\coloneqq\sup\limits_{(t,x,v),(s,y,w)\in Q}\frac{|f(t,x,v)-f(s,y,w)|}{\left(|t-s|^{\frac12}+|y-x-(s-t)v|^{\frac13}+|v-w|\right)^\beta}.
	\end{equation}
	Given open sets $W \subset \bbR\times\bbR^n$ and $V \subset \mathbb{R}^n$, the kinetic H\"older space $C_\textnormal{kin}^\beta(W \times V)$ is then defined as the set of all functions $f \in L^\infty_{\loc}(W \times V)$ with $[f]_{C_\textnormal{kin}^\beta(Q)}<\infty$ for any kinetic cylinder $Q \Subset W \times V$.
\end{definition}



Next, we define various notions of continuity of the nonlinearity $a(t,x,v,\xi)$.
\begin{definition}[Dini and H\"older coefficients] \label{def:DH}
	Given a kinetic cylinder $Q \subset \bbR\times\bbR^n\times\bbR^n$, let ${\pmb{\omega}}:\bbR^+\to\bbR^+$ be a non-decreasing function with ${\pmb{\omega}}(0)=0$ such that
	\begin{equation}\label{defn.osc}
		\sup_{\substack{(t,y,w),(t,x,v)\in Q}}|a(t,y,w,\xi)-a(t,x,v,\xi)|\leq {\pmb{\omega}}\left(\max\left\{|y-x|^{\frac13},|w-v|\right\}\right)|\xi|.
	\end{equation}
	\begin{itemize}
		\item We say that $a$ is Dini continuous in $Q$, if  
		\begin{equation}\label{cond.dini}
			\int_{0}^{1}\frac{{\pmb{\omega}}(\rho)}\rho\,d\rho<\infty.
		\end{equation} 
		\item We say that $a$ is H\"older continuous with exponent $\beta \in (0,1)$ in $Q$, if ${\pmb{\omega}}(\rho)\leq c\rho^\beta$ for some $c>0$.
		\item Given open sets $W \subset \bbR^{n+1}$ and $V \subset \mathbb{R}^n$, we say that $a$ is Dini continuous in $W \times V$, if $a$ is Dini continuous in every kinetic cylinder $Q \Subset W \times V$.
		\item We say that $a$ is H\"older continuous with exponent $\beta \in (0,1)$ in $W \times V$, if $a$ is H\"older continuous with exponent $\beta$ in every kinetic cylinder $Q \Subset W \times V$.
		\item Given $A:\mathbb{R} \times \bbR^{n}\times \bbR^n \to \mathbb{R}^{n \times n}$, we say that $A$ satisfies any of the above conditions, if that condition is satisfied by $a(t,x,v,\xi)=A(t,x,v)\xi$.
	\end{itemize}
\end{definition}
In order to formulate our main results involving general nondivergence-type data $\mu$, we next define a truncated kinetic Riesz-type potential of $\mu$ by
\begin{equation}\label{defn.riesz}
	I_{\alpha}^{|\mu|}(z_0,R)\coloneqq \int_{0}^{R}\frac{|\mu|(Q_r(z_0))}{r^{4n+2-\alpha}}\frac{\,dr}{r},
\end{equation}
where $\alpha\in (0,4n+2)$ and 
$
	|\mu|(Q_r(z_0))\coloneqq \norm{\mu}_{L^1(Q_r(z_0))}.
$

\subsection{Fine gradient regularity}
We prove the following pointwise estimates of the gradient of solutions via potentials, extending the known gradient potential estimates for nonlinear parabolic equations due to Duzaar and Mingione (see \cite{DuzMin11a}) to the kinetic setting.
\begin{theorem}[Kinetic gradient potential estimates]\label{thm.dini}
	Let $f\in H^1_{\mathrm{kin}}(W\times V)$ be a weak solution to \eqref{eq.main} for $G \equiv 0$, where $\mu \in L^2(W \times V)$ and $a$ satisfies Assumption \ref{assump}. If $a$ is Dini continuous in $W \times V$, then for any $R>0$ and almost every $z_0=(t_0,x_0,v_0) \in W \times V$ with $Q_{2R}(z_0)\subset W\times V$, we have
	\begin{equation*}
		|\nabla_v f(z_0)|\leq c\left(\dashint_{Q_R(z_0)}|\nabla_v f|\,dz+I_1^{|\mu|}(z_0,R)\right),
	\end{equation*}
	where $c=c(n,\Lambda,{\pmb{\omega}})$.
\end{theorem}

\begin{remark}[Consistency with parabolic case]\label{rmk.dini} \normalfont
	Consider the spatially homogeneous case when $u=u(t,v)$ is a weak solution to the parabolic equation
	\begin{equation*}
		\partial_t u-\divergence_v(a(t,v,\nabla_v u))=\mu,
	\end{equation*}
	where $a(t,v,\xi)$ is Dini continuous and $\mu=\mu(t,v)$.
	Then $u$ is also a weak solution to 
	\begin{equation*}
		\partial_t u+v\cdot \nabla_x u-\divergence_v(a(t,v,\nabla_v u))=\mu.
	\end{equation*}
	In this case we recover from Theorem \ref{thm.dini} the parabolic gradient potential estimates from \cite[Theorem 1.3]{DuzMin11a}, since
	\begin{equation*}
		\int_{0}^{R}\frac{|\mu|(Q_r(z_0))}{r^{4n+1}}\frac{\,dr}{r}\lesssim\int_{0}^{R}\frac{|\mu|(I_r(t_0)\times B_r(v_0))}{r^{n+1}}\frac{\,dr}{r}.
	\end{equation*}
\end{remark}

In view of a slight variation of Theorem \ref{thm.dini} and the mapping properties of the kinetic Riesz potentials, we have the following criteria for gradient continuity.
\begin{corollary}[Gradient continuity via potentials]\label{cor.dini}
	Let $f\in H^1_{\mathrm{kin}}(W\times V)$ be a weak solution to \eqref{eq.main} with $G \equiv 0$, where $\mu \in L^2(W \times V)$, while $a$ satisfies Assumption \ref{assump} and is Dini continuous in $Q_{2R}(z_0)\subset W\times V$. If 
	\begin{equation}\label{ass.cor.dini}
		\lim_{\rho\to0}\sup_{z_1\in Q_{R}(z_0)}I^{|\mu|}_1(z_1,\rho)=0,
	\end{equation} then $\nabla_v f$ is continuous in $Q_{R/2}(z_0)$.
	
	In particular, if $\mu\in L^{4n+2,1}(W \times V)$, then $\nabla_v f$ is continuous in $W \times V$.
\end{corollary}

For the precise definition of the Lorentz space $L^{4n+2,1}(W \times V)$, we refer to Definition \ref{def.Lorentz} below. We note that since the Dini assumption on $a$ and the Lorentz assumption on $\mu$ in Corollary \ref{cor.dini} go beyond the standard H\"older and $L^p$ scales, Corollary \ref{cor.dini} can be thought of as a regularity result of borderline flavour. Indeed, being able to detect such fine scales is one of the key strengths of the nonlinear potential-theoretic methods we use in comparison to more traditional approaches.

We also have the following criterion for VMO regularity of the gradient, which yields slightly weaker control on the oscillations of $\nabla_v u$ than continuity under slightly weaker assumptions on the data than in Corollary \ref{cor.dini}.

\begin{corollary}[VMO criterion]\label{cor.vmo}
	Let $f\in H^1_{\mathrm{kin}}(W\times V)$ be a weak solution to \eqref{eq.main} with $G \equiv 0$, where $\mu \in L^2(W \times V)$, while $a$ satisfies Assumption \ref{assump} and is Dini continuous in $Q_{2R}(z_0)\subset W\times V$. If 
	\begin{equation}\label{ass.cor.vmo}
		\sup_{z_1\in Q_{R}(z_0)}I^{|\mu|}_1(z_1,\rho)<\infty\quad\text{and}\quad \lim_{\rho\to0}\sup_{z_1\in Q_{R}(z_0)}\frac{|\mu|(Q_\rho(z_1))}{\rho^{4n+1}}=0,
	\end{equation} then $\nabla_v f\in \mathrm{VMO}(Q_{R/2}(z_0))$.
\end{corollary}
Here $\mathrm{VMO}(Q_{R/2}(z_0))$ denotes the standard space of (vector-valued) functions with vanishing mean oscillation in $Q_{R/2}(z_0)$, see e.g. \cite{Sarason,MPS00}.

Our gradient potential estimates also imply the following $L^p$ estimates.
\begin{corollary}[Calder\'on-Zygmund estimates]\label{cor.cal}
	Let $f\in H^1_{\mathrm{kin}}(W\times V)$ be a weak solution to \eqref{eq.main} with $G \equiv 0$, where $a$ satisfies Assumption \ref{assump} and is Dini continuous in $Q_{2R}(z_0)\subset W\times V$. Then for any $q\in[2,4n+2)$, we have the estimate
	\begin{align}\label{ineq1.cal}
		\left(\dashint_{Q_R(z_0)}|\nabla_v f|^{\frac{q(4n+2)}{4n+2-q}}\,dz\right)^{\frac{4n+2-q}{q(4n+2)}}\leq c\left(\dashint_{Q_{2R}(z_0)}|\nabla_v f|\,dz+R \left(\dashint_{Q_{2R}(z_0)}|\mu|^q\,dz\right)^{\frac1q}\right)
	\end{align}
	for some constant $c=c(n,\Lambda,q,{\pmb{\omega}})$. 
\end{corollary}

For $\beta \in [0,1]$ and $p \in [1,\infty)$, we consider various fractional maximal functions, namely
\begin{equation}\label{defn.smaxi}
	M^{\#,p}_{\beta,R}(F)(z_0)\coloneqq \sup_{0<r<R} r^{-\beta} \left (\dashint_{Q_r(z_0)}|F-(F)_{Q_r(z_0)}|^p\,dz \right )^{\frac1p},
\end{equation}
\begin{equation}\label{defn.maxi}
	M_{\beta,R}(g)(z_0)\coloneqq \sup_{0<r<R}r^{\beta} \dashint_{Q_r(z_0)}{|g|}\,dz,
\end{equation}
for all measurable function $F:Q_R(z_0) \to \mathbb{R}^n$, $g:Q_R(z_0) \to \mathbb{R}$. We also write $M^{\#}_{\beta,R}(F):=M^{\#,1}_{\beta,R}(F)$, $M_R(g):=M_{0,R}(g)$. Armed with these notions, we are also able to obtain stronger control of the oscillations of the gradient by assuming H\"older continuity of $a(t,\cdot,\cdot,\xi)$.
\begin{theorem}[Pointwise maximal function estimates - nondivergence data] \label{thm.holn}
	Let $f\in H^1_{\mathrm{kin}}(W\times V)$ be a weak solution to \eqref{eq.main} with $G \equiv 0$, let $\mu \in L^2(W \times V)$ and assume that $a$ satisfies Assumption \ref{assump}. Then there exists some $\alpha=\alpha(n,\Lambda) \in (0,1)$ such that if $a$ is H\"older continuous with exponent $\beta \in (0,\alpha)$ in $W\times V$,
	then for any $R>0$ and almost every $z_0=(t_0,x_0,v_0) \in W \times V$ with $Q_{2R}(z_0)\subset W\times V$, we have
	\begin{equation}\label{ineq0.holn}
		M^{\#}_{\beta,R}(\nabla_vf)(z_0) \leq c\left(\dashint_{Q_R(z_0)}|\nabla_v f|\,dz+I^{|\mu|}_1(z_0,R)+M_{1-\beta,R}(\mu)(z_0)\right),
	\end{equation}
	where $c=c(n,\Lambda,\beta,{\pmb{\omega}})$.
	
	In particular, we have the implication
	\begin{equation}\label{ineq1.holn}
		\mu\in L^{\frac{4n+2}{1-\beta},\infty}(W \times V)\Longrightarrow \nabla_v f\in C_{\textnormal{kin}}^{\beta}(W \times V).
	\end{equation}
\end{theorem}
For the precise definition of the Marcinkiewicz space $L^{\frac{4n+2}{1-\beta},\infty}(W \times V)$, we refer to Definition \ref{def.Lorentz} below. Moreover, in view of \eqref{ineq1.holn}, Theorem \ref{thm.holn} with $\mu \equiv 0$ clearly implies Theorem \ref{thm.holmodel} above.
Next, as indicated in Section \ref{hom}, in the case of linear kinetic equations we are able to improve the range of the exponent $\beta$ in Theorem \ref{thm.holn}.
\begin{theorem}[Linear case - nondivergence data]\label{cor.lin}
	Let $f\in H^1_{\mathrm{kin}}(W\times V)$ be a weak solution to \eqref{eq.main} with $G \equiv 0$, $\mu \in L^2(W \times V)$ and $a(t,x,v,\xi)=A(t,x,v)\xi$ for some $A:\bbR \times \bbR^{n}\times \bbR^n \to \mathbb{R}^{n \times n}$ that satisfies \eqref{assumplin}. If $a$ is H\"older-continuous with exponent $\beta\in(0,1)$ in $W\times V$, then $f$ satisfies the estimate \eqref{ineq0.holn} and the implication \eqref{ineq1.holn} from Theorem \ref{thm.holn} with respect to $\beta$.
\end{theorem}
We note that Theorem \ref{cor.lin} with $\mu \equiv 0$ clearly implies Theorem \ref{thm.holmodellin} above.

Let us denote by $\mathrm{BMO}(W \times V)$ a space of functions with bounded mean oscillation of vector-valued functions in the following sense. Indeed, we say that a measurable function $F:W \times V \to \mathbb{R}^n$ belongs to $\mathrm{BMO}(W \times V)$, if
\begin{align}\label{defn.bmo}
	[F]_{\mathrm{BMO}(W \times V)}\coloneqq \sup_{Q_{R}(z_0)\Subset W \times V}\left(\dashint_{Q_R(z_0)}|F-(F)_{Q_R(z_0)}|^2\,dz\right)^{\frac12}<\infty.
\end{align}

We then also have a corresponding result to Theorem \ref{thm.holn} when the right-hand side is given by divergence-type data.

\begin{theorem}[Pointwise maximal function estimates - divergence data]\label{thm.hold}
	Let $f\in H^1_{\mathrm{kin}}(W\times V)$ be a weak solution to \eqref{eq.main} with $\mu \equiv 0$, let $G \in L^2(W \times V,\mathbb{R}^n)$ and assume that $a$ satisfies Assumption \ref{assump}. Then there exists some $\alpha=\alpha(n,\Lambda) \in (0,1)$ such that if $a$ is H\"older continuous with exponent $\beta \in (0,\alpha)$ in $W\times V$,
	then for any $R>0$ and almost every $z_0=(t_0,x_0,v_0) \in W \times V$ with $Q_{2R}(z_0)\subset W\times V$, we have
	\begin{equation}\label{ineq1.hold}
		M^{\#}_{\beta,R}(\nabla_vf)(z_0) \leq c\left(M_{R}(|\nabla_vf|)(z_0)+M^{\#,2}_{{\beta},R}(G)(z_0)\right),
	\end{equation}
	where $c=c(n,\Lambda,\widetilde{\beta},{\pmb{\omega}})$.
	Moreover, we have the following two implications
	\begin{align}\label{ineq2.hold}
		G\in \textnormal{BMO}(W \times V)\Longrightarrow \nabla_vf\in \textnormal{BMO}(W \times V)
	\end{align}
	and
	\begin{equation}\label{ineq3.hold}
		G\in C^{{\beta}}_{\textnormal{kin}}(W \times V)\Longrightarrow \nabla_v f\in C^{{\beta}}_{\textnormal{kin}}(W \times V).
	\end{equation}
\end{theorem}

Again, if we restrict ourselves to the linear case, then we can improve the exponent $\beta$. Moreover, in the case of divergence-type data we are also able to obtain a significant amount of H\"older regularity with respect to the spatial variable $x$.

\begin{theorem}[Linear case - divergence data]\label{cor.lind}
	Let $f\in H^1_{\mathrm{kin}}(W\times V)$ be a weak solution to \eqref{eq.main} with $\mu \equiv 0$, let $G \in L^2(W \times V,\mathbb{R}^n)$ and $a(t,x,v,\xi)=A(t,x,v)\xi$ for some $A:\bbR \times \bbR^{n}\times \bbR^n \to \mathbb{R}^{n \times n}$ that satisfies \eqref{assumplin}. If $a$ is H\"older-continuous with exponent $\beta\in(0,1)$ in $W\times V$, then $f$ satisfies the estimate \eqref{ineq1.hold} and the implications \eqref{ineq2.hold}-\eqref{ineq3.hold} from Theorem \ref{thm.holn} with respect $\beta$.
	
	In addition, if $G\in C^{{\beta}}_{\textnormal{kin}}(W \times V)$ and $Q_{2R}(z_0)\Subset W\times V$, then for any $\epsilon>0$ we have
	\begin{align}\label{ineq.lind}
		\sup_{(t,v)\in I_R(t_0)\times B_R(v_0)}\sup_{x_1,x_2\in B_{R^3}(x_0+(t-t_0)v_0)}\frac{|f(t,x_1,v)-f(t,x_2,v)|}{|x_1-x_2|^{\frac{1+\beta-\epsilon}{3}}}<\infty.
	\end{align}
\end{theorem}

\subsection{Related previous results}
In recent years, there has been a lot of interest in studying the regularity of solutions to kinetic Fokker-Planck equations.
In the nonlinear case, H\"older regularity for the solution itself was studied in \cite{GarNys23}. On the other hand, to the best of our knowledge no gradient regularity results have been obtained for nonlinear kinetic Fokker-Planck equations of the type \eqref{eq.main} in the previous literature. 

In the case of linear kinetic Fokker-Planck equations in divergence form, local boundedness and H\"older regularity in the spirit of De Giorgi-Nash-Moser was studied for instance in \cite{CinPasPol08,WanZha09,WanZha11,GolImbMouVas19,Zhu21,GueMou22,AncReb22,Sil22,CMS22,DHirsch22,Zhu24,Hou24}. Concerning higher regularity, Schauder-type estimates of order greater or equal to two have been obtained in \cite{Loh23}. Concerning gradient estimates, in \cite{ManPol98,DonYas24} first-order $L^p$ estimates were proved for linear kinetic equations in divergence form with divergence-type data and coefficients of VMO-type. On the other hand, we are not aware of any previous pointwise gradient estimates similar to the ones we obtain in the linear divergence form case.

In contrast, there is a large existing literature concerned with the higher regularity for linear kinetic equations in nondivergence form. Schauder-type estimates were for example obtained in 
\cite{MM97,BraBra07,HenSne20,ImbMou21,PolRebStr22,DonYas22s,DonGouYas22,Loh23,HenWan24}, while Sobolev regularity was for instance studied in \cite{BraCerMan96,NieZac22,DonYas22,DonYas24}. 

Moreover, various results concerning the existence of various notions of solutions to kinetic equations were e.g.\ proved in \cite{LitNys21,GarNys23,AlbArmMouNov24,PasCyrLuk24,AveHouNys24,AveHou24}.

Furthermore, while the present paper seems to be the first one that studies pointwise estimates in terms of Riesz-type potentials and fractional maximal functions for nonlinear kinetic equations, corresponding pointwise estimates are very well-studied for elliptic and parabolic equations. Indeed, following the pioneering zero-order potential estimates for solutions to nonlinear elliptic equations due to Kilpeläinen and Maly (see \cite{KM94}), Mingione in \cite{Min11} proved that pointwise estimates in terms of Riesz potentials remain valid also for the gradient of solutions to nonlinear elliptic PDE. Soon after, Duzaar and Mingione in \cite{DuzMin11a} proved that similar gradient potential estimates also hold for nonlinear parabolic equations. Moreover, pointwise estimates that provide control on the oscillations of the gradient were proved in the elliptic setting by Kuusi and Mingione in \cite{KuuMin12}. In addition, in the case of divergence-type data, pointwise estimates in terms of sharp maximal functions were established by Breit, Cianchi, Diening, Kuusi and Schwarzacher in \cite{BCDKS}. Most of these results were then generalized to more general elliptic and parabolic problems, see for instance \cite{TWAJM,DM1,KuuMin13,KuuMin14,CMJEMS,Baroni,KuuMin18,BYP,BDGP,DFJMPA,NNARMA,CKW23,DZ24} for a non-exhaustive list of further contributions in this direction. 

\subsection{Technical approach and novelties}

Let us summarize our approach in a heuristic manner with a particular focus on the main difficulties that we encounter and the technical novelties we incorporate to surmount them. Indeed, the main difficulties we face in contrast to the previous literature arise primarily from the following two sources:
\begin{itemize}
	\item The nonlinearity of the equation: Rules out many tools commonly used in the linear kinetic setting such as explicit fundamental solutions and Fourier methods.
	\item The kinetic nature of the equation: The additional transport term in contrast to the parabolic case leads to a lack of ellipticity in the spatial variable $x$.
\end{itemize} 

Both of these difficulties are already present in the homogeneous constant coefficient case when $\mu \equiv 0$, $G \equiv 0$ and $a$ does not depend on $x,v$, which is the starting point of our proof. We surmount these difficulties by applying an iteration of De Giorgi-Nash-Moser-type H\"older estimates on \emph{fractional difference quotients} inspired by recent developments in the realm of nonlocal equations (see e.g.\ \cite{BLS18,DKLN1,DieKimLeeNow24p}). This allows us to first prove H\"older continuity of the spatial gradient $\nabla_x f$, which leads to the transport term $v \cdot \nabla_x f$ being essentially of lower order. This observation can then be exploited to prove also H\"older regularity for the velocity gradient $\nabla_v f$.

The next goal is then to use these strong gradient decay estimates to prove our fine pointwise gradient estimates that involve a dependence of $a$ also on $x,v$ and general forcing terms $\mu$ and $\divergence_v G$. The general strategy is similar to the parabolic case with nondivergence data treated in \cite{DuzMin11a} and the elliptic case with divergence data addressed in \cite{BCDKS}. Indeed, we compare our solution of \eqref{eq.main} with a corresponding solution of a homogeneous kinetic equation of the type that we treated already in the previous step of the proof. The goal is then to estimate the error between the velocity gradients of the two solutions in a sufficiently sharp way that allows us to transform the strong decay we obtained in the homogeneous case into sharp pointwise control of the velocity gradient of solutions of \eqref{eq.main} in terms of potentials and fractional maximal functions of the data.

However, obtaining suitable comparison estimates at the gradient level is considerably more involved in our kinetic setting due to the additional presence of the spatial variable $x$. Indeed, in contrast to the parabolic setting from \cite{DuzMin11a} we first prove a higher differentiability result of the solution with respect to the spatial variable $x$ in terms of the $L^1$-norm of $\mu$. We accomplish this by invoking a \emph{nonlinear atomic decomposition} that is often applied to differentiate equations with nondifferentiable ingredients in the elliptic and parabolic setting, see for instance \cite{KM1,KM2,Min07,AKM18,DFM1,DFM2,DKLN1,DieKimLeeNow24p}. Roughly speaking, the idea is to consider difference quotients of the solution with increment $|h|$ in the $x$ direction and apply zero-order comparison estimates on kinetic cylinders whose size depends on $|h|$ itself. In view of interpolation arguments, this additional differentiability with respect to the $x$-variable finally leads to suitable comparison estimates in terms of the velocity gradients that seem to be new already for linear kinetic equations and might be of independent interest. Together with the strong estimates we obtained in the homogeneous case, following the parabolic approach from \cite{DuzMin11a} then leads to suitable excess decay estimates involving general data, which yields our pointwise gradient estimates and their consequences in a rather standard way.

\subsection{Outline}
The paper is structured as follows. In Section \ref{sec2}, we gather some basic notation as well as some fundamental properties of the various kinetic notions we use. We then conclude Section \ref{sec2} by discussing various estimates for solutions to kinetic equations that were essentially known prior to the present work.
In Section \ref{sec3}, we then turn to prove gradient H\"older estimates for nonlinear kinetic equations without forcing terms and with constant coefficients.
In Section \ref{sec4}, we then lay the foundation for obtaining fine estimates for solutions to general nonlinear kinetic equations of the type \eqref{eq.main} by establishing various comparison estimates.
Finally, in Section \ref{sec5} we then combine the results obtained in the previous sections in order to establish our main results concerning pointwise gradient estimates of solutions to \eqref{eq.main} and their applications to gradient regularity.

\section{Preliminaries} \label{sec2}
First of all, throughout this paper by $c$ we denote general positive
constants which could vary line by line. In addition, we use a parentheses to highlight
relevant dependencies on parameters, i.e., $c = c(n,\Lambda)$ indicates that the constant $c$ depends only on $n$ and $\Lambda$.

For any function $g\in L^1(U)$ and any open set $U\subset\setR^m$ $(m \geq 1)$ with positive and finite $m$-dimensional Lebesgue measure $|U|$, we define the integral average of $g$ in $U$ as
	\begin{align*}
		\dashint_{U}g\,dx:=(g)_{U}:=\dfrac{1}{\abs{U}}\int_{U}g\,dx.
	\end{align*}

Throughout the paper, we write the variables in the form
	$z\coloneqq (t,x,v) \in \mathbb{R} \times \mathbb{R}^n \times \mathbb{R}^n$, where $t\in\setR$ represents a time variable, $x\in\setR^n$ represents a spatial variable, and $v\in\setR^n$ represents a velocity variable. We use subscript of the variables $z,t,x$, and $v$ do indicate dependences of various objects if necessary.

We usually denote by $W$ a subset of $\bbR^{n+1}$ and by $V$ a subset of $\bbR^n$, respectively.  In addition, for $v_0,x_0 \in \mathbb{R}^n$ we write 
\begin{align}\label{defn.xball}
    B_r(v_0)\coloneqq\{v\in\bbR^n\,:|v-v_0|<r\}, \quad B^x_r(x_0)\coloneqq\{x\in\bbR^n\,:|x-x_0|<r^3\}.
\end{align}

We now introduce several lemmas which will be used in the remaining sections. First of all, we observe a straightforward scaling-invariance property of \eqref{eq.main}.
\begin{lemma}[Scaling invariance]\label{lem.scale}
    Let $f\in H^1_{\mathrm{kin}}(Q_r(z_0))$ be a weak solution to 
    \begin{equation*}
        \partial_t f+v\cdot\nabla_x f-\divergence_v(a(t,x,v,\nabla_v f))=\mu - \divergence_v G \quad\text{in }Q_r(z_0).
    \end{equation*}
    Then {for $r,M>0$, and $z_0\in\setR^{2n+1}$,}
    \begin{equation*}
        f_{r,z_0}(t,x,v)=f(t_0+r^2t,x_0+r^3x+r^2v_0t,v_0+rv)/(rM)
    \end{equation*}
    is a weak solution to 
    \begin{equation*}
        \partial_t f_{r,z_0}+v\cdot \nabla_x f_{r,z_0}-\divergence_v(a_{r,z_0}(t,x,v,\nabla_v f_{r,z_0}))=\mu_{r,z_0} - \divergence_v G_{r,z_0} \quad\text{in }Q_1,
    \end{equation*}
    where 
    \begin{align*}
        a_{r,z_0}(t,x,v,\xi)=a(t_0+r^2t,x_0+r^3x+r^2v_0t,v_0+rv,M\xi)/M,
    \end{align*}
    \begin{align*}
        \mu_{r,z_0}(t,x,v)=r\mu(t_0+r^2t,x_0+r^3x+r^2v_0t,v_0+rv)/M
    \end{align*}
	and
	\begin{align*}
		G_{r,z_0}(t,x,v)=G(t_0+r^2t,x_0+r^3x+r^2v_0t,v_0+rv)/M.
	\end{align*}
    In addition, $a_{r,z_0}$ satisfies \eqref{ass.coef} and 
    \begin{equation*}
       \sup_{(t,y,w),(t,x,v)\in Q_1}| a_{r,z_0}(t,y,w,\xi)-a_{r,z_0}(t,x,v,\xi)|\leq {\pmb{\omega}}_{r}(\max\{|x-y|^{\frac13},|v-w|\})|\xi|,
    \end{equation*}
    where 
    $
        {\pmb{\omega}}_{r}(\rho)\coloneqq{\pmb{\omega}}(r\rho).
    $
\end{lemma}
\subsection{Covering-type lemmas}
We prove a few elementary properties of kinetic cylinders.
\begin{lemma}\label{lem.cov1}
    Let $Q_r(t_1,x_1,v_1)\cap Q_r(t_2,x_2,v_2)\neq \emptyset$ with $t_1\geq t_2$. Then we have 
    \begin{equation*}
        Q_{r}(t_2,x_2,v_2)\subset Q_{4r}(t_1,x_1,v_1).
    \end{equation*}
\end{lemma}
\begin{proof}
    Let $z_0\in Q_r(t_1,x_1,v_1)\cap Q_r(t_2,x_2,v_2)$. Then we have 
    \begin{equation}\label{ineq1.cov1}
       t_0\in (t_1-r^{2},t_1],\quad v_0\in B_r(v_1),\quad |x_0-x_1-v_1(t_0-t_1)|<r^3
    \end{equation}
    and
    \begin{equation}\label{ineq2.cov1}
        t_0\in (t_2-r^{2},t_2],\quad v_0\in B_r(v_2),\quad |x_0-x_2-v_2(t_0-t_2)|<r^3.
    \end{equation}
    We now assume $z\in Q_r(t_2,x_2,v_2)$, which implies
    \begin{equation}\label{ineq3.cov1}
         t\in (t_2-r^{2},t_2],\quad v\in B_r(v_2),\quad|x-x_2-v_2(t-t_2)|<r^3.
    \end{equation}
    Thus 
    \begin{equation*}
        t\in (t_1-(4r)^{2},t_1]\quad\text{and}\quad v\in B_{4r}(v_1).
    \end{equation*}
    Next, we observe that
    \begin{align*}
        |x-x_1-v_1(t-t_1)|&\leq |x-x_2-v_2(t-t_2)|+|x_2-x_0+v_2(t_0-t_2)|\\
        &\quad+|x_0-x_1-v_1(t_0-t_1)|+|(v_2-v_1)(t-t_0)|\leq 7r^3.
    \end{align*}
    Therefore, we have $z\in Q_{4r}(t_1,x_1,v_1)$, which completes the proof.
\end{proof}

\begin{lemma}\label{lem.cov2}
    Let $1/2\leq \rho<R\leq1$ and $r\in\left(0,\frac{R-\rho}{1000\sqrt{n}}\right]$. Then there are a finite index set  $\mathcal{K}$ and a sequence $\{z_k\}_{k\in\mathcal{K}}$ such that 
    \begin{equation}\label{est3.cov}
        Q_\rho\subset \bigcup_{k\in \mathcal{K}}Q_r(z_k)\subset \bigcup_{k\in \mathcal{K}}Q_{4r}(z_k)\subset Q_{R}
    \end{equation}
    and
    \begin{equation}\label{est4.cov}
        \sup_{z\in\bbR^{2n+1}}\sum_{k\in \mathcal{K}}\mbox{\Large$\chi$}_{Q_{4r}(z_k)}(z)\leq c
    \end{equation}
for some constant $c=c(n)$.
\end{lemma}
\begin{proof}
Let us define a cube in $\bbR^n$ by
\begin{equation*}
    K_r(w)=\{v\in\bbR^n\,:v_i\in(w_i-r/2,w_i+r/2]\text{ for any }i\in[1,n]\},
\end{equation*}
where we write $w=(w_i),v=(v_i)\in\bbR^n$.
    Note that there are collections of disjoint cubes $\{I_{r/\sqrt{n}}(t_i)\}_{i\in\mathsf{I}}$ and $\{ K_{r/\sqrt{n}}(v_k)\}_{k\in\mathsf{K}}$ such that 
    \begin{equation*}
        I_\rho\subset \bigcup_{i\in\mathsf{I}}I_{r/\sqrt{n}}(t_i) \subset \bigcup_{i\in\mathsf{I}}I_{4r}(t_i)\subset I_{\rho+4r}, \quad B_\rho\subset \bigcup_{k\in\mathsf{K}}K_{r/\sqrt{n}}(v_k)\subset\bigcup_{k\in\mathsf{K}}B_{4r}(v_k)\subset B_{\rho+5r},
    \end{equation*}
    where $v_k\in B_{\rho+r}$ and $t_i\in I_\rho$.
     We next note that there is a collection of disjoint cubes $\{K_{(r/\sqrt{n})^3}(x_j)\}_{j\in\mathsf{J}}$ satisfying 
    \begin{equation*}
        B_{\rho(\rho^2+2r^2)}\subset \bigcup_{j\in\mathsf{J}}K_{(r/\sqrt{n})^3}(x_j)\subset B_{\rho(\rho^2+3r^2)}.
    \end{equation*}
    For any $r>0$, we now define
    \begin{equation}\label{defn.cov2}
        {K}_{r}(t_i,x_j,v_k)\coloneqq\{(t,x,v)\,:\, t\in I_{r}(t_i),\,v\in K_{r}(v_k),\, x\in K_{r^3}(x_j+v_k(t-t_i))\}.   
        \end{equation}
    Then we observe that $\{{K}_{r/\sqrt{{n}}}(t_i,x_j,v_k)\}_{i,j,k}$ is a collection of disjoint sets.

    We now prove the following:
    \begin{equation}\label{inc.cov2}
        Q_\rho\subset\bigcup_{i,j,k}Q_{r}(t_i,x_j,v_k)\subset\bigcup_{i,j,k}Q_{4r}(t_i,x_j,v_k)\subset Q_{R}.
    \end{equation}
    For any $z\in Q_\rho$, there are points $t_i$ and $v_k$ such that $t\in I_{r/\sqrt{n}}(t_i)$ and $v\in K_{r/\sqrt{n}}(v_k)$.
    This implies that $x-v_k(t-t_i)\in B_{\rho(\rho^2+2r^2)}$ and the equivalence
    \begin{equation*}
        x-v_k(t-t_i)\in K_{r/\sqrt{n}}(x_j)\Longleftrightarrow x\in K_{r/\sqrt{n}}(x_j+v_k(t-t_i))
    \end{equation*}
    for some $j\in\mathsf{J}$. Therefore, we have $ z\in K_{r/\sqrt{{n}}}(t_i,x_j,v_k)$,
    which implies
    \begin{align*}
        Q_\rho\subset \bigcup_{i,j,k}Q_{r}(t_i,x_j,v_k).
    \end{align*}
    In addition, for any $z\in Q_{4r}(t_i,x_j,v_k)$, we have $|x-x_j-v_k(t-t_i)|\leq (4r)^3$,
    leading to
    \begin{equation*}
        |x|\leq \rho(\rho^2+3r^2)+(\rho+r)(4r)^2+(4r)^3<R^3,
    \end{equation*}
    where we used that $r\leq(R-\rho)/(1000\sqrt{n})$.
    Therefore, $z\in Q_{R}$ and we have proved \eqref{inc.cov2}.
    
    We are now going to prove \eqref{est4.cov}. We may write $\mathcal{K}=\mathsf{I}\times \mathsf{J}\times \mathsf{K}$
    and suppose
    \begin{equation*}
        \sum_{k\in \mathcal{K}}\mbox{\Large$\chi$}_{Q_{4r}(z_k)}(z_0)> (16\sqrt{{n}})^{4n+2}|S_{{n}}|^2,
    \end{equation*}
    where we denote by $|S_{{n}}|$ the area of the unit sphere in ${n}$ dimensions. For
    \begin{equation*}
        \mathcal{K}_0=\{k\in\mathcal{K}\,:\,\mbox{\Large$\chi$}_{Q_{4r}(z_k)}(z_0)=1\},
    \end{equation*}
    there exists $k_0\in\mathcal{K}_0$ such that $t_{k_0}\geq t_k$ for any $k\in \mathcal{K}_0$.
    Then by Lemma \ref{lem.cov1}, we get
    \begin{equation}\label{cov2.ineq1}
        \bigcup_{k\in\mathcal{K}_0}Q_{4r}(z_k)\subset Q_{16r}(z_{k_0}).
    \end{equation}
    Thus, from the fact that $\{K_{r\sqrt{n}}(z_k)\}$ is a disjoint set defined in \eqref{defn.cov2} and \eqref{cov2.ineq1}, we obtain
    \begin{align*}
        \left((16\sqrt{{n}})^{4n+2}|S_{{n}}|^2+1\right)|Q_{4r}|&\leq \sum_{k\in\mathcal{K}_0}|Q_{4r}(z_k)|\\
        &= (4\sqrt{{n}})^{4n+2}\sum_{k\in\mathcal{K}_0}|K_{r/\sqrt{{n}}}(z_k)||S_{{n}}|^2\\
        &\leq (4\sqrt{{n}})^{4n+2}|Q_{16r}(z_{k_0})||S_{{n}}|^2\\
        &\leq (16\sqrt{{n}})^{4n+2}|S_{{n}}|^2|Q_{4r}|,
    \end{align*}
    which gives a contradiction. Therefore, we have 
    \begin{equation*}
        \sup_{z\in\bbR^{2n+1}}\sum_{k\in \mathcal{K}}\mbox{\Large$\chi$}_{Q_{4r}(z_k)}(z)\leq (16\sqrt{{n}})^{4n+2}|S_{{n}}|^2.
    \end{equation*}
    This completes the proof.
\end{proof}

\subsection{Embeddings via fractional maximal functions and Riesz potentials}
Let us gather some useful estimates involving the various maximal functions we use throughout the paper. The following pointwise estimate is the kinetic version of \cite[Proposition 1]{KuMiG}, see also \cite{DeVoreSharpley}.
\begin{lemma}\label{lem.ptmax}
    Let $g\in L^1(Q_{2R}(z_0))$ and $\beta\in(0,1)$. Then we have 
    \begin{align*}
        \frac{|g(z_1)-g(z_2)|}{d(z_1,z_2)^\beta}\leq c\left(M^{\#}_{\beta,R/2}(g)(z_1)+M^{\#}_{\beta,R/2}(g)(z_2)\right)
    \end{align*}
    for a.e.\ $z_1=(t_1,x_1,z_2),z_2=(t_2,x_2,v_2)\in Q_{R/16}(z_0)$, where $c=c(n,\beta)$ and 
    \begin{align*}
        d(z_1,z_2)\coloneqq|t_1-t_2|^{\frac12}+|x_1-(x_2+v_2(t_1-t_2))|^{\frac13}+|v_1-v_2|.
    \end{align*}
\end{lemma}
\begin{proof}
    Let us fix $z_1,z_2\in Q_{R/16}(z_0)$ and write 
    \begin{align}\label{ineq0.ptmax}
        \rho\coloneqq \max\left\{|t_1-t_2|^{\frac12},|x_1-(x_2+v_1(t_2-t_1))|^{\frac13},|v_1-v_2|\right\}<R/8.
    \end{align}
    We may assume $t_1>t_2$. We first observe that
    \begin{equation}\label{ineq1.ptmax}
    \begin{aligned}
        \sum_{i=0}^\infty\left|(g)_{Q_{2^{-i}\rho}(z_1)}-(g)_{Q_{2^{-i-1}\rho}(z_1)}\right|&\leq c\rho^\beta\sum_{i=0}^\infty 2^{-i\beta}\dashint_{Q_{2^{-i}\rho}(z_1)}\frac{|g-(g)_{Q_{2^{-i}\rho}(z_1)}|}{(2^{-i}\rho)^\beta}\,dz\\
        &\leq c\rho^\beta M^{\#}_{\beta,R/2}(g)(z_1)
    \end{aligned}
    \end{equation}
    for some constant $c=c(n,\beta)$. This implies that the following limit exists:
    \begin{align*}
        \lim_{i\to\infty}(g)_{Q_{2^{-i}\rho}(z_1)}.
    \end{align*}
    From this, it is easy to see that also $\lim_{r \to0}(g)_{Q_{r}(z_1)}$ is well-defined. Using this and taking into account \cite[Theorem 10.3]{ImbSil20} with $s=1$, we define the pointwise value $g(z_1):=\lim_{r \to0}(g)_{Q_{r}(z_1)}$.
    We next observe 
    \begin{align}\label{ineq2.ptmax}
        g(z_1)=\left[\sum_{i=0}^{\infty}\left((g)_{Q_{2^{-i}\rho}(z_1)}-(g)_{Q_{2^{-i-1}\rho}(z_1)}\right)\right]+(g)_{Q_\rho(z_1)}.
    \end{align}
    Thus, we get 
    \begin{align*}
        |g(z_1)-g(z_2)|&\leq \sum_{i=0}^\infty\left|(g)_{Q_{2^{-i}\rho}(z_1)}-(g)_{Q_{2^{-i-1}\rho}(z_1)}\right|\\
        &\quad+\sum_{i=0}^\infty\left|(g)_{Q_{2^{-i}\rho}(z_2)}-(g)_{Q_{2^{-i-1}\rho}(z_2)}\right|\\
        &\quad+|(g)_{Q_{\rho}(z_1)}-(g)_{Q_{\rho}(z_2)}|\coloneqq \sum_{k=1}^3J_k.
    \end{align*}
    By \eqref{ineq1.ptmax}, we estimate
    \begin{align*}
        J_1+J_2\leq c\rho^\beta \left[M^{\#}_{\beta,R/2}(g)(z_1)+M^{\#}_{\beta,R/2}(g)(z_2)\right],
    \end{align*}
    where $c=c(n,\beta)$. In light of Lemma \ref{lem.cov1}, we next estimate $J_3$ as 
    \begin{align*}
        J_3\leq c\dashint_{Q_{4\rho}(z_1)}|g-(g)_{Q_{4\rho}(z_1)}|\,dz\leq c\rho^\beta M^{\#}_{\beta,R/2}(g)(z_1).
    \end{align*}
    Combining all the estimates $J_1$, $J_2$ and $J_3$ together with \eqref{ineq0.ptmax} yields the desired estimate.
\end{proof}
\begin{corollary}\label{cor.ptmax}
    Let $g\in L^1({Q_{2R}(z_0)})$ and $\beta\in(0,1)$. Then we have 
    \begin{align}\label{ineq1.ptmaxc}
        \|g\|_{L^\infty(Q_{R/16}(z_0))}\leq c\left(R^\beta\|M^{\#}_{\beta,R/2}(g)\|_{L^\infty(Q_{R/2}(z_0))}+\dashint_{Q_{2R}(z_0)}|g|\,dz\right)
    \end{align}
    and
    \begin{align}\label{ineq2.ptmaxc}
        [g]_{C_{\mathrm{kin}}^\beta(Q_{R/16}(z_0))}\leq c\|M^{\#}_{\beta,R/2}(g)\|_{L^\infty(Q_{R/2}(z_0))}
    \end{align}
    for some constant $c=c(n,\beta)$.
\end{corollary}
\begin{proof}
    Using \eqref{ineq1.ptmax} and \eqref{ineq2.ptmax}, we have \eqref{ineq1.ptmaxc}. In addition, by recalling \eqref{defn.hol} and using Lemma \ref{lem.ptmax}, we get \eqref{ineq2.ptmaxc}.
\end{proof}
Before moving on, let us recall the definition of Lorentz spaces.
\begin{definition} \label{def.Lorentz}
	Given open sets $W \subset \mathbb{R}^{n+1}$, $V \subset \mathbb{R}^n$, for any $p \in [1,\infty)$ and any $q \in (0,\infty]$ we define the Lorentz space $L^{p,q}(W \times V)$ as the space of measurable functions $\mu: W\times V \to \mathbb{R}^n$ such that the quasinorm
	$$ ||\mu||_{L^{p,q}(W \times V)}:=
	\begin{cases} \normalfont
		p^\frac{1}{q} \left (\int_0^\infty \lambda^q |\{x \in W\times V : |\mu(x)| \geq \lambda \}|^\frac{q}{p} \frac{d\lambda}{\lambda} \right )^\frac{1}{q}  & \text{ if } q<\infty\\
		\sup_{\lambda>0} \lambda \hspace{0.2mm} | \{x \in W\times V : |\mu(x)| \geq \lambda \} |^\frac{1}{p} & \text{ if } q=\infty
	\end{cases} $$
	is finite.
\end{definition}

In particular, for any $p \in [1,\infty)$ by Cavalieri's principle we have $L^{p,p}(W \times V)=L^p(W \times V)$, so that the Lorentz spaces refine the scale of $L^p$ spaces. Moreover, it is easy to see that $L^{p,q}(W \times V) \hookrightarrow L^{p,\sigma}(W \times V)$ for all $p \in [1,\infty)$ and all $q,\sigma \in (0,\infty]$ with $q<\sigma$.

We now provide an embedding result via fractional maximal functions. We follow the proof of the corresponding elliptic version given in \cite[Proposition 2.5]{DieNow23}.
\begin{lemma}\label{lem.fmax}
    Let $\beta\in(0,1)$ and $\mu\in L^{\frac{4n+2}{\beta},\infty}(Q_{2R}(z_0))$. Then there is a constant $c=c(n,\beta)$ such that 
    \begin{align*}
        \sup_{z_1\in Q_{R/2}(z_0)}M_{\beta,R/2}(\mu)(z_1)\leq c\|\mu\|_{L^{\frac{4n+2}{\beta},\infty}(Q_{2R}(z_0))}.
    \end{align*}
\end{lemma}
\begin{proof}
    We assume $\|\mu\|_{L^{\frac{4n+2}{\beta},\infty}(Q_{2R}(z_0))}>0$, otherwise the desired estimates follows trivially. Let us define $\overline{\mu}\coloneqq\mu/\|\mu\|_{L^{\frac{4n+2}{\beta},\infty}(Q_{2R}(z_0))}$, so that $\|\overline{\mu}\|_{L^{\frac{4n+2}{\beta},\infty}(Q_{2R}(z_0))}=1$. For any $z_1\in Q_{R/2}(z_0)$ and $r\in(0,R/2]$, we have 
    \begin{align*}
        r^{\beta}\dashint_{Q_r(z_1)}|\overline{\mu}|\,dz&\leq r^\beta\int_{0}^{r^{-\beta}}\,d\lambda+cr^{\beta-(4n+2)}\int^{\infty}_{r^{-\beta}}|\{z\in Q_r(z_1)\,:\,|\overline{\mu}(z)|>\lambda\}|\,d\lambda\\
        &\leq c+c\|\overline{\mu}\|^{\frac{4n+2}{\beta}}_{L^{\frac{4n+2}{\beta},\infty}(Q_{2R}(z_0))}r^{\beta-(4n+2)}\int^{\infty}_{r^{-\beta}}\lambda^{-\frac{4n+2}{\beta}}\,d\lambda\leq c
    \end{align*}
    for some constant $c=c(n,\beta)$, which in view of rescaling completes the proof.
\end{proof}

Next, we extend the mapping properties from classical Riesz potentials (see e.g.\ \cite{Cianchi}) to their kinetic counterparts defined in \eqref{defn.riesz}.
\begin{lemma}\label{lem.riesz}
    Let us assume $\mu\in L^q(Q_R(z_0))$ with $1<q<4n+2$. Then we have
    \begin{equation}\label{ineq1.riesz}
        \left\|I^{|\mu|}_{1}(z,R/4)\right\|_{ L^{\frac{q(4n+2)}{4n+2-q}}(Q_{R/4}(z_0))}\leq c\|\mu\|_{L^q(Q_R(z_0))}
    \end{equation}
    for some constant $c=c(n,q)$. In addition, if $\mu\in L^{4n+2,1}(Q_R(z_0))$, then 
    \begin{equation}\label{ineq2.riesz}
        \lim_{\rho\to0}\left\|I^{|\mu|}_{1}(z,\rho)\right\|_{L^\infty(Q_{R/4}(z_0))}=0.
    \end{equation}
\end{lemma}
\begin{proof} 
    Let us fix $\delta\in(0,\frac{R}{4})$ and $z_1\in Q_{R/4}(z_0)$.
    We first note that
    \begin{align*}
        I_1^{|\mu|}(z_1,R/4)=c\int_{0}^{\delta}\dashint_{Q_r(z_1)}|\mu|(z)\,dz{\,dr}+c\int_{\delta}^{R/4}\dashint_{Q_r(z_1)}{|\mu|(z)}\,dz{\,dr}\coloneqq J_1+J_2.
    \end{align*}
    for some constant $c=c(n)$. First, we observe that
    \begin{align*}
        J_1\leq c\delta M_{R/4}(|\mu|)(z_1),
    \end{align*}
where $c=c(n)$. We next note from H\"older's inequality that
    \begin{align*}
        J_2\leq c\int^{R/4}_{\delta}\left(\int_{Q_r(z_1)}|\mu|^q\,dz\right)^{\frac1q}r^{-\frac{4n+2}{q}}{\,dr}\leq c\delta^{\frac{q-(4n+2)}{q}}\|\mu\|_{L^q(Q_{R/4}(z_1))},
    \end{align*}
    where $c=c(n,q)$. Combining the estimates $J_1$ and $J_2$ yields
    \begin{align*}
        I_1^{|\mu|}(z_1,R/4)\leq \delta M_{R/4}(|\mu|)(z_1)+c\delta^{\frac{q-(4n+2)}{q}}\|\mu\|_{L^q(Q_R(z_0))}.
    \end{align*}
    We now choose $\delta=\left(\frac{\|\mu\|_{L^q(Q_R(z_0))}}{M_{R/4}(|\mu|)(z_1)}\right)^{\frac{q}{4n+2}}$ and observe that
    \begin{align*}
        I_1^{|\mu|}(z_1,R/4)\leq c {\|\mu\|_{L^q(Q_R(z_0))}}^{\frac{q}{4n+2}}M_{R/4}(|\mu|)^{\frac{4n+2-q}{4n+2}}(z_1)
    \end{align*}
    holds for some $c=c(n,q)$ and any $z_1\in Q_{R/4}(z_0)$.
    Therefore, we get 
    \begin{align*}
        \left\|I_1^{|\mu|}(z,R/4)\right\|_{L^{\frac{q(4n+2)}{4n+2-q}}(Q_{R/4}(z_0))}&\leq  c{\|\mu\|_{L^q(Q_R(z_0))}}^{\frac{q}{4n+2}}\|M_{R/4}(|\mu|)\|_{L^q(Q_{R/4}(z_0))}^{\frac{4n+2-q}{4n+2}}\\
        &\leq c\|\mu\|_{L^q(Q_R(z_0))}
    \end{align*}
    for some constant $c=c(n,q)$, where we have used the standard strong $p$-$p$ estimates for the maximal function. Thus, we have proved \eqref{ineq1.riesz}. 
    
    On the other hand, by following the same lines as in the proof of \cite[Lemma 2.11]{DieKimLeeNow24p} with the cylinder $Q_{2R}^{s}$ replaced by the kinetic cylinder $Q_{R}(z_0)$, we obtain (2.8) in \cite[Lemma 2.11]{DieKimLeeNow24p} with $s=1$, $g=\mu$ and $R=\rho$. Therefore, we deduce \eqref{ineq2.riesz}, which completes the proof.
\end{proof}
\subsection{Technical lemmas}
Next, we mention two technical lemmas which will be useful in the remainder of this paper (see \cite[Lemma 6.1]{Giu03} and \cite[Lemma 3.4]{HanLin11}).
\begin{lemma}\label{lem.tech}
    Let $\varphi:[r_1,r_2]\to\bbR_+$ be a bounded function. Assume that for all $r_1\leq \rho<R\leq r_2$, we have
    \begin{equation*}
        \varphi(\rho)\leq \varphi(R)/2+M(R-\rho)^{-N},
    \end{equation*}
    where $M,N>0$. Then we have 
    \begin{equation*}
        \varphi(r_1)\leq cM(r_2-r_1)^{-N}
    \end{equation*}
    for some constant $c=c(N)$.
\end{lemma}
\begin{lemma}\label{lem.tech2}
    Let $\phi(\rho):[0,R]\to\bbR_+$ be a bounded and non-decreasing function. Assume that for all $0<\rho\leq r\leq R$, we have
    \begin{align*}
        \phi(\rho)\leq M\left[(\rho/r)^\sigma+\epsilon\right]\phi(r)+Nr^{\overline{\sigma}},
    \end{align*}
    where $M,N,\sigma,\overline{\sigma}>0$ with $\sigma>\overline{\sigma}$. Then for any $\gamma\in(\overline{\sigma},\sigma)$, there is a constant $\epsilon_0=\epsilon_0(M,\sigma,\overline{\sigma},\gamma)$ such that if $\epsilon<\epsilon_0$, then for any $0<\rho\leq r\leq R$,
    \begin{align*}
         \phi(\rho)\leq c\left((\rho/r)^\gamma\phi(r)+N\rho^{\overline{\sigma}}\right)
    \end{align*}
    holds, where $c=c(M,\sigma,\overline{\sigma},\gamma)$.
\end{lemma}

\subsection{Zero-order regularity results for solutions to kinetic equations}
We provide a H\"older regularity result for solutions to kinetic Fokker Planck equations that is a rather straightforward consequence of \cite[Theorem 4]{GolImbMouVas19}.

\begin{lemma}\label{lem.dnm}
    Let $f\in H^{1}_{\mathrm{kin}}(W\times V)$ be a weak solution to \eqref{eq.main} with $\mu\in L^\infty$ and $G \equiv 0$. If $a$ satisfies Assumption \ref{assump}, then there is a constant $\alpha_0=\alpha_0(n,\Lambda)\in(0,1)$ such that
    \begin{equation*}
        r^{\alpha_0}[f]_{C_{\mathrm{kin}}^{\alpha_0}(Q_{r/2}(z_0))}+\|f\|_{L^\infty(Q_{r/2}(z_0))}\leq c\left(\dashint_{Q_r(z_0)}|f|\,dz+r^{2} \|\mu\|_{L^\infty(Q_r(z_0))}\right)
    \end{equation*}
    for some constant $c=c(n,\Lambda)$, whenever $Q_{r}(z_0)\subset W\times V$.
\end{lemma}
\begin{proof}
In view of Lemma \ref{lem.scale}, it suffices to show the assertion for $Q_r(z_0)=Q_1$.
We first observe that $f$ is a weak solution to the following linear Fokker-Planck equation
\begin{equation}\label{eq.dnm}
    \partial_t f+v\cdot\nabla_xf-\divergence(A(t,x,v)\nabla_v f)=\mu\quad\text{in }Q_1,
\end{equation}
where 
\begin{equation}\label{eq2.dnm}
    A(t,x,v)\coloneqq \int_{0}^1\nabla_\xi a(t,x,v,s\nabla_v f)\,ds
\end{equation}
satisfies
\begin{equation}\label{unifell}
    \Lambda^{-1}|\zeta|^2\leq A(t,x,v)\zeta\cdot\zeta\leq \Lambda |\zeta|^2\quad\text{for any }\zeta\in\bbR^n.
\end{equation}
This follows from the fact that $a(t,x,v,0)=0$, \eqref{ass.coef} and 
\begin{equation*}
    -\divergence(a(t,x,v,\nabla_vf))=-\divergence(a(t,x,v,\nabla_vf)-a(t,x,v,0))=-\divergence(A(t,x,v)\nabla_v f).
\end{equation*}

Let us fix $3/4\leq \rho<R\leq 1$. Then by Lemma \ref{lem.cov2}, there is a collection $\{Q_r(z_k)\}$ with $r=(R-\rho)/(1000\sqrt{n})$ such that 
    \begin{equation}\label{inc.dnm}
        Q_\rho\subset\bigcup_{k}Q_r(z_k)\subset \bigcup_{k}Q_{4r}(z_k)\subset Q_R.
    \end{equation}
    By \cite[Theorem 4]{GolImbMouVas19}, we have 
    \begin{align*}
        \|f\|_{L^\infty(Q_r(z_k))}\leq c\left[ (R-\rho)^{-(2n+1)}\|f\|_{L^2(Q_{2r}(z_k))}+\|\mu\|_{L^\infty(Q_{2r}(z_k))}\right].
    \end{align*}
    This implies 
    \begin{equation*}
        \sup_{Q_\rho}|f| \leq c\left[(R-\rho)^{-(2n+1)}\|f\|_{L^2(Q_{R})}+\|\mu\|_{L^\infty(Q_{1})}\right]
    \end{equation*}
    for some constant $c=c(n,\Lambda)$.
    In light of Young's inequality, we have 
    \begin{align*}
        \sup_{Q_\rho}|f| &\leq c\left[(R-\rho)^{-(2n+1)}\sup_{Q_R}|f|^{\frac12}\|f\|_{L^1(Q_R)}^{\frac12}+\|\mu\|_{L^\infty(Q_{1})}\right]\\
        &\leq \frac{1}{2}\sup_{Q_R}|f|+c\left[(R-\rho)^{-(4n+2)}\|f\|_{L^1(Q_{1})}+\|\mu\|_{L^\infty(Q_{1})}\right].
    \end{align*}
    By Lemma \ref{lem.tech}, we get
    \begin{equation}\label{dnm.ineq1}
        \sup_{Q_{3/4}}|f|\leq c\left(\|f\|_{L^1(Q_{1})}+\|\mu\|_{L^\infty(Q_{1})}\right)
    \end{equation}
    for some constant $c=c(n,\Lambda)$. In addition, by \cite[Theorem 4]{GolImbMouVas19} we get
    \begin{equation}\label{dnm.ineq2}
[f]_{C_{\mathrm{kin}}^{\alpha_0}(Q_{1/2})}+\|f\|_{L^\infty(Q_{1/2})}\leq c\left(\|f\|_{L^2(Q_{3/4})}+\|\mu\|_{L^\infty(Q_{3/4})}\right),
    \end{equation}
    where $\alpha_0=\alpha_0(n,\Lambda)\in(0,1)$ and $c=c(n,\Lambda)$. Combining \eqref{dnm.ineq1} and \eqref{dnm.ineq2} yields the desired estimate.
\end{proof}
Using this, we prove the following lemma.
\begin{lemma}\label{lem.dnmaff}
    Let $\mu=0$, $G \equiv 0$ and $f\in H^1_{\mathrm{kin}}(W\times V)$ be a weak solution to \eqref{eq.main}, where $a$ satisfies Assumption \ref{assump} and does not depend on $x,v$, that is, $a(t,x,v,\nabla_v f)=a(t,\nabla_v f)$. Then there is a constant $\alpha_0=\alpha_0(n,\Lambda)\in(0,1)$ such that for any $Q_r(z_0)\subset W\times V$, we have
    \begin{equation*}
        \|f-l\|_{L^\infty(Q_{r/2}(z_0))}+r^{\alpha_0}[f-l]_{C_{\mathrm{kin}}^{\alpha_0}(Q_{r/2}(z_0))}\leq c\dashint_{Q_r(z_0)}|f-l|\,dz,
    \end{equation*}
    where $l(v)=\mathcal{A}\cdot (v-v_0)+b$ for some $\mathcal{A}\in\bbR^n$, $b\in\bbR$, and any $v,v_0\in\setR^n$.
\end{lemma}
\begin{proof}
By Lemma \ref{lem.scale}, we may assume $Q_r(z_0)=Q_1$. Next, using the fact that $a(t,\nabla_v l)$ is constant with respect to the $v$ variable, we observe that
    \begin{equation*}
        \partial_t l+v\cdot \nabla_x l-\divergence_v(a(t,\nabla_v l))=0.
    \end{equation*}
    Therefore, we have 
    \begin{equation*}
        \partial_t (f-l)+v\cdot \nabla_x (f-l)-\divergence_v(A_l\nabla_v(f-l) )=0\quad\text{in }Q_1,
    \end{equation*}
    where 
    \begin{equation*}
        A_l(t,x,v)\coloneqq \int_{0}^1 \nabla_\xi a(t,s\nabla_v f+(1-s)\nabla_v l)\,ds
    \end{equation*} 
satisfies \eqref{unifell} with $A$ replaced by $A_l$. Thus, by Lemma \ref{lem.dnm}, we have the desired estimate.
\end{proof}
We now recall a Poincar\'e-type inequality observed in \cite{LitNys21,JesCyrCle24,AlbArmMouNov24}.
\begin{lemma}\label{lem.spi}
    Let $f\in H^1_{\mathrm{kin}}(W\times V)$ be a weak solution to \eqref{eq.main} with $\mu=0$ and $G \equiv 0$. 
    If $a$ satsifies Assumption \ref{assump}, then for any $q\in(1,2]$,
    \begin{equation*}
        \|f-(f)_{Q_r(z_0)}\|_{L^q(Q_r(z_0))}\leq cr\|\nabla_vf\|_{L^q(Q_{r}(z_0))}
    \end{equation*}
    for some constant $c=c(n,\Lambda,q)$, whenever $Q_{2r}(z_0)\subset W\times V$.
\end{lemma}
\begin{proof}
    As in the proof of Lemma \ref{lem.dnm}, we have that $f$ is a weak solution to 
    \begin{equation*}
    \partial_t f+v\cdot\nabla_xf-\divergence(A(t,x,v)\nabla_v f)=0\quad\text{in }Q_{r}(z_0),
\end{equation*}
where $A$ is determined as in \eqref{eq2.dnm}. By \cite[Corollary 21]{JesCyrCle24}, we have the desired estimate.
\end{proof}

We end this section with the following reverse H\"older inequality.
\begin{lemma}\label{lem.self}
    Let $f$ be a weak solution to \eqref{eq.main} with $\mu=0$ and $ G\equiv 0$. If $a$ satsifies Assumption \ref{assump}, then there is constant $c=c(n,\Lambda)\geq1$ such that 
    \begin{align}\label{ineq.self}
        \left(\dashint_{Q_{r/2}(z_0)}|\nabla_vf|^{2}\,dz\right)^{\frac1{2}}\leq c\dashint_{Q_{r}(z_0)}|\nabla_vf|\,dz,
    \end{align}
    whenever $Q_{r}(z_0)\subset W\times V$.
\end{lemma}
\begin{proof}
We first note $f$ is a weak solution to \eqref{eq.dnm} with \eqref{unifell}. Therefore, we have a reverse H\"older type inequality of $f$ as in \cite[Theorem 6]{GolImbMouVas19}. In light of this estimate and Lemma \ref{lem.spi}, 
we are now able to follow the same lines as in the proof of \cite[Theorem 2]{JesCyrCle24} with $F=0$. Thus, \eqref{ineq.self} directly follows.
\end{proof}

\section{Gradient H\"older regularity for homogeneous nonlinear kinetic equations with constant coefficients} \label{sec3}
In this section, we are going to prove the gradient H\"older regularity of solutions to \eqref{eq.main} in the case when $a(t,x,v,\xi)=a(t,\xi)$ and $\mu \equiv 0$, $G \equiv 0$.

For $h\in\setR^n$, we introduce difference quotient operators.
\begin{equation*}
    f^x_h(t,x,v)\coloneqq f(t,x+h,v)\quad\text{and}\quad f^v_h(t,x,v)\coloneqq f(t,x,v+h)
\end{equation*}
\begin{equation*}
    \delta_h^xf\coloneqq f^x_h-f\quad\text{and}\quad \delta_h^vf\coloneqq f^v_h-f.
\end{equation*}
We first prove the H\"older regularity of $\nabla_x f$.
\begin{lemma}\label{lem.grax}
    Let $f\in H^1_{\mathrm{kin}}(W\times V)$ be a weak solution to \eqref{eq.main} with $\mu \equiv 0$, $G \equiv 0$, where $a$ satisfies Assumption \ref{assump} and does not depend on $x,v$, that is, $a(t,x,v,\nabla_v f)=a(t,\nabla_v f)$. Then there is a constant $\alpha_0=\alpha_0(n,\Lambda)\in(0,1)$ such that $\nabla_x f\in C^{\alpha_0}_{\mathrm{kin}}(W \times V)$ and
    \begin{equation*}
        r^{\alpha_0}[\nabla_x f]_{C_{\mathrm{kin}}^{\alpha_0}(Q_{r/2}(z_0))}+\|\nabla_xf\|_{L^\infty(Q_{r/2}(z_0))}\leq cr^{-1}\dashint_{Q_r(z_0)}|f-l|\,dz
    \end{equation*}
     whenever $Q_r(z_0)\subset W\times V$, where $c=c(n,\Lambda)$ and $l(v)=\mathcal{A}\cdot (v-v_0)+b$ for some $\mathcal{A}\in\bbR^n$, $b\in\bbR$, and $v,v_0\in\setR^n$ is an arbitrary affine function. 
\end{lemma}
\begin{proof}
We may assume $r=1$ and $z_0=0$. 
We first note from Lemma \ref{lem.dnmaff} that $f\in C^{\alpha_0}_{\mathrm{kin}}(Q_1)$ and there is a constant $c=c(n,\Lambda)$ such that for any $Q_{2r}(z_1)\subset Q_1$,
\begin{equation}\label{ineq1.grax}
\begin{aligned}
   r^{\alpha_0}\sup_{(t,v)\in I_r(t_1)\times B_r(v_1)}[(f-l)(t,\cdot,v)]_{C_x^{\frac{\alpha_0}3}\left(B^x_{r}(x_1+v_1(t-t_1))\right)}&\leq cr^{\alpha_0}[f-l]_{C_{\mathrm{kin}}^{\alpha_0}(Q_r(z_1))}\\
   &\leq c\|f-l\|_{L^1(Q_{2r}(z_1))}
\end{aligned}
\end{equation}
holds, where the ball $B^x_r(\cdot)$ is defined in \eqref{defn.xball} and $[(f-l)(t,\cdot,v)]_{C_x^{\frac{\alpha_0}3}}$ means the usual H\"older semi-norm with respect to the $x$-variable.
Let us fix $h\in(0,1/10^{10}]$.
We note that
    \begin{equation*}
        \partial_t \delta_h^x(f-l)+v\cdot \nabla_x \delta_h^x(f-l)-\divergence_v(A_h\nabla_v\delta_h^x(f-l) )=0\quad\text{weakly in }Q_{3/4},
    \end{equation*}
    where we have used the fact that $\delta_h^x(f-l)=\delta_h^xf$ and
    \begin{equation*}
        A_h(t,x,v)\coloneqq\int_{0}^1\nabla_\xi a(t,s\nabla_v f^x_h(t,x,v)+(1-s)\nabla_v f(t,x,v))\,ds 
    \end{equation*} 
    satisfies \eqref{unifell} with $A$ replaced by $A_h$. By Lemma \ref{lem.dnm} with $\mu=0$, we obtain 
    \begin{equation}\label{ineq11.grax}
        \left\|\frac{\delta_h^x(f-l)}{|h|^{\frac{\alpha_0}3}}\right\|_{C_{\mathrm{kin}}^{\alpha_0}(Q_{r_0/4}(z_1))}\leq c\left\|\frac{\delta_h^x(f-l)}{|h|^{\frac{\alpha_0}3}}\right\|_{L^1(Q_{r_0/2}(z_1))}
    \end{equation}
    for some $c=c(n,\Lambda)$, whenever $Q_{r_0/2}(z_1)\Subset Q_{3/4}$ with $r_0\coloneqq 1/1000$.
    Using \eqref{ineq1.grax}, we get
    \begin{align*}
        \left\|\frac{\delta_h^x(f-l)}{|h|^{\frac{\alpha_0}3}}\right\|_{L^1(Q_{r_0/2}(z_1))}&\leq\sup_{(t,x)\in I_{\frac{r_0}2}(t_1)\times B_{\frac{r_0}2}(v_1)}[(f-l)(t,\cdot,v)]_{C_x^{\frac{\alpha_0}3}(B^x_{r_0}(x_1+(t-t_1)v_1))}\\
        &\leq c\|f-l\|_{L^{1}(Q_{2r_0}(z_1))}.
    \end{align*}
    Thus, we deduce
    \begin{equation}\label{ineq2.grax}
    \begin{aligned}
       &\sup_{(t,v)\in I_{r_0/4}(t_1)\times B_{r_0/4}(v_1)}\left\|\frac{\delta_h^x(f-l)(t,\cdot,v)}{|h|^{\frac{\alpha_0}3}}\right\|_{C_x^{\frac{\alpha_0}3}(B^x_{r_0/4}(x_1+(t-t_1)v_1))}\\
       &\quad\leq c\|f-l\|_{L^1(Q_{2r_0}(z_1))}.
    \end{aligned}
    \end{equation}
    In view of \cite[Lemma A.1.2]{FerRos24}, we obtain
    \begin{equation*}
        \sup_{(t,v)\in I_{\frac{r_0}4}(t_1)\times B_{\frac{r_0}4}(v_1)}\left\|(f-l)(t,\cdot,v)\right\|_{C_x^{\frac{2{\alpha_0}}3}(B^x_{r_0/4}(x_1+(t-t_1)v_1))}\leq c\|f-l\|_{L^1(Q_{2r_0}(z_1))}
    \end{equation*}
    for some constant $c=c(n,\Lambda)$.

    We note that there is a positive integer $i=i(n,\Lambda)$ such that 
    \begin{equation*}
        \frac{i\alpha_0}{3}<1<\frac{(i+1)\alpha_0}{3}.
    \end{equation*}
    We now iterate $i-1$ times the above procedure to see that
    \begin{align}\label{ineq21.grax}
       & \sup_{(t,v)\in I_{\frac{r_0}{2^i}}(t_1)\times B_{\frac{r_0}{2^i}}(v_1)}\left\|(f-l)(t,\cdot,v)\right\|_{C_x^{\frac{i\alpha_0}{3}}\left(B^x_{\frac{r_0}{2^i}}(x_1+(t-t_1)v_1)\right)}\leq c\|f-l\|_{L^1(Q_{2r_0}(z_1))}
    \end{align}
    for some constant $c=c(n,\Lambda)$. As in \eqref{ineq11.grax}, we observe 
   \begin{equation*}
        \left\|\frac{\delta_h^x(f-l)}{|h|^{\frac{i\alpha_0}3}}\right\|_{C_{\mathrm{kin}}^{{\alpha_0}}(Q_{2^{-(i+1)}r_0}(z_1))}\leq c\left\|\frac{\delta_h^x(f-l)}{|h|^{\frac{i\alpha_0}3}}\right\|_{L^1(Q_{2^{-i}r_0}(z_1))}.
    \end{equation*}
     Using \eqref{ineq21.grax}, we get 
    \begin{equation*}
    \begin{aligned}
       &\sup_{(t,v)\in I_{\frac{r_0}{2^{i+1}}}(t_1)\times B_{\frac{r_0}{2^{i+1}}}(v_1)}\left\|\frac{\delta_h^x(f-l)(t,\cdot,v)}{|h|^{\frac{i\alpha_0}3}}\right\|_{C_x^{\frac{\alpha_0}3}\left(B^x_{\frac{r_0}{2^{i+1}}}(x_1+(t-t_1)v_1)\right)}
\leq c\|f-l\|_{L^1(Q_{2r_0}(z_1))}.
    \end{aligned}
    \end{equation*}
    We now employ \cite[Lemma A.1.2]{FerRos24} once more to obtain
    \begin{equation}\label{ineq3.grax}
    \begin{aligned}
        &\sup_{(t,v)\in I_{\frac{r_0}{2^{i+1}}}(t_1)\times B_{\frac{r_0}{2^{i+1}}}(v_1)}\left\|\nabla_x(f-l)(t,\cdot,v)\right\|_{L^{\infty}\left(B^x_{\frac{r_0}{2^{i+1}}}(x_1+(t-t_1)v_1)\right)}\leq c\|f-l\|_{L^1(Q_{2r}(z_1))},
    \end{aligned}
    \end{equation}
    where $c=c(n,\Lambda)$. Since we have proved \eqref{ineq3.grax} whenever $Q_{2r}(z_1)\Subset Q_1$, by standard covering arguments we derive 
    \begin{align}\label{ineq4.grax}
        \sup_{(t,v)\in I_{7/8}\times B_{7/8}}\left\|\nabla_x(f-l)(t,\cdot,v)\right\|_{L^\infty\left(B_{7/8}\right)}&\leq c\|f-l\|_{L^1(Q_{1})}.
    \end{align}
    In addition, by using \eqref{ineq11.grax} with $h=\epsilon e_i$, where $\epsilon>0$ and $e_i$ is the standard $i$-th unit vector, we conclude that
    \begin{equation*}
        \left\|\frac{\delta_h^x(f-l)}{|h|}\right\|_{C_{\mathrm{kin}}^{\alpha_0}(Q_{1/2})}\leq c\left\|\frac{\delta_h^x(f-l)}{|h|}\right\|_{L^1(Q_{3/4})}.
    \end{equation*}
    By passing to the limit $\epsilon\to0$ and using \eqref{ineq4.grax}, we arrive at
    \begin{equation*}
        \left\|\partial_{x_i}(f-l)\right\|_{C_{\mathrm{kin}}^{\alpha_0}(Q_{1/2})}\leq c\left\|f-l\right\|_{L^1(Q_{1})}
    \end{equation*}
    for some constant $c=c(n,\Lambda)$ and every $i \in \{1,\dots,n\}$. This completes the proof.
\end{proof}

We are now ready to prove also H\"older regularity of the velocity gradient of $f$.

\begin{theorem}\label{thm.hol}
	Let $f\in H^1_{\mathrm{kin}}(W\times V)$ be a weak solution to \eqref{eq.main} with $\mu \equiv 0$, $G \equiv 0$, where $a$ satisfies Assumption \ref{assump} and does not depend on $x,v$, that is, $a(t,x,v,\nabla_v f)=a(t,\nabla_v f)$. Then there is a constant $\alpha=\alpha(n,\Lambda)\in(0,1)$ such that $\nabla_v f\in C^{\alpha}_{\mathrm{kin}}(W\times V)$ and
	\begin{align*}
		r^\alpha[\nabla_vf]_{C^\alpha(Q_r(z_0))}+\|\nabla_vf-(\nabla_vf)_{Q_r(z_0)}\|_{L^\infty(Q_r(z_0))}
		\leq c\dashint_{Q_{2r}(z_0)}|\nabla_vf-(\nabla_vf)_{Q_{2r}(z_0)}|\,dz
	\end{align*}
	for some constant $c=c(n,\Lambda)$, whenever $Q_{2r}(z_0)\subset W\times V$.
\end{theorem}

\begin{proof}
We may assume $Q_r(z_0)=Q_1$ and let $l\coloneqq(\nabla_v f)_{Q_{1/2}} \cdot v+(f)_{Q_{1/2}}$. As in the proof of \eqref{ineq1.grax}, together with Lemma \ref{lem.dnm}, we observe that
\begin{equation}\label{ineq1.hol}
\begin{aligned}
   r^{\alpha_0}\sup_{{t\in I_r(t_1)
   }}\sup_{x\in B^x_r(x_1+(t-t_1)v_1)}[(f-l)(t,x,\cdot)]_{C_v^{{\alpha_0}}\left(B_{r}(v_1)\right)}
   \leq c\|f-l\|_{L^1(Q_{2r}(z_1))}
\end{aligned}
\end{equation}
for some constant $c=c(n,\Lambda)$, whenever $Q_{2r}(z_1)\subset Q_1$. Let us fix $h\in(0,1/10^{10}]$. Then we observe that
    \begin{equation*}
        \partial_t \delta_h^v(f-l)+v\cdot \nabla_x \delta_h^v(f-l)-\divergence_v(A_h\nabla_v\delta_h^v(f-l) )=-h\cdot\nabla_x f_h^v\quad\text{weakly in }Q_{3/4},
    \end{equation*}
    where
    \begin{equation*}
        A_h(t,x,v)\coloneqq\int_{0}^1\nabla_\xi a(t,s\nabla_v f^v_h(t,x,v)+(1-s)\nabla_v f(t,x,v))\,ds 
    \end{equation*} 
    satisfies \eqref{unifell} with $A$ replaced by $A_h$.
    By \eqref{ineq1.hol} and Lemma \ref{lem.grax}, we get
   \begin{align*}
        \left\|\frac{\delta_h^v(f-l)}{|h|^{\alpha_0}}\right\|_{C_{\mathrm{kin}}^{\alpha_0}(Q_{r_0/2}(z_1))}&\leq \left\|\frac{\delta_h^v(f-l)}{|h|^{\alpha_0}}\right\|_{L^1(Q_{r_0}(z_1))}+|h|^{1-\alpha_0}\|\nabla_x f_h^v\|_{L^\infty(Q_{r_0}(z_1))}\\
        &\leq c\|f-l\|_{L^1(Q_{2r_0}(z_1))},
    \end{align*}
    where $c=c(n,\Lambda)$ and $r_0\coloneqq 1/1000$, whenever $Q_{4r_0}(z_1)\Subset Q_1$. We note that there is a positive integer $i=i(n,\Lambda)$ such that 
    \begin{equation*}
        i\alpha_0<1<(i+1)\alpha_0.
    \end{equation*}
    By following the same inductive step as in the proof of Lemma \ref{lem.grax}, we have 
    \begin{align*}
        \left\|\frac{\delta_h^v(f-l)}{|h|^{i\alpha_0}}\right\|_{C_{\mathrm{kin}}^{\alpha_0}(Q_{r_0/2^{i+1}}(z_1))}\leq c\|f-l\|_{L^1(Q_{2r_0}(z_1))}.
    \end{align*}
    In light of \cite[Lemma A.1.2]{FerRos24}, we now get
    \begin{equation}\label{ineq2.hol}
    \begin{aligned}
        &\sup_{t\in I_{\frac{r_0}{2^{i+1}}}(t_1)}\sup_{x\in B^x_{\frac{r_0}{2^{i+1}}}(x_1+(t-t_1)v_1)}\left\|\nabla_v(f-l)(t,x,\cdot)\right\|_{L^\infty\left(B_{\frac{r_0}{2^{i+1}}}(v_1)\right)}\\
        &\leq c\|f-l\|_{L^1(Q_{2r}(z_1))}
    \end{aligned}
    \end{equation}
    for some constant $c=c(n,\Lambda)$. By standard covering arguments, we obtain the Lipschitz estimate
    \begin{align}\label{ineq3.hol}
        \|\nabla_v(f-l)\|_{L^\infty(Q_{7/8})}\leq c\|f-l\|_{L^1(Q_{15/16})}.
    \end{align}

    Since we have 
    \begin{align*}
        \left\|\frac{\delta_h^v(f-l)}{|h|}\right\|_{C_{\mathrm{kin}}^{\alpha_0}(Q_{1/2})}&\leq \left\|\frac{\delta_h^v(f-l)}{|h|}\right\|_{L^1(Q_{3/4})}+\|\nabla_x f_h^x\|_{L^\infty(Q_{3/4})},
    \end{align*}
    by taking $h=\epsilon e_i$, where for any $i \in \{1,\dots,n\}$ we denote by $e_i$ the standard unit vector and taking limit $\epsilon\to0$, we obtain
    \begin{align*}
        \|\nabla_v(f-l)\|_{C_{\mathrm{kin}}^{\alpha_0}(Q_{1/2})}\leq c\left(  \|\nabla_v(f-l)\|_{L^\infty(Q_{7/8})}+\|\nabla_x f\|_{L^\infty(Q_{3/4})}\right)
    \end{align*}
    for some constant $c=c(n,\Lambda)$.  Therefore, using \eqref{ineq3.hol} and Lemma \ref{lem.grax}, we get
    \begin{align}\label{ineq4.hol}
        \|\nabla_v(f-l)\|_{C_{\mathrm{kin}}^{\alpha_0}(Q_{1/2})}\leq c\|f-l\|_{L^1(Q_{15/16})}.
    \end{align}
    We now apply H\"older's inequality, Lemma \ref{lem.spi} and Lemma \ref{lem.self} to right-hand side of \eqref{ineq4.hol} to see that 
    \begin{equation}\label{ineq5.hol}
    \begin{aligned}
        \|\nabla_v(f-l)\|_{C_{\mathrm{kin}}^{\alpha_0}(Q_{1/2})}\leq c\|f-l\|_{L^2(Q_{15/16})}&\leq c\|\nabla_v(f-l)\|_{L^2(Q_{15/16})}\\
        &\leq c\|\nabla_v(f-l)\|_{L^1(Q_{1})}
    \end{aligned}
    \end{equation}
    for some constant $c=c(n,\Lambda)$. Plugging the fact that 
    \begin{align*}
        \|\nabla_v(f-l)\|_{L^1(Q_{1})}=\int_{Q_1}|\nabla_vf-(\nabla_vf)_{Q_{1/2}}|\,dz\leq c\dashint_{Q_1}|\nabla_vf-(\nabla_vf)_{Q_1}|\,dz
    \end{align*}
    into the right-hand side of \eqref{ineq5.hol} yields the desired estimate.
\end{proof}

\section{Comparison estimates} \label{sec4}
In this section, we provide several comparison estimates. For the remainder of this paper, we always assume that $a$ satisfies Assumption \ref{assump} for some $\Lambda \geq 1$.

Since the existence of the weak solution to the corresponding initial boundary value problem of \eqref{eq.main} is established for smooth domains (see \cite{LitNys21,GarNys23}), we need to regularize the kinetic cylinder $Q_r(z_0)$.
For any $r>0$ and $z_0\in\bbR^{2n+1}$, we define a cylinder $Q^{t,x}_{r,v_0}(t_0,x_0)\subset \bbR^{n+1}$ by
\begin{equation}\label{defn.txcy}
    Q^{t,x}_{r,v_0}(t_0,x_0)\coloneqq \{(t,x)\in \bbR^{n+1}\,:\, t\in I_r(t_0)\text{ and } |x-x_0-v_0(t-t_0)|<r^3\}
\end{equation}
to see that 
\begin{equation*}
    Q^{t,x}_{r,v_0}(t_0,x_0)\times B_r(v_0)=Q_r(z_0).
\end{equation*}
In view of \eqref{defn.wrv0} in Appendix \ref{appen}, there is a smooth set $W_{r,v_0}(t_0,x_0)\subset\bbR^{n+1} $ such that
\begin{equation*}
    Q^{t,x}_{r,v_0}(t_0,x_0)\subset W_{r,v_0}(t_0,x_0)\subset W_{5r/4,v_0}(t_0,x_0)\subset Q_{2r,v_0}^{t,x}(t_0,x_0).
\end{equation*}
We now define a regularized cylinder of $Q_r(t_0,x_0,v_0)$ as 
\begin{equation}\label{defn.regcy}
    \mathcal{R}Q_r(t_0,x_0,v_0)\coloneqq W_{r,v_0}(t_0,x_0)\times B_r(v_0).
\end{equation}
By \eqref{appen.sec} in Appendix \ref{appen}, we observe that if $g\in H^{1}_{\mathrm{kin}}(Q_{2r}(z_0))$, then $g$ belongs to the space $\mathcal{W}(\mathcal{R}Q_r(z_0))$ defined in Section \ref{setup}. In addition, we recall the space $\mathcal{W}_0(\mathcal{R}Q_r(z_0))$ which is the closure of the space
\begin{align*}
    &C^{\infty}_{\mathcal{K},0}(\mathcal{R}Q_r(z_0))\coloneqq\{g\in C^{\infty}(\overline{\mathcal{R}Q_r(z_0)})\,:\, g\equiv 0\text{ on }\partial_{\mathcal{K}}(\mathcal{R}Q_r(z_0))\}
\end{align*}
with respect to the norm of $H^{1}_{\mathrm{kin}}$, where the Kolmogorov boundary $\partial_{\mathcal{K}}(\mathcal{R}Q_r(z_0))$ is defined in \eqref{defn.kbdry}. The space was introduced in \cite{LitNys21}.

\begin{definition}
	For $r>0$ and $z_0=(t_0,x_0,v_0) \in \mathbb{R}^{2n+1}$, let $f\in H^{1}_{\mathrm{kin}}(Q_{2r}(z_0))$. We say that $w\in\mathcal{W}(\mathcal{R}Q_r(z_0))$ is a weak solution to  \begin{equation}\label{eq.ivp}
		\left\{
		\begin{alignedat}{3}
			\partial_t w+v\cdot \nabla_x w-\divergence_v(a(t,x,v,\nabla_v w))&= 0&&\qquad \mbox{in  $\mathcal{R}Q_{r}(z_0)$}, \\
			w&=f&&\qquad \mbox{on $\partial_{\mathcal{K}}(\mathcal{R}Q_r(z_0))$},
		\end{alignedat} \right.
	\end{equation}
if $w$ is a weak solution to $\partial_t w+v\cdot \nabla_x w-\divergence_v(a(t,x,v,\nabla_v w))=0$ in $\mathcal{R}Q_{r}(z_0)$ and $w-f\in \mathcal{W}_{0}(\mathcal{R}Q_r(z_0))$.
\end{definition}

We now deduce from \cite{GarNys23} that the following existence result holds.
\begin{lemma}\label{lem.exist}
For $r>0$ and $z_0=(t_0,x_0,v_0) \in \mathbb{R}^{2n+1}$, let $f\in H^{1}_{\mathrm{kin}}(Q_{2r}(z_0))$. Then there is a unique weak solution $w\in\mathcal{W}(\mathcal{R}Q_r(z_0))$ with 
$w-f\in \mathcal{W}_{0}(\mathcal{R}Q_r(z_0))$
to \eqref{eq.ivp}.
\end{lemma}
\begin{proof}
    In \cite{GarNys23}, it is assumed that $a(t,x,v,\xi)$ satisfies for any $\xi_1,\xi_2,\xi\in\bbR^n$,
    \begin{enumerate}
        \item $|a(t,x,v,\xi_1)-a(t,x,v,\xi_2)|\leq \Lambda|\xi_1-\xi_2|$,
        \item $(a(t,x,v,\xi_1)-a(t,x,v,\xi_2))\cdot(\xi_1-\xi_2)\geq \Lambda^{-1}|\xi_1-\xi_2|^2$,
        \item $a(t,x,v,\lambda \xi)=\lambda a(t,x,v,\xi),\quad \forall\lambda\in\bbR\setminus\{0\}.$
    \end{enumerate}
    However, a careful inspection of the proof reveals that the existence result given in \cite[Theorem 1.5]{GarNys23} remains true without the assumption $(c)$. Since \eqref{ass.coef} implies the conditions $(a)$ and $(b)$, there exists a unique weak solution $w$ to \eqref{eq.ivp}.
\end{proof}

Next, we provide zero-order comparison estimates.
\begin{lemma}\label{lem.comp1}
    For $r>0$ and $z_0=(t_0,x_0,v_0) \in \mathbb{R}^{2n+1}$, let $\mu \in L^2(Q_{2r}(z_0))$ and let $f\in H^1_{\mathrm{kin}}(Q_{2r}(z_0))$ be a weak solution to 
    \begin{equation*}
        \partial_t f+v\cdot \nabla_x f-\divergence_v(a(t,x,v,\nabla_v f))=\mu\quad\text{in }Q_{2r}(z_0).
    \end{equation*}
    Then there is a unique weak solution $w\in \mathcal{W}(\mathcal{R}Q_r(z_0))$ to \eqref{eq.ivp}
such that 
\begin{equation*}
    \sup_{t\in I_r(t_0)}\|(f-w)(t,\cdot)\|_{L^1(B^x_{r}(x_0+v_0(t-t_0))\times B_r(v_0))}\leq c|\mu|(Q_{2r}(z_0))
\end{equation*}
with $c=c(n,\Lambda)$.
\end{lemma}
\begin{proof}
By Lemma \ref{lem.scale}, we may assume $Q_{2r}(z_0)=Q_2$ and $|\mu|(Q_2)=1$.
    Let us choose a smooth decreasing function $\oldphi:\bbR\to\bbR$ such that $0\leq\oldphi\leq1$ and $\oldphi(t)=0$ if $t>t_0$ for some $t_0\in(-1,0)$ with further properties to be specified later. Let us define 
    \begin{equation}\label{defn.phij}
    \begin{aligned}
        \Phi_j(\tau)\coloneqq\begin{cases}
            1&\quad\text{if }\tau\geq j+1\\
            \tau-j&\quad\text{if }j\leq \tau<j+1\\
            0&\quad\text{if }-j\leq\tau<j\\
            \tau+j&\quad\text{if }-j-1\leq \tau<-j\\
            -1&\quad\text{if }\tau<-j-1.
        \end{cases}
    \end{aligned}
    \end{equation}
    Then we consider the truncated function $\Phi_j(f-w)\in L^2(W_1;H^1_{0}(B_1))$, where we write $W_1\coloneqq W_{1,0}(0,0)$ and the smooth set $W_{1,0}(0,0)$ is determined in \eqref{defn.wrv0}. We next define
\begin{equation}\label{defn.varphi.comp1}
        \varphi\coloneqq \Phi_j(f-w)\oldphi 
    \end{equation}
    and observe that $\phi\in L^2(W_{1};H_{0}^1(B_1))$.

    Since $f-w\in \mathcal{W}_0(W_{1}\times B_1)$, there is a sequence $f_k-w_k\in C^\infty(\overline{W_1\times B_1})$ with $f_k-w_k\equiv 0$ on $\partial_{\mathcal{K}}(W_1\times B_1)$ and $f_k-w_k\to f-w\text{ in }H^1_{\mathrm{kin}}(W_1\times B_1)$. Therefore, we observe 
    \begin{equation}\label{ineq0.comp1}
    \begin{aligned}
        \int_{W_1}\left<(\partial_t+v\cdot\nabla_x)(f-w),\varphi\right>\,dx\,dt&=\lim_{k\to\infty}\int_{W_1}\left<(\partial_t+v\cdot\nabla_x)(f_k-w_k),\varphi_k\right>\,dx\,dt\\
        &=\lim_{k\to\infty}\int_{B_1}\int_{W_1}\left((\partial_t+v\cdot\nabla_x)(f_k-w_k)\right)\varphi_k\,dz\\
        &\eqqcolon\lim_{k\to\infty}J_k,
    \end{aligned}
    \end{equation}
    where we write $\varphi_k\coloneqq \Phi_j(f_k-w_k)\oldphi$ and we have used the fact that $\phi_k\to\phi$ in $L^2(W_1;H^1(B_1))$.
    We now use integration by parts and the divergence theorem to obtain
    \begin{equation}\label{ineq01.comp1}
    \begin{aligned}
        J_k&=\int_{B_1}\int_{W_1}(\partial_t+v\cdot\nabla_x)\left(\oldphi(t)\int_{0}^{f_k-w_k}\Phi_j(s)\,ds\right)\,dz\\
        &\quad-\int_{B_1}\int_{W_1}\partial_t\oldphi(t)\int_{0}^{f_k-w_k}\Phi_j(s)\,ds\,dz\\
        &=\int_{B_1}\left(\int_{\partial W_1}\oldphi(t)\left(\int_{0}^{f_k-w_k}\Phi_j(s)\,ds\right)\,(1,v)\cdot N_{t,x}\,dS_{t,x}\right)\,dv\\
        &\quad-\int_{B_1}\int_{W_1}\partial_t\oldphi(t)\int_{0}^{f_k-w_k}\Phi_j(s)\,ds\,dz,
    \end{aligned}
    \end{equation}
    where we denote by $N_{t,x}$ the outer unit normal vector to $W_1$ at $(t,x)$ and by $dS_{t,x}$ the surface measure on $\partial W_1$, respectively. Since $f_k-w_k\equiv 0$ on $\partial_{\mathcal{K}}(W_1\times B_1)$, we have 
    \begin{equation*}
        J_k\geq -\int_{B_1}\int_{W_1}\partial_t\oldphi(t)\int_{0}^{f_k-w_k}\Phi_j(s)\,ds\,dz.
    \end{equation*}
Since we will choose $\oldphi$ such that $\partial_t\oldphi(t)\leq0$, using Fatou's lemma, we have 
\begin{align}\label{ineq1.comp1}
    \int_{W_1}\left<(\partial_t+v\cdot\nabla_x)(f-w),\varphi\right>\,dx\,dt\geq \int_{B_1}\int_{W_1}(-\partial_t\oldphi(t))\int_{0}^{f-w}\Phi_j(s)\,ds\,dz.
\end{align}

     By testing the weak formulation of the equation 
    \begin{equation}\label{eq2.comp1}
        \partial_t(f-w)+v\cdot\nabla_x(f-w)-\divergence_v(a(t,x,v,\nabla_v f)-a(t,x,v,\nabla_v w))=\mu\quad\text{in }\mathcal{R}Q_1
    \end{equation}
    with $\phi$ and using \eqref{ineq1.comp1}, we deduce
    \begin{align*}
        &\int_{B_1}\int_{W_1}(-\partial_t\oldphi(t))\int_{0}^{f-w}\Phi_j(s)\,ds\,dz\\
        &\quad+\int_{\mathcal{R}Q_1}\left[(a(t,x,v,\nabla_v f)-a(t,x,v,\nabla_v w))\cdot\nabla_v(\Phi_k(f-w))\right]\,dz\leq\int_{\mathcal{R}Q_1}\mu \varphi\,dz.
    \end{align*}
    Using \eqref{ass.coef} and the fact that $\Phi_j'\geq0$ and $|\Phi_j|,\oldphi\leq1$, we get 
    \begin{equation}\label{ineq4.comp1}
       \int_{B_1}\int_{W_1}(-\partial_t\oldphi(t))\Psi_j(f-w)\,dz\leq |\mu|(Q_2),
    \end{equation}
    where we write $\Psi_j(s)\coloneqq\int_{0}^s\Phi_j(\xi)\,d\xi$ and $\Psi_j\geq0$.
    Let us fix $t_1\in (-1,0)$ and choose $\oldphi(t)$ to be a decreasing smooth function such that $\oldphi(t)\equiv 1$ on $t<t_1-\delta$ and $\oldphi(t)\equiv 0$ on $t>t_1+\delta$, $\oldphi'(t)\leq0$, $\oldphi'(t)=-\delta/2$ on $t_1-\delta/2<t<t_1+\delta/2$, where $\delta<\frac14\min\{t_1-1,|t_1|\}$ to see that
    \begin{align*}
       \int_{t_1-\delta/2}^{t_1}\int_{B_1}\int_{B_1}\frac{\Psi_j(f-w)}{\delta/2}\,dv\,dx\,dt \leq \int_{I_1}\int_{B_1}\int_{B_1}(-\partial_t\oldphi(t))\Psi_j(f-w)\,dv\,dx\,dt
       \leq c|\mu|(Q_2).
    \end{align*}
    By taking $\delta\to0$ and using Lemma \ref{lem.diff}, since $t_1 \in (-1,0)$ is arbitrary, we obtain
    \begin{align}\label{ineq5.comp1}
        \sup_{t\in I_1}\int_{B_1}\int_{B_1}\Psi_j(f-w)\,dv\,dx\leq c|\mu|(Q_2)\leq c
    \end{align}
    for some constant $c$. We next observe that 
    \begin{align*}
        \Psi_0(\tau)=\begin{cases}
            -1/2+\tau&\quad\text{if }\tau\geq1,\\
            |\tau|^2/2&\quad\text{if }-1\leq\tau<1,\\
            -1/2-\tau&\quad\text{if }\tau\leq-1.
        \end{cases}
    \end{align*}
   Let us fix $\tau\in I_1$ to see that    \begin{align*}
        \int_{B_1}\int_{B_1}\Psi_0(f-w)(\tau,x,v)\,dv\,dx&= \int_{\mathcal{B}'}\frac{(f-w)^2(\tau,x,v)}{2}\,dv\,dx\\
        &\quad+\int_{\mathcal{B}''}{|(f-w)|(\tau,x,v)}\,dv\,dx-\frac{|\mathcal{B}''|}{2},
    \end{align*}
    where we write 
    \begin{align*}
        \mathcal{B}'\coloneqq\{(x,v)\in B_1\times B_1\,:|(f-w)(\tau,x,v)|\leq 1\}
    \end{align*}
    and
    \begin{align*}
        \mathcal{B}''\coloneqq\{(x,v)\in B_1\times B_1\,:|(f-w)(\tau,x,v)|> 1\}.
    \end{align*}
    Therefore, we have
    \begin{align*}
        \int_{B_1\times B_1}|(f-w)(\tau,x,v)|\,dv\,dx\leq \int_{B_1}\int_{B_1}\Psi_0(f-w)(\tau,x,v)\,dv\,dx+|B_1|^2 \leq c,
    \end{align*}
    where we have used that $\int_{\mathcal{B}'}{|f-w|}\,dv\,dx\leq {|\mathcal{B}'|} $ and \eqref{ineq5.comp1}. This completes the proof.
\end{proof}
Using certain different quotient techniques originating in \cite{KM1,Min07}, we obtain the following higher differentiability result with respect to the spatial variable $x$, which is crucial to obtain comparison estimates at the gradient level.
\begin{lemma}\label{lem.xreg}
Let $\mu \in L^2(Q_1)$ and let $f\in H^1_{\mathrm{kin}}(Q_1)$ be a weak solution to  
    \begin{equation*}
        \partial_t f+v\cdot \nabla_x f-\divergence_v(a(t,x,v,\nabla_v f))=\mu \quad\text{in }Q_1.
    \end{equation*}
Then 
    \begin{equation*}
        [f]_{L^1_tW^{\gamma,1}_xL^1_v(Q_{1/2})}\leq c\left(\|f\|_{L^1(Q_1)}+\|\mu\|_{L^1(Q_1)}\right),
    \end{equation*}
    where $c=c(n,\Lambda)$ and $\gamma=\gamma(n,\Lambda)\in(0,1)$.
\end{lemma}
\begin{proof}
    Let us take $\beta=1/2$ and choose $|h|<1/(40000\sqrt{n})^{\frac1\beta}$. By Lemma \ref{lem.cov2}, there is a collection $\{Q_{|h|^\beta}(z_k)\}_{k\in\mathcal{K}}$ such that
    \begin{equation}\label{ineq1.xreg}
        Q_{3/4}\subset \bigcup_{k\in\mathcal{K}}Q_{|h|^\beta}(z_k)\subset \bigcup_{k\in\mathcal{K}}Q_{4|h|^\beta}(z_k)\subset Q_1
    \end{equation}
    and
    \begin{equation}\label{ineq2.xreg}
        \sup_{z\in\bbR^{2n+1}}\sum_{k\in\mathcal{K}}\mbox{\Large$\chi$}_{Q_{4|h|^\beta}(z_k)}(z)\leq c(n).
    \end{equation}
    We will prove 
    \begin{equation*}
        \|\delta_h^x f\|_{L^1(Q_{3/4})}\leq c{|h|^{\alpha_0/6}}(\|f\|_{L^1(Q_1)}+\|\mu\|_{L^1(Q_1)}),
    \end{equation*}
    where $c=c(n,\Lambda)$ and the constant $\alpha_0$ is determined in Lemma \ref{lem.dnm}. We first observe 
    \begin{align*}
        \|\delta_h^xf\|_{L^1(Q_{|h|^\beta}(z_k))}\leq \|\delta_h^x(f-w)\|_{L^1(Q_{|h|^\beta}(z_k))}+\|\delta_h^xw\|_{L^1(Q_{|h|^\beta}(z_k))}\coloneqq J_1+J_2,
    \end{align*}
    where $w$ is a weak solution to \eqref{eq.ivp} with $Q_r(z_0)$ replaced by $Q_{2|h|^\beta}(z_k)$.
    In light of Lemma \ref{lem.comp1}, we first get 
    \begin{align*}
        J_1\leq c\|(f-w)\|_{L^1(Q_{2|h|^\beta}(z_k))}\leq c|h|^{2\beta}|\mu|(Q_{4|h|^\beta}(z_k)).
    \end{align*}

    We next observe from Lemma \ref{lem.dnm} and Lemma \ref{lem.comp1} that 
    \begin{align*}
        J_2&\leq |h|^{\frac{\alpha_0}3}[w]_{C_{\mathrm{kin}}^{\alpha_0}\left(Q_{\frac32|h|^\beta}(z_k)\right)}|Q_{\frac32|h|^\beta}(z_k)|\\
        &\leq c|h|^{\frac{\alpha_0}{3}(1-\beta)}\left(\dashint_{{Q_{2|h|^\beta}}(z_k)}|w|\,dz\right)|Q_{4|h|^\beta}(z_k)|\\
        &\leq c|h|^{\frac{\alpha_0}3(1-\beta)}\left(\dashint_{{Q_{2|h|^\beta}}(z_k)}|w-f|\,dz+\dashint_{{Q_{2|h|^\beta}}(z_k)}|f|\,dz\right)|Q_{4|h|^\beta}(z_k)|\\
        &\leq c|h|^{\frac{\alpha_0}3(1-\beta)}\left(|h|^{2\beta}\frac{|\mu|(Q_{4|h|^\beta}(z_k))}{|Q_{4|h|^\beta}(z_k)|} +\dashint_{{Q_{2|h|^\beta}}(z_k)}|f|\,dz\right)|Q_{4|h|^\beta}(z_k)|.
    \end{align*}
    Combining the estimates $J_1$ and $J_2$, we have 
    \begin{align}\label{ineq4.xreg}
         \|\delta_h^xf\|_{L^1(Q_{|h|^\beta}(z_k))}\leq c\left(|h|^{2\beta}|\mu|(Q_{4|h|^\beta}(z_k))+|h|^{\frac{\alpha_0}3(1-\beta)}\int_{{Q_{2|h|^\beta}}(z_k)}|f|\,dz\right).
    \end{align}
    Thus, using \eqref{ineq1.xreg}, \eqref{ineq2.xreg} and \eqref{ineq4.xreg}, we get
    \begin{align*}
        \int_{Q_{3/4}}|\delta_h^xf|\,dz&\leq \sum_{k\in\mathcal{K}}\int_{Q_{|h|^\beta}(z_k)}|\delta_h^xf|\,dz\\
        &\leq c\sum_{k\in\mathcal{K}}\left(|h|^{2\beta}|\mu|(Q_{4|h|^\beta}(z_k))+|h|^{\frac{\alpha_0}3(1-\beta)}\int_{{Q_{2|h|^\beta}}(z_k)}|f|\,dz\right)\\
        &\leq c|h|^{\frac{\alpha_0}{6}}\left(\|f\|_{L^1(Q_1)}+|\mu|(Q_1)\right).
    \end{align*}
    By the embedding as in \cite[Lemma 2.2]{DieKimLeeNow24p}, we deduce
    \begin{align*}
        \int_{Q_{1/2}}\frac{|f(t,x,v)-f(t,y,v)|}{|x-y|^{n+\alpha_0/8}}\,dx\,dt\,dv\leq c\left(\|f\|_{L^1(Q_1)}+|\mu|(Q_1)\right)
    \end{align*}
    for some constant $c=c(n,\Lambda)$. By taking $\gamma=\alpha_0/8$, we have the desired estimate.
\end{proof}

Combining the previous two results with an interpolation argument, we are now able to deduce first-order comparison estimates.
\begin{lemma}\label{lem.comp2}
    For $r>0$ and $z_0=(t_0,x_0,v_0) \in \mathbb{R}^{2n+1}$, let $\mu \in L^2(Q_{2r}(z_0))$ and let $f\in H^1_{\mathrm{kin}}(Q_{2r}(z_0))$ be a weak solution to 
    \begin{equation*}
        \partial_t f+v\cdot \nabla_x f-\divergence_v(a(t,x,v,\nabla_v f))=\mu\quad\text{in }Q_{2r}(z_0).
    \end{equation*}
    Then there is a unique weak solution $w\in \mathcal{W}(\mathcal{R}Q_r(z_0))$ to \eqref{eq.ivp}
such that 
\begin{equation*}
    \dashint_{Q_{r/2}(z_0)}|\nabla_v (f-w)|\,dz\leq \frac{c |\mu|(Q_{2r}(z_0))}{r^{4n+1}}
\end{equation*}
with $c=c(n,\Lambda)$.
\end{lemma}

\begin{proof}
By Lemma \ref{lem.scale}, we may assume $Q_{2r}(z_0)=Q_2$ and $|\mu|(Q_2)=1$. Let us fix $\xi>1$. As in the proof of Lemma \ref{lem.comp1}, by testing \eqref{eq2.comp1} with $\Phi_j(f-w)$, where the function $\Phi_j$ is defined in \eqref{defn.phij}, we obtain
\begin{align*}
    \int_{\mathcal{R}Q_1}\left[(a(t,x,v,\nabla_v f)-a(t,x,v,\nabla_v w))\cdot\nabla_v(\Phi_k(f-w))\right]\,dz \leq\int_{\mathcal{R}Q_1}\mu\Phi_j(f-w)\,dz.
\end{align*}
Therefore, we have 
\begin{align*}
    \int_{A_j}|\nabla_v(f-w)|^2\,dz
    \leq  c\int_{\mathcal{R}Q_1}\left[(a(t,x,v,\nabla_v f)-a(t,x,v,\nabla_v w))\cdot\nabla_v(\Phi_j(f-w))\right]\,dz \leq c
\end{align*}
for some constant $c=c(\Lambda)$, where the set $A_j$ is defined as 
\begin{align*}
    A_j:=\{z\in \mathcal{R}Q_1\,:\, j<|f(z)-w(z)|\leq j+1\}.
\end{align*}
Here we have used the fact that $\Phi_j\leq 1$ and $|\mu|(Q_2)=1$.
We now follow the computations as in \cite[Lemma 4.1]{DuzMin11a} to see that
\begin{equation}\label{ineq.comp2}
    \int_{\mathcal{R}Q_1}\frac{|\nabla_v(f-w)|^2}{(1+|f-w|)^\xi}\,dz\leq \sum_{j\geq0}\int_{A_j}\frac{|\nabla_v(f-w)|^2}{(1+j)^\xi}\,dz\leq \sum_{j\geq0}\frac{c}{(1+j)^\xi}\leq c
\end{equation}
for some constant $c=c(n,\Lambda,\xi)$.
    By the Sobolev embedding as in \cite[Chapter 1, Proposition 3.1]{Die93}, we first observe
    \begin{equation}\label{ineq001.comp2}
    \begin{aligned}
        \|(f-w)(\cdot,x,\cdot)\|_{L^{\frac{n+1}n}(I_{1/2}\times B^v_{1/2})}&\leq c\|(f-w)(\cdot,x,\cdot)\|^{\frac{n}{n+1}}_{L^1(I_{1/2};W^{1,1}(B^v_{1/2}))}\\
        &\quad\times\|(f-w)(\cdot,x,\cdot)\|^{\frac1{n+1}}_{L^\infty(I_{1/2};L^1(B^v_{1/2}))},
    \end{aligned}
    \end{equation}
    where we denote by $B^v_{1/2}$ the standard Euclidean ball in $\bbR^n$ with radius $1/2$ corresponding to the $v$-direction.
     We now integrate both sides of \eqref{ineq001.comp2} with respect to the $x$-variable and then use H\"older's inequality to obtain 
\begin{equation}\label{ineq01.comp2}
     \begin{aligned}
         &\int_{B^x_{1/2}}\|(f-w)(\cdot,x,\cdot)\|_{L^{\frac{n+1}n}(I_{1/2}\times B^{v}_{1/2})}\,dx\\
         &\leq c\|f-w\|_{L^1(I_1\times B_{1/2}^x;W^{1,1}(B^v_{1/2}))}^{\frac{n}{n+1}}\sup_{t\in I_{1/2}}\|(f-w)(t,\cdot)\|^{\frac1{n+1}}_{L^1(B^x_{1/2}\times B^v_{1/2})}.
     \end{aligned}
     \end{equation}

     We next use standard interpolation arguments and H\"older's inequality to see that
    \begin{equation}\label{ineq1.comp2}
    \begin{aligned}
        &\int_{B^x_{1/2}}\|(f-w)(\cdot,x,\cdot)\|^q_{L^q(I_{1/2}\times B^v_{1/2})}\,dx\\
        &\leq \int_{B^x_{1/2}}\|(f-w)(\cdot,x,\cdot)\|^{\beta q}_{L^{\frac{n+1}n}(I_{1/2}\times B^v_{1/2})}\|(f-w)(\cdot,x,\cdot)\|^{(1-\beta) q}_{L^{1}(I_{1/2}\times B^v_{1/2})}\,dx\\
        &\leq \left(\int_{B^x_{1/2}}\|(f-w)(\cdot,x,\cdot)\|_{L^{\frac{n+1}n}(I_{1/2}\times B^v_{1/2})}\,dx\right)^{\beta q}\\
        &\quad\times\left(\int_{B^x_{1/2}}\|(f-w)(\cdot,x,\cdot)\|^{\frac{n}{n-\gamma}}_{L^{1}(I_{1/2}\times B^v_{1/2})}\,dx\right)^{{1-\beta q}},
    \end{aligned}
    \end{equation}
    where the constant $\gamma=\gamma(n,\Lambda)$ is determined in Lemma \ref{lem.xreg} and
    \begin{equation}\label{betaq.comp2}
        q\coloneqq \frac{n(\gamma+1)+\gamma}{n(\gamma+1)}\quad\text{and}\quad \beta\coloneqq \frac{\gamma(n+1)}{n(\gamma+1)+\gamma}.
    \end{equation}
    Here we have used the fact that $\beta q<1$.
    We first observe that the fractional Sobolev inequality as in \cite[Theorem 6.10]{DinPalVal12} implies
    \begin{align}\label{ineq2.comp2}
        \|(f-w)(t,\cdot,v)\|_{L^{\frac{n}{n-\gamma}}(B^x_{1/2})}\leq c\|(f-w)(t,\cdot,v)\|_{|W^{\gamma,1}(B^x_{1/2})},
    \end{align}
    where $c=c(n,\Lambda)$.
    By \eqref{betaq.comp2}, Minkowski's integral inequality and \eqref{ineq2.comp2}, we obtain
    \begin{equation}\label{ineq3.comp2}
    \begin{aligned}
        &\left(\int_{B^x_{1/2}}\|(f-w)(\cdot,x,\cdot)\|^{\frac{n}{n-\gamma}}_{L^{1}(I_{1/2}\times B^v_{1/2})}\,dx\right)^{{1-\beta q}}\\
        &=\left(\int_{B^x_{1/2}}\|(f-w)(\cdot,x,\cdot)\|^{\frac{n}{n-\gamma}}_{L^{1}(I_{1/2}\times B^v_{1/2})}\,dx\right)^{{\frac{q(1-\beta)(n-\gamma)}{n}}}\\
        &\leq \left(\int_{I_{1/2}\times B^v_{1/2}}\|(f-w)(t,\cdot,v)\|_{L^{\frac{n}{n-\gamma}}(B^x_{1/2})}\,dt\,dv\right)^{q(1-\beta)}\\
        &\leq c\left(\int_{I_{1/2}\times B^v_{1/2}}\|(f-w)(t,\cdot,v)\|_{W^{\gamma,1}(B^x_{1/2})}\,dt\,dv\right)^{q(1-\beta)},
    \end{aligned}
    \end{equation}
    where $c=c(n,\Lambda)$.
    Combining the estimates \eqref{ineq01.comp2}, \eqref{ineq1.comp2} and \eqref{ineq3.comp2} yields 
    \begin{align*}
        &\left(\int_{B^x_{1/2}}\|(f-w)(\cdot,x,\cdot)\|^q_{L^q(I_{1/2}\times B^v_{1/2})}\,dx\right)^{\frac1q}\\
        &\leq c\left(\|f-w\|_{L^1(I_{1/2}\times B^x_{1/2};W^{1,1}(B^v_{1/2}))}^{\frac{n}{n+1}}\sup_{t\in I_{1/2}}\|(f-w)(t,\cdot)\|^{\frac1{n+1}}_{L^1(B^x_{1/2}\times B^v_{1/2})}\right)^{\beta}\\
        &\quad\times \left(\int_{I_{1/2}\times B^v_{1/2}}\|(f-w)(t,\cdot,v)\|_{W^{\gamma,1}(B^x_{1/2})}\,dt\,dv\right)^{1-\beta}.
    \end{align*}
To estimate further, we note that $f-w$ is a weak solution to
\begin{equation*}
\partial_t(f-w)+v\cdot\nabla_x(f-w)-\divergence_v(A(t,x,v)\nabla_v(f-w))=\mu\quad\text{in }Q_1,
\end{equation*}
    where 
    \begin{equation*}
    A(t,x,v)=\int_{0}^1a(t,x,v,s\nabla_vf+(1-s)\nabla_vw)\,ds
    \end{equation*}
    satisfies \eqref{unifell}. Therefore, we now use Lemma \ref{lem.xreg} with $f$ replaced by $f-w$ to see that 
    \begin{align}\label{ineq33.comp2}
    [f-w]_{W^{\gamma,1}(Q_{1/2})}\leq c\left(\|f-w\|_{L^1(Q_1))}+\|\mu\|_{L^1(Q_1)}\right)
    \end{align}
    for some constant $c=c(n,\Lambda)$. We now use \eqref{ineq33.comp2} and Lemma \ref{lem.comp1} to see that
    \begin{align*}
        &\left(\int_{B^x_{1/2}}\|(f-w)(\cdot,x,\cdot)\|^q_{L^q(I_{1/2}\times B_{1/2})}\,dx\right)^{\frac1q}\\
        &\leq c\left(\left(\int_{Q_{1/2}}|\nabla_v(f-w)|\,dz\right)^{\frac{n}{n+1}}|\mu|(Q_2)^{\frac1{n+1}}+|\mu|(Q_2)\right)^{\beta}\left(|\mu|(Q_{2})\right)^{1-\beta}
    \end{align*}
    for some constant $c=c(n,\Lambda)$.
    In view of H\"older's inequality and \eqref{ineq.comp2} with $\xi=q$, we obtain that
    \begin{align*}
        \|\nabla_v(f-w)\|_{L^1(Q_{1/2})}&\leq \left(\int_{Q_{1/2}}\frac{|\nabla_v(f-w)|^2}{(1+|f-w|)^{q}}\,dz\right)^{\frac12}\left(\int_{Q_{1/2}}(1+|f-w|)^{q}\,dz\right)^{\frac1{2}}\\
        &\leq c+\left(\int_{Q_{1/2}}\abs{f-w}^{q}\,dz\right)^{\frac1{2}}.
    \end{align*}
     Therefore, we have 
    \begin{align*}
        &\|\nabla_v(f-w)\|_{L^1(Q_{1/2})}\\
        &\leq c+c\left(|\mu|(Q_2)\right)^{1-\frac\beta2}\left(\left(\int_{Q_{1/2}}|\nabla_v(f-w)|\,dz\right)^{\frac{n}{n+1}}|\mu|(Q_2)^{\frac1{n+1}}+|\mu|(Q_2)\right)^{\frac\beta2}\\
        &\leq c+c|\mu|(Q_2)^{1-\frac{\beta}2+\frac{\beta}{2(n+1)}}\|\nabla_v(f-w)\|_{L^1(Q_{1/2})}^{\frac{\beta n}{2(n+1)}}+c|\mu|(Q_2).
    \end{align*}
    By Young's inequality along with the fact that $|\mu|(Q_2)=1$, we have 
    \begin{align*}
        \|\nabla_v(f-w)\|_{L^1(Q_{1/2})}\leq \frac12 \|\nabla_v(f-w)\|_{L^1(Q_{1/2})}+c
    \end{align*}
    for some constant $c=c(n,\Lambda)$. This completes the proof.
\end{proof}

We next prove another comparison estimate at the gradient level that involves freezing the coefficient.
\begin{lemma}\label{lem.comp3}
    For $r>0$ and $z_0=(t_0,x_0,v_0) \in \mathbb{R}^{2n+1}$, let  $w\in \mathcal{W}(\mathcal{R}Q_r(z_0))$ be the weak solution to \eqref{eq.ivp} and let $g\in \mathcal{W}(\mathcal{R}Q_{r/4}(z_0))$ be the weak solution to
\begin{equation}\label{eq.comp3}
\left\{
\begin{alignedat}{3}
\partial_t g+v\cdot \nabla_xg-\divergence_v(a(t,x_0,v_0,\nabla_v g))&= 0&&\qquad \mbox{in  $\mathcal{R}Q_{r/4}(z_0)$}, \\
g&=w&&\qquad  \mbox{on $\partial_{\mathcal{K}}\mathcal{R}Q_{r/4}(z_0)$}.
\end{alignedat} \right.
\end{equation}
Assume that there is a non-decreasing function ${\pmb{\omega}}:\bbR^+\to\bbR^+$ with ${\pmb{\omega}}(0)=0$ such that
    \begin{align}\label{defn.omega.comp}
    \sup_{(t,x,v),(t,y,u)\in Q_{2r}(z_0)}|a(t_,x,v,\xi)-a(t,y,u,\xi)|\leq {\pmb{\omega}}(\max\{|x-y|^{\frac13},|v-w|\})|\xi|.
\end{align}
Then we have
\begin{equation*}
    \dashint_{\mathcal{R}Q_{r/4}(z_0)}|\nabla_v(g-w)|^2\,dz\leq c\,{\pmb{\omega}}(r/2)^2\left(\dashint_{\mathcal{R}Q_{r/2}(z_0)}|\nabla_vw|\,dz\right)^2
\end{equation*}
for some constant $c=c(n,\Lambda)$.

\end{lemma}
\begin{proof}
We may assume $r=1$ and $z_0=0$. Since $g-w\in \mathcal{W}_0(\mathcal{R}Q_{1/4})$, we have 
\begin{align*}
    J_1+J_2&\coloneqq\int_{W_{1/4,0}(0,0)}\left<(\partial_t+v\cdot\nabla_x)(g-w),g-w\right>\,dt\,dx\\
    &\quad+\int_{\mathcal{R}Q_{1/4}}(a(t,0,0,\nabla_vg)-a(t,0,0,\nabla_vw))\cdot(\nabla_vg-\nabla_vw)\,dz\\
    &= \int_{\mathcal{R}Q_{1/4}}(a(t,x,v,\nabla_vw)-a(t,0,0,\nabla_vw))\cdot(\nabla_vg-\nabla_vw)\,dz\eqqcolon J_3,
\end{align*}
where $W_{1/4,0}(0,0)$ is defined in \eqref{defn.wrv0}.
Using standard approximation arguments as in \eqref{ineq0.comp1}, we observe that $J_1\geq0$.
In addition, we note from \eqref{ass.coef} that
\begin{align*}
    J_2\geq \frac1c\int_{\mathcal{R}Q_{1/4}}|\nabla_v(g-w)|^2\,dz
\end{align*}
for some constant $c=c(\Lambda)$. In light of Young's inequality, we next estimate $J_3$ as
\begin{align*}
    J_3\leq c\,{\pmb{\omega}}(1/2)^2\int_{\mathcal{R}Q_{1/4}}|\nabla_vw|^2+\frac{1}{2c}\int_{\mathcal{R}Q_{1/4}}|\nabla_v(g-w)|^2\,dz.
\end{align*}
Therefore, combining the estimates $J_1$, $J_2$, and $J_3$ leads to
\begin{equation}\label{ineq1.comp3}
    \dashint_{\mathcal{R}Q_{1/4}}|\nabla_v(g-w)|^2\,dz\leq c\,{\pmb{\omega}}(1/2)^2\dashint_{\mathcal{R}Q_{1/4}}|\nabla_vw|^2\,dz.
\end{equation}
Applying Lemma \ref{lem.self} with $r=1/2$ and $z_0=0$ into the right-hand side of \eqref{ineq1.comp3} yields the desired estimate.
\end{proof}
\section{Decay estimates and proofs of main results} \label{sec5}
In this section, we prove decay estimates of the gradient of the solution to \eqref{eq.main} and derive several pointwise gradient estimates in terms of the data.
\subsection{Nondivergence-type data}
In this subsection, we provide gradient estimates of the solution to \eqref{eq.main} with nondivergence-type data.
First, we establish the following excess decay estimate. 
\begin{lemma}[Excess decay - nondivergence data]\label{lem.decay}
    For $r>0$ and $z_0=(t_0,x_0,v_0) \in \mathbb{R}^{2n+1}$, let $\mu \in L^2(Q_{r}(z_0))$ and let $f\in H^1_{\mathrm{kin}}(Q_{r}(z_0))$ be a weak solution to 
    \begin{equation*}
        \partial_t f+v\cdot\nabla_x f-\divergence_v(a(t,x,v,\nabla_vf))=\mu\quad\text{in }Q_{r}(z_0).
    \end{equation*}
    Assume $a(\cdot)$ satisfies \eqref{defn.omega.comp} with $2r$ replaced by $r$.
    Then for any $\rho\in(0,1]$, we have 
    \begin{equation}\label{goal.decay}
    \begin{aligned}
        E(\nabla_vf;Q_{\rho r}(z_0))&\leq c\rho^{\alpha}E(\nabla_vf;Q_r(z_0))+c\rho^{-(4n+2)}{\pmb{\omega}}(r)\dashint_{Q_r(z_0)}|\nabla_vf|\,dz\\
        &\quad+c\rho^{-(4n+2)}\frac{|\mu|(Q_r(z_0))}{ r^{4n+1}}
    \end{aligned}
    \end{equation}
   for some constants $c=c(n,\Lambda)$ and $\alpha=\alpha(n,\Lambda)\in(0,1)$, where the constant $\alpha$ is determined in Theorem \ref{thm.hol}.
\end{lemma}
\begin{proof}
By Lemma \ref{lem.scale}, we may assume $Q_r(z_0)=Q_1$.
It suffices to consider the case that $\rho\in(0,2^{-10}]$, as \eqref{goal.decay} directly holds by taking a suitable constant $c=c(n,\Lambda)$, when $\rho\geq 2^{-10}$. Let $w\in \mathcal{W}(\mathcal{R}Q_{1/2})$ and $g\in \mathcal{W}(\mathcal{R}Q_{1/8})$ be weak solutions to \eqref{eq.ivp} and \eqref{eq.comp3} with $r=1/2$ and $z_0=0$, respectively.
Then we have
\begin{align}\label{ineq0.decay}
    E(\nabla_v f;Q_{\rho })\leq E(\nabla_v g;Q_{\rho })+E(\nabla_v(f-g);Q_{\rho })\eqqcolon J_1+J_2.
\end{align}
In light of Theorem \ref{thm.hol}, we obtain
\begin{equation}\label{ineq10.decay}
\begin{aligned}
    J_1&\leq c\rho^\alpha E(\nabla_v g;Q_{1/8})\\
    &\leq c\rho^\alpha \left[E(\nabla_v (g-w);Q_{1/8})+E(\nabla_v (w-f);Q_{1/8})+E(\nabla_v f;Q_{1/8})\right]\\
    &\leq c\rho^{\alpha}\left[{\pmb{\omega}}(1/4)\|\nabla_vw\|_{L^1(\mathcal{R}Q_{1/4})}+|\mu|(Q_{1/4})+E(\nabla_v f;Q_{1/4})\right],
\end{aligned}
\end{equation}
where we have used Lemmas \ref{lem.comp3} and \ref{lem.comp2} for the last inequality. We further estimate $J_1$ as 
\begin{equation}\label{ineq1.decay}
\begin{aligned}
    J_1&\leq c\rho^{\alpha}\left[{\pmb{\omega}}(1)\left(\|\nabla_v(w-f)\|_{L^1(\mathcal{R}Q_{1/4})}+\|\nabla_v f\|_{L^1(\mathcal{R}Q_{1/4})}\right)+|\mu|(Q_{1})\right]\\
    &\quad+c\rho^\alpha E(\nabla_v f;Q_{1/4})\\
    &\leq c\left[\rho^\alpha E(\nabla_v f;Q_1)+{\pmb{\omega}}(1)\|\nabla_v f\|_{L^1(Q_1)}+|\mu|(Q_1)\right]
\end{aligned}
\end{equation}
for some constant $c=c(n,\Lambda)$.
Similarly, we estimate $J_2$ as
\begin{equation}\label{ineq2.decay}
\begin{aligned}
    J_2&\leq E(\nabla_v(f-w);Q_\rho)+E(\nabla_v(g-w);Q_\rho)\\
    &\leq c\rho^{-(4n+2)}\left[|\mu|(Q_1)+{\pmb{\omega}}(1)\|\nabla_vw\|_{L^1(\mathcal{R}Q_{1/4})}\right]\\
    &\leq c\rho^{-(4n+2)}\left[|\mu|(Q_1)+{\pmb{\omega}}(1)\|\nabla_vf\|_{L^1(Q_{1})}\right],
\end{aligned}
\end{equation}
where $c=c(n,\Lambda)$.
Plugging the estimates $J_1$ and $J_2$ into \eqref{ineq0.decay} yields the desired estimate.
\end{proof}

Using Lemma \ref{lem.decay}, we now prove gradient estimates in terms of kinetic Riesz potentials.
\begin{lemma}\label{lem.point}
    For $R>0$ and $z_0=(t_0,x_0,v_0) \in \mathbb{R}^{2n+1}$, let $\mu \in L^2(Q_{R}(z_0))$ and let $f\in H^1_{\mathrm{kin}}(Q_{R}(z_0))$ be a weak solution to 
    \begin{equation*}
        \partial_t f+v\cdot\nabla_x f-\divergence_v(a(t,x,v,\nabla_vf))=\mu\quad\text{in }Q_{R}(z_0),
    \end{equation*}
    where $a(\cdot)$ satisfies \eqref{defn.omega.comp} with $2r$ replaced by $R$ and
    \begin{equation*}
        \int_{0}^{1}\frac{{\pmb{\omega}}(\rho)}{\rho}\,d\rho<\infty.
    \end{equation*}
    Then there is a positive integer $m=m(n,\Lambda,{\pmb{\omega}})$ such that for any $i\geq0$,
    \begin{equation}\label{est.point}
        \dashint_{Q_{2^{-im}R}(z_0)}|\nabla_v f|\,dz\leq c\dashint_{Q_R(z_0)}|\nabla_v f|\,dz+cI^{|\mu|}_{1}(z_0,R)
    \end{equation}
    holds, where $c=c(n,\Lambda,{\pmb{\omega}})$.
\end{lemma}
\begin{proof}
Let us fix a positive integer $m\geq1$ which will be determined later. Then we have that for any $r\in(0,R]$,
 \begin{align*}
        E(\nabla_vf;Q_{2^{-(k+1)m}r}(z_0))&\leq c2^{-m\alpha}E(\nabla_vf;Q_{2^{-km}r}(z_0))\\
        &\quad+c2^{-m(4n+2)}{\pmb{\omega}}\left(2^{-km}r\right)\dashint_{Q_{2^{-km}r}(z_0)}|\nabla_vf|\,dz\\
        &\quad+c2^{-m(4n+2)}\frac{|\mu|(Q_{2^{-km}r}(z_0))}{ (2^{-km}r)^{4n+1}}
    \end{align*}
    holds, where $c=c(n,\Lambda)$.
We now choose $m=m(n,\Lambda)\geq1$ so that 
\begin{equation}\label{choim.point}
\begin{aligned}
        E(\nabla_vf;Q_{2^{-(k+1)m}r}(z_0))&\leq {E(\nabla_vf;Q_{2^{-km}r}(z_0))}/4\\
        &\quad+c\,{\pmb{\omega}}\left(2^{-km}r\right)\dashint_{Q_{2^{-km}r}(z_0)}|\nabla_vf|\,dz\\
        &\quad+c\frac{|\mu|(Q_{2^{-km}r}(z_0))}{ (2^{-km}r)^{4n+1}},
    \end{aligned}
    \end{equation}
    where $c=c(n,\Lambda)$. By summing $k=1,2,\ldots i$, we have 
    \begin{equation}\label{ineq00.point}
    \begin{aligned}
        \sum_{k=1}^{i}E(\nabla_vf;Q_{2^{-(k+1)m}r}(z_0))&\leq \frac{1}{4}\sum_{k=1}^{i}{E(\nabla_vf;Q_{2^{-km}r}(z_0))}\\
        &\quad+c\sum_{k=1}^{i}{\pmb{\omega}}\left(2^{-km}r\right)\dashint_{Q_{2^{-km}r}(z_0)}|\nabla_vf|\,dz\\
        &\quad +c\sum_{k=1}^{i}\frac{|\mu|(Q_{2^{-km}r}(z_0))}{ (2^{-km}r)^{4n+1}}\eqqcolon\sum_{i=1}^3J_i.
    \end{aligned}
    \end{equation}

    First, we note
    \begin{align*}
        J_2&\leq c\sum_{k=0}^{i}{\pmb{\omega}}\left(2^{-km}r\right)\left[\sum_{l=1}^{k}E(\nabla_v f;Q_{2^{-lm}r}(z_0))+|(\nabla_vf)_{Q_{r}(z_0)}|\right]\\
        &\leq c\sum_{k=1}^{i}{\pmb{\omega}}\left(2^{-km}r\right)\sum_{l=0}^{k}E(\nabla_v f;Q_{2^{-lm}r}(z_0))\\
        &\quad+c\sum_{k=1}^{i}{\pmb{\omega}}\left(2^{-km}r\right)|(\nabla_vf)_{Q_{r}(z_0)}|\eqqcolon J_{2,1}+J_{2,2}.
    \end{align*}
    By Fubini's theorem, we have 
    \begin{align*}
        J_{2,1}&\leq c\sum_{l=1}^{i}\sum_{k=l}^{i}{\pmb{\omega}}\left(2^{-km}r\right)E(\nabla_v f;Q_{2^{-lm}r}(z_0))+ c\sum_{k=1}^{i}{\pmb{\omega}}\left(2^{-km}r\right)E(\nabla_v f;Q_{r}(z_0))\\
        &\leq c\int_{0}^{r}\frac{{\pmb{\omega}}(\rho)}{\rho}\,d\rho\sum_{l=0}^{i}E(\nabla_v f;Q_{2^{-lm}r}(z_0))
    \end{align*}
    for some constant $c=c(n,\Lambda)$, where we have also used that
    \begin{align*}
       \sum_{k=l}^{i}{\pmb{\omega}}\left(2^{-km}r\right)\leq c\sum_{k=l}^{i}\int^{2^{-(k-1)m}r}_{2^{-km}r}\frac{{\pmb{\omega}}(\rho)}{\rho}\,d\rho \leq c\int_{0}^{r}\frac{{\pmb{\omega}}(\rho)}{\rho}\,d\rho.
    \end{align*}
    Similarly, we get
    \begin{align*}
        J_{2,2}\leq c\left(\int_{0}^{r}\frac{{\pmb{\omega}}(\rho)}{\rho}\,d\rho\right)|(\nabla_vf)_{Q_{r}(z_0)}|,
    \end{align*}
    which implies 
    \begin{align}\label{ineq001.point}
        J_2\leq c\int_{0}^{r}\frac{{\pmb{\omega}}(\rho)}{\rho}\,d\rho\left[\sum_{l=0}^{i}E(\nabla_v f;Q_{2^{-lm}r}(z_0))+|(\nabla_vf)_{Q_{r}(z_0)}|\right]
    \end{align}
    for some constant $c=c(n,\Lambda)$.
    Thus we now choose $\sigma=\sigma(n,\Lambda,{\pmb{\omega}})\in(0,1]$  sufficiently small so that 
    \begin{equation}\label{choir.point}
        \int_{0}^{r}\frac{\omega(\rho)}{\rho}\,d\rho\leq \frac{1}{4c},
    \end{equation}
    where $r\coloneqq\sigma R$. Thus, we have
    \begin{align*}
        J_2\leq \frac{1}{4}\sum_{l=0}^{i}E(\nabla_v f;Q_{2^{-lm}r}(z_0))+c|(\nabla_vf)_{Q_{r}(z_0)}|.
    \end{align*}
    On the other hand, we observe that
    \begin{align*}
        J_3\leq cI^{|\mu|}_1(z_0,r).
    \end{align*}

    Combining all the above estimates for $J_1,J_2$ and $J_3$ yields 
    \begin{equation*}
    \begin{aligned}
        \sum_{k=1}^{i}E(\nabla_vf;Q_{2^{-km}r}(z_0))&\leq \frac{1}{2}\sum_{k=1}^{i}{E(\nabla_vf;Q_{2^{-km}r}(z_0))}+c|(\nabla_vf)_{Q_{r}(z_0)}|+cI^{|\mu|}_1(z_0,r),
    \end{aligned}
    \end{equation*}
    where $c=c(n,\Lambda,{\pmb{\omega}})$. Therefore, we have 
    \begin{align}\label{ineq1.point}
        \sum_{k=1}^iE(\nabla_vf;Q_{2^{-km}r}(z_0))\leq c\dashint_{Q_r(z_0)}|\nabla_v f|\,dz+cI^{|\mu|}_1(z_0,r),
    \end{align}
    where $c=c(n,\Lambda,{\pmb{\omega}})$. We are now ready to prove \eqref{est.point}. To do this, we observe that for any positive integer $i\geq1$, either
    \begin{equation}\label{fir.point}
        r=\sigma R< 2^{-im}R
    \end{equation}
    or
    \begin{equation}\label{sec.point}
        2^{-(i_0+1)m}r<2^{-im }R\leq 2^{-i_0m}r
    \end{equation}
    holds for some nonnegative integer $i_0\geq0$.
    If \eqref{fir.point} holds, then 
    \begin{equation*}
        \dashint_{Q_{2^{-im}R}(z_0)}|\nabla_v f|\,dz\leq  c\dashint_{Q_{R}(z_0)}|\nabla_v f|\,dz,
    \end{equation*}
    where $c=c(n,\Lambda,{\pmb{\omega}})$, as the constant $r$ determined in \eqref{choir.point} depends only on $n,\Lambda$ and ${\pmb{\omega}}$. If \eqref{sec.point} holds, then
    \begin{align*}
        \dashint_{Q_{2^{-im}R}(z_0)}|\nabla_v f|\,dz&\leq c\dashint_{Q_{2^{-i_0m}r}(z_0)}|\nabla_v f|\,dz\\
        &\leq c\left(\sum_{k=1}^{i_0}E(\nabla_vf;Q_{2^{-km}r}(z_0))+|(\nabla_vf)_{Q_r(z_0)}|\right)\\
        &\leq c\left(\dashint_{Q_r(z_0)}|\nabla_v f|\,dz+I^{|\mu|}_1(z_0,r)\right)\\
        &\leq c\left(\dashint_{Q_R(z_0)}|\nabla_v f|\,dz+I^{|\mu|}_1(z_0,R)\right)
    \end{align*}
     for some constant $c=c(n,\Lambda,{\pmb{\omega}})$. This completes the proof.
\end{proof}

By using Lemma \ref{lem.point}, we now improve the estimate \eqref{est.point} by imposing a H\"older assumption on the nonlinearity $a$.
\begin{lemma}\label{lem.pointh}
For $R>0$ and $z_0=(t_0,x_0,v_0) \in \mathbb{R}^{2n+1}$, let $\mu \in L^2(Q_{R}(z_0))$ and let $f\in H^1_{\mathrm{kin}}(Q_R(z_0))$ be a weak solution to 
    \begin{equation*}
        \partial_t f+v\cdot\nabla_x f-\divergence_v(a(t,x,v,\nabla_vf))=\mu\quad\text{in }Q_R(z_0),
    \end{equation*}
    where $a$ satisfies \eqref{defn.omega.comp} with $2r$ replaced by $R$ and
    \begin{equation}\label{ass.pointh}
        {\pmb{\omega}}(\rho)\leq \rho^\beta
    \end{equation}
    for some $\beta\in(0,\alpha)$, where the constant $\alpha$ is determined in Theorem \ref{thm.hol}. Then we have
    \begin{equation*}
        M^{\#}_{R,\beta}(\nabla_v f)(z_0)\leq c\left(\dashint_{Q_R(z_0)}|\nabla_v f|\,dz+cI^{|\mu|}_1(z_0,R)+M_{1-\beta,R}(\mu)(z_0)\right)
    \end{equation*}
    holds, where $c=c(n,\Lambda,\beta)$.
\end{lemma}
\begin{proof}
    By \eqref{ass.pointh}, we get \eqref{est.point}. In addition, by Lemma \ref{lem.decay}, we have 
     \begin{equation}\label{ineq0.pointh}
     \begin{aligned}
        E(\nabla_vf;Q_{\rho r}(z_0))&\leq c\rho^{\alpha}E(\nabla_vf;Q_r(z_0))+cr^\beta \rho^{-(4n+2)} \dashint_{Q_r(z_0)}|\nabla_vf|\,dz\\
        &\quad+c\rho^{-(4n+2)}\frac{|\mu|(Q_r(z_0))}{ r^{4n+1}},
    \end{aligned}
    \end{equation}
    for some constant $c=c(n,\Lambda)$, where $r\leq R/2$ and $\rho\in(0,1/2]$. By \eqref{est.point}, we deduce
    \begin{align*}
        \dashint_{Q_r(z_0)}|\nabla_vf|\,dz\leq c\left(\dashint_{Q_R(z_0)}|\nabla_v f|\,dz+I^{|\mu|}_1(z_0,R)\right),
    \end{align*}
    where $c=c(n,\Lambda)$.

    We now plug this into the second term in the right-hand side of \eqref{ineq0.pointh} to get that
    \begin{equation}\label{ineq1.pointh}
     \begin{aligned}
        E(\nabla_vf;Q_{\rho r}(z_0))&\leq c\rho^{\alpha}E(\nabla_vf;Q_r(z_0))\\
        &\quad+cr^\beta \rho^{-(4n+2)} \left(\dashint_{Q_R(z_0)}|\nabla_v f|\,dz+I^{|\mu|}_1(z_0,R)\right)\\
        &\quad+c\rho^{-(4n+2)}\frac{|\mu|(Q_r(z_0))}{ r^{4n+1}},
    \end{aligned}
    \end{equation}
    where $c=c(n,\Lambda)$. We now divide both sides given in \eqref{ineq1.pointh} by $(\rho r)^{\beta}$ and choose $\rho=\rho(n,\Lambda,\beta)$ sufficiently small such that $\rho^{\alpha-\beta}\leq \frac{1}{4}$.
    Therefore, we obtain
    \begin{align*}
        \frac{E(\nabla_vf;Q_{\rho r}(z_0))}{(\rho r)^\beta}&\leq \frac14\frac{E(\nabla_vf;Q_r(z_0))}{r^\beta}+c\left(\dashint_{Q_R(z_0)}|\nabla_v f|\,dz+I^{|\mu|}_1(z_0,R)\right)\\
        &\quad+cr^{1-\beta}\dashint_{Q_r(z_0)}|\mu|\,dz
    \end{align*}
    for some constant $c=c(n,\Lambda,\beta)$. 
    By considering \eqref{defn.smaxi} and \eqref{defn.maxi}, we derive 
    \begin{align*}
        M^{\#}_{R,\beta}(\nabla_vf)(z_0)&\leq \frac14M^{\#}_{R,\beta}(\nabla_vf)(z_0)+c\left(\dashint_{Q_R(z_0)}|\nabla_v f|\,dz+I^{|\mu|}_1(z_0,R)\right)\\
        &\quad+cM_{1-\beta,R}(\mu)(z_0).
    \end{align*}
    This completes the proof.
\end{proof} 
We now prove Theorem \ref{thm.dini}, Corollary \ref{cor.dini} and Corollary \ref{cor.vmo}.
\begin{proof}[Proof of Theorem \ref{thm.dini}, Corollary \ref{cor.dini} and Corollary \ref{cor.vmo}.]
Using Lemma \ref{lem.point} and a generalized Lebesgue differentiation theorem given by \cite[Theorem 10.3]{ImbSil20} with $s=1$, we obtain the pointwise gradient potential estimates given in Theorem \ref{thm.dini}. 

We are now going to prove Corollary \ref{cor.vmo}. Suppose that \eqref{ass.cor.vmo} holds. In light of the first assumption in \eqref{ass.cor.vmo} and the estimate given in Theorem \ref{thm.dini}, we have $\nabla_v f\in L^\infty(Q_{3R/4}(z_0))$. We now prove
\begin{equation}\label{vmo.goal1}
\lim_{r\to0}\sup_{z_1\in Q_{3R/16}(z_0)}E(\nabla_v f;Q_r(z_1))=0.
\end{equation}
Let us fix $\epsilon>0$.
We note from Lemma \ref{lem.decay} that for any $z_1\in Q_{3R/16}(z_0)$,   $r\leq R/1000$ and $\rho\leq 1$, we have 
\begin{align*}
&E(\nabla_vf;Q_{\rho r}(z_1))\\
&\quad\leq c\rho^{\alpha}E(\nabla_vf;Q_r(z_1))+c\,\rho^{-(4n+2)}{\pmb{\omega}}(r)\dashint_{Q_r(z_1)}|\nabla_vf|\,dz+c\rho^{-(4n+2)}\frac{|\mu|(Q_r(z_1))}{ r^{4n+1}}.
    \end{align*}
    By Lemma \ref{lem.cov1}, we have $Q_{r}(z_1)\subset Q_{3R/4}(z_0)$. Thus, we further estimate
    \begin{align*}
        E(\nabla_vf;Q_{\rho r}(z_1))&\leq c\left(\left[\rho^{\alpha}+\rho^{-(4n+2)}{\pmb{\omega}}(r)\right]\|\nabla_v f\|_{L^\infty(Q_{3R/4}(z_0))}+\rho^{-(4n+2)}\frac{|\mu|(Q_r(z_0))}{ r^{4n+1}}\right).
    \end{align*}
    Next, we choose $\rho$ sufficiently small so that 
    \begin{align*}
        c\rho^\alpha\|\nabla_v f\|_{L^\infty(Q_{3R/4}(z_0))}\leq \epsilon/2.
    \end{align*}
    By \eqref{cond.dini} and the second assumption given in \eqref{ass.cor.vmo}, we next choose $r$ sufficiently small so that 
    \begin{align*}
        c\left(\rho^{-(4n+2)}{\pmb{\omega}}(r)\|\nabla_v f\|_{L^\infty(Q_{3R/4}(z_0))}+\rho^{-(4n+2)}\frac{|\mu|(Q_r(z_0))}{ r^{4n+1}}\right)\leq \epsilon/2.
    \end{align*}
    This implies that for any $\epsilon>0$, there are constants $r\in(0,R/1000]$ and $\rho\in(0,1]$ such that
    \begin{align*}
        \sup_{z_1\in Q_{3R/16}(z_0)}E(\nabla_vf;Q_{r\rho}(z_1))\leq \epsilon,
    \end{align*} which implies \eqref{vmo.goal1} in view of standard covering arguments. This completes the proof of Corollary \ref{cor.vmo}. 
    
    We are now ready to prove Corollary \ref{cor.dini}. Combining the estimates given in \eqref{ineq00.point} and \eqref{ineq001.point},
    we deduce
    \begin{align*}
        J\coloneqq|\nabla_v f(z_1)-(\nabla_v f)_{Q_r(z_1)}|&\leq \sum_{i=0}^{\infty}E(\nabla_v f;Q_{2^{-km}r}(z_1))\\
        &\leq cE(\nabla_vf;Q_{r}(z_1))\\
        &\quad+c\left(\int_{0}^r\frac{{\pmb{\omega}}(\rho)}{\rho}\,d\rho\right)\|\nabla_vf\|_{L^\infty(Q_{3R/4}(z_0))}\\
        &\quad+cI_1^{|\mu|}(z_1,r),
    \end{align*}
    where $c=c(n,\Lambda)$ and the positive integer $m=m(n,\Lambda)$ is determined in Lemma \ref{lem.point}. We note that \eqref{ass.cor.dini} implies \eqref{ass.cor.vmo}.
    Therefore, we now use \eqref{vmo.goal1} to see that 
    \begin{align*}
         \lim_{r\to0}\sup_{z_1\in Q_{3R/16}(z_0)}E(\nabla_vf;Q_{r}(z_1))=0.
    \end{align*}
    This together with \eqref{cond.dini} and \eqref{ass.cor.dini} implies that $J\to0$, as $r\to0$. Thus, $\nabla_vf$ is continuous. Furthermore, when $\mu\in L^{4n+2,1}(Q_{2R}(z_0))$, then $I^{|\mu|}_1(z,R/4)\in L^\infty(Q_{R/4}(z_0))$ which follows from \eqref{ineq2.riesz}. Therefore, there holds \eqref{ass.cor.dini} with $R$ replaced by $R/4$ and the continuity of $\nabla_vf$ in $W \times V$ follows. This completes the proof.
\end{proof}
We next provide the proof of Corollary \ref{cor.cal}.
\begin{proof}[Proof of Corollary \ref{cor.cal}.]
    Let us fix $q\in[2,4n+2)$. Then by Theorem \ref{thm.dini}, we have 
    \begin{align*}
        |\nabla_v f(z_1)|\leq c\left(\dashint_{Q_{R/1000}(z_1)}|\nabla_v f|\,dz+I^{|\mu|}_{1}(z_1,R/1000)\right),
    \end{align*}
    where $c=c(n,\Lambda)$ and $z_1\in Q_{R}(z_0)$. Therefore, using Lemma \ref{lem.riesz}, we have 
    \begin{align*}
        \left(\dashint_{Q_{R/1000}(z_1)}|\nabla_v f|^{\frac{q(4n+2)}{4n+2-q}}\,dz\right)^{\frac{4n+2-q}{q(4n+2)}}&\leq c\dashint_{Q_{R/10}(z_1)}|\nabla_vf|\,dz\\
        &\quad+c\,R\left(\dashint_{Q_{R/10}(z_1)}|\mu|^q\,dz\right)^{\frac1q},
    \end{align*}
    where $c=c(n,q)$.
    Standard covering arguments now yield \eqref{ineq1.cal}, completing the proof.
\end{proof}
We end this section with providing the proof of Theorem \ref{thm.holn}.
\begin{proof}[Proof of Theorem \ref{thm.holn}.]
By Lemma \ref{lem.pointh}, we get \eqref{ineq0.holn}. We now use Lemma \ref{lem.fmax} together with the fact that $\mu\in L^{\frac{4n+2}{1-\beta},\infty}(Q_{2R}(z_0))$ to see that $\|M_{1-\beta,R/2}(\mu)\|_{L^\infty(Q_{R/2}(z_0))}<\infty$. Since $L^{\frac{4n+2}{1-\beta},\infty}(Q_{2R}(z_0))\subset L^{{4n+2},1}(Q_{2R}(z_0))$, we have  $\|I^{|\mu|}_1(z_1,R/2)\|_{L^\infty(Q_{R/2}(z_0))}<\infty$, which follows from \eqref{ineq2.riesz}. Therefore, we have $\|M^{\#}_{\beta,R/2}(\nabla_vf)\|_{L^\infty(Q_{R/2}(z_0))}<\infty$, so that in light of \eqref{ineq2.ptmaxc}, we have $\nabla_vf\in C_{\mathrm{kin}}^{\beta}(Q_{R/16}(z_0))$. By standard covering arguments, we arrive at the desired result. This completes the proof.
\end{proof}
Before we prove Theorem \ref{cor.lin}, we first observe the following lemma.
\begin{lemma}\label{lem.higsolhol}
    For $R>0$ and $z_0=(t_0,x_0,v_0) \in \mathbb{R}^{2n+1}$, let $\mu \in L^2(Q_{2R}(z_0))$ and let $f\in H^1(Q_{2R}(z_0))$ be a weak solution to
     \begin{equation*}
        \partial_tf+v\cdot\nabla_x f=\divergence_v(a(t)\nabla_vf)+\mu\quad\text{in }Q_{2R}(z_0).
    \end{equation*}
    Then for any $\beta\in(0,1)$, we have 
    \begin{align*}
        R^\beta[f]_{C_{\mathrm{kin}}^\beta(Q_R(z_0))}\leq c\left(\dashint_{Q_{2R}(z_0)}|f-(f)_{Q_{2R}(z_0)}|\,dz+R^2\|\mu\|_{L^\infty(Q_{2R}(z_0))}\right)
    \end{align*}
    with $c=c(n,\Lambda)$.
\end{lemma}
\begin{proof}
By \eqref{est.bdddiv} in Lemma \ref{lem.bdddiv} below, we have 
    \begin{equation*}
       R^\beta[f]_{C_{\mathrm{kin}}^\beta(Q_R(z_0))}\leq c\left(\dashint_{Q_{3R/2}(z_0)}|\nabla_vf|\,dz+R^2\|\mu\|_{L^\infty(Q_{3R/2}(z_0))}\right),
    \end{equation*}
where $c=c(n,\Lambda,\beta)$.
    By the standard energy estimates given in \cite[Lemma 2.6]{GolImbMouVas19} and Lemma \ref{lem.dnmaff}, we further estimate
    \begin{equation*}
       R^\beta[f]_{C_{\mathrm{kin}}^\beta(Q_R(z_0))}\leq c\left(\dashint_{Q_{2R}(z_0)}|f-(f)_{Q_{2R}(z_0)}|\,dz+R^2\|\mu\|_{L^\infty(Q_{2R}(z_0))}\right)
    \end{equation*}
    for some constant $c=c(n,\Lambda,\beta)$, which completes the proof.
\end{proof}
We are now ready to prove Theorem \ref{cor.lin}.
\begin{proof}[Proof of Theorem \ref{cor.lin}.]
    We first prove that for any weak solution $f\in H^1(Q_{2R}(z_0))$ to 
    \begin{equation*}
        \partial_tf+v\cdot\nabla_x f=\divergence_v(a(t)\nabla_vf)\quad\text{in }Q_{2R}(z_0),
    \end{equation*}
    for any $\beta \in (0,1)$ we have
    \begin{equation}\label{res.lin}
        R^{\beta}[\nabla_v f]_{C_{\mathrm{kin}}^\beta(Q_R(z_0))}\leq c\dashint_{Q_{2R(z_0)}}|\nabla_vf-(\nabla_vf)_{Q_{2R(z_0)}}|\,dz,
    \end{equation}
	where $c=c(n,\Lambda,\beta)$. By Lemma \ref{lem.scale}, we may assume $R=1$ and $z_0=0$.
    
    We first note from Lemma \ref{lem.grax} and Lemma \ref{lem.spi} that
    \begin{equation}\label{ineq1.lin}
        \|\nabla_x f\|_{L^\infty(Q_{3/2})}\leq c\dashint_{Q_{7/4}}|f-l|\,dz\leq c\dashint_{Q_{2}}|\nabla_vf-(\nabla_vf)_{Q_2}|\,dz
    \end{equation}
    for some constant $c=c(n,\Lambda)$, where $l=(\nabla_vf)_{Q_2}\cdot v+(f)_{Q_2}$.
    We next note for any $h\in(0,1/10^{10}]$,
    \begin{equation*}
        \partial_t\delta_h^vf+v\cdot\nabla_x\delta_h^vf=\divergence_v(a(t)\nabla_v\delta_h^vf)-h\cdot\nabla_xf_h^v\quad\text{in }Q_{3/2}.
    \end{equation*}
 By Lemma \ref{lem.higsolhol}, for any $\beta\in(0,1)$, we get
 \begin{equation*}
     \left[\frac{\delta_h^vf}{|h|}\right]_{C_{\mathrm{kin}}^{\beta}(Q_{1})}\leq c\left(\dashint_{Q_{3/2}}\left|\frac{\delta_h^vf}{|h|}-\left(\frac{\delta_h^vf}{|h|}\right)_{Q_{3/2}}\right|\,dz+\|\nabla_xf\|_{L^\infty(Q_{3/2})}\right)
 \end{equation*}
 for some constant $c=c(n,\Lambda,\beta)$. We now take $h=\epsilon e_i$, where $e_i$ is the standard $i$-th direction unit vector ($i=1,2,\dots,n$) and take $h\to0$, in order to deduce that
\begin{equation}\label{ineq2.lin}
     [\nabla_vf]_{C_{\mathrm{kin}}^{\beta}(Q_{1})}\leq c\left(\dashint_{Q_{2}}|\nabla_vf-(\nabla_vf)_{Q_{2}}|\,dz+\|\nabla_xf\|_{L^\infty(Q_{3/2})}\right).
 \end{equation}
 Then \eqref{res.lin} follows by plugging \eqref{ineq1.lin} into the the second term in the right-hand side of \eqref {ineq1.lin}.
 In view of \eqref{res.lin}, we can now apply Theorem \ref{thm.holn} with $\alpha$ replaced by 1 to get the desired result.
\end{proof}

\subsection{Divergence-type data}
In this subsection, we prove pointwise estimates of the gradient of the solution to \eqref{eq.main} when the right-hand side is given by divergence-type data $\divergence_v G$ for some vector valued function $G\in L^2$. 

To handle divergence-type data, we will use a different excess functional in this subsection. For any $G:Q_r(z_0)\to\bbR^n$ with $G\in L^2(Q_r(z_0))$, we write
\begin{equation*}
     E_2(G;Q_{ r}(z_0))\coloneqq \left(\dashint_{Q_r(z_0)}|G-(G)_{Q_r(z_0)}|^2\,dz\right)^{\frac12}.
\end{equation*}
We now prove excess decay estimates when the right-hand side is in divergence form.
\begin{lemma}[Excess decay - divergence data]\label{lem.decay2}
    For $r>0$ and $z_0=(t_0,x_0,v_0) \in \mathbb{R}^{2n+1}$, let $G \in L^2(Q_r(z_0),\mathbb{R}^n)$ and let $f\in H^1_{\mathrm{kin}}(Q_r(z_0))$ be a weak solution to 
    \begin{equation*}
        \partial_t f+v\cdot\nabla_x f-\divergence_v(a(t,x,v,\nabla_vf))=-\divergence_v G\quad\text{in }Q_r(z_0),
    \end{equation*}
    where $G\in L^2(Q_r(z_0))$ and $a(\cdot)$ satisfies \eqref{defn.omega.comp} with $2r$ replaced by $r$.
    Then for any $\rho\in(0,1]$, we have 
    \begin{equation}\label{ineq0.decay2}
    \begin{aligned}
        E(\nabla_vf;Q_{\rho r}(z_0))&\leq c\rho^{\alpha}E(\nabla_vf;Q_r(z_0))+c\rho^{-(4n+2)}{\pmb{\omega}}(r)\dashint_{Q_r(z_0)}|\nabla_vf|\,dz\\
        &\quad+c\rho^{-(2n+1)}E_2(G;Q_r(z_0))
    \end{aligned}
    \end{equation}
    for some constants $c=c(n,\Lambda)$ and $\alpha=\alpha(n,\Lambda)\in(0,1)$, where the constant $\alpha$ is determined in Theorem \ref{thm.hol}.
\end{lemma}
\begin{proof}
By Lemma \ref{lem.scale}, we assume $Q_r(z_0)=Q_1$. When $\rho>2^{-10}$, the assertion follows trivially with a suitable constant $c=c(n,\Lambda)$. Thus we may assume $\rho\in(0,2^{-10}]$. Let $w\in \mathcal{W}(\mathcal{R}Q_{1/2})$ and $g\in \mathcal{W}(\mathcal{R}Q_{1/8})$ be weak solutions to \eqref{eq.ivp} and \eqref{eq.comp3} with $\mu=-\divergence_v G$, $r=1/2$ and $z_0=0$, respectively. Since $f-w\in \mathcal{W}_0(\mathcal{R}Q_{1/2})$, by testing the equation
\begin{align*}
    \partial_t(f-w)+v\cdot\nabla_x(f-w)-\divergence(a(t,x,v,\nabla_vf)-a(t,x,v,\nabla_vw))
    =-\divergence_v(G-(G)_{Q_{1/2}})
\end{align*}
with $f-w$, we deduce
\begin{align*}
    \dashint_{\mathcal{R}Q_{1/2}}|\nabla_v(f-w)|^2\,dz\leq c\dashint_{\mathcal{R}Q_{1/2}}|(G-(G)_{Q_{1/2}})\cdot(\nabla_v(f-w))|\,dz.
\end{align*}
By Young's inequality, we obtain
\begin{align}\label{ineq1.decay2}
    \dashint_{\mathcal{R}Q_{1/2}}|\nabla_v(f-w)|^2\,dz\leq c\dashint_{Q_1}|G-(G)_{Q_1}|^2\,dz
\end{align}
for some constant $c=c(n,\Lambda)$, where we have used the fact that $\mathcal{R}Q_{1/2}\subset Q_1$.
We next observe that
\begin{align*}
    E_2(\nabla_vf;Q_{\rho})\leq E_2(\nabla_vg;Q_\rho)+E_2(\nabla_v(f-g);Q_\rho)\coloneqq J_1+J_2.
\end{align*}
As in the estimates concerning $J_1$ given in \eqref{ineq10.decay} and \eqref{ineq1.decay} together with \eqref{ineq1.decay2}, we have 
\begin{align*}
    J_1\leq c\rho^\alpha E(\nabla_v g;Q_{1/8}) \leq c\left[\rho^\alpha E(\nabla_v f;Q_1)+{\pmb{\omega}}(1)\|\nabla_vf\|_{L^1(Q_1)}+E_2(G;Q_1)\right],
\end{align*}
where $c=c(n,\Lambda)$.
By \eqref{ineq2.decay} along with \eqref{ineq1.decay2}, we estimate $J_2$ as 
\begin{align*}
    J_2\leq c\left[\rho^{-(2n+1)}E_2(G;Q_1)+\rho^{-(4n+2)}{\pmb{\omega}}(1)\|\nabla_vf\|_{L^1(Q_1)}\right],
\end{align*}
where $c=c(n,\Lambda)$. By combining all the above estimates for $J_1$ and $J_2$, we arrive at the desired estimate.
\end{proof}
Using Lemma \ref{lem.decay2}, we now obtain the following pointwise estimates.
\begin{lemma}\label{lem.pointh2}
    For $R>0$ and $z_0=(t_0,x_0,v_0) \in \mathbb{R}^{2n+1}$, let $G \in L^2(Q_R(z_0),\mathbb{R}^n)$ and let $f\in H^1_{\mathrm{kin}}(Q_R(z_0))$ be a weak solution to 
    \begin{equation*}
        \partial_t f+v\cdot\nabla_x f-\divergence_v(a(t,x,v,\nabla_vf))=-\divergence_v G\quad\text{in }Q_R(z_0),
    \end{equation*}
    where $G\in L^2(Q_R(z_0))$ and $a(\cdot)$ satisfies \eqref{defn.omega.comp} with $2r$ replaced by $R$ and the function ${\pmb{\omega}}$ satisfies \eqref{ass.pointh} for some $\beta\in(0,\alpha)$, where the constant $\alpha$ is determined in Theorem \ref{thm.hol}. Then we have 
    \begin{align*}
        M^{\#}_{{\beta},R}(|\nabla_vf|)(z_0)\leq c\left(M_{ R}(|\nabla_vf|)(z_0)+M^{\#,2}_{{\beta},R}(|G|)(z_0)\right)
    \end{align*}
    for some constant $c=c(n,\Lambda,{\beta})$.
\end{lemma}
\begin{proof}
    Let us fix $\rho\in(0,1/4]$. By \eqref{ineq0.decay2} and \eqref{ass.pointh}, we deduce that for any $r\in(0,R]$,
    \begin{equation*}
    \begin{aligned}
        \frac{E(\nabla_vf;Q_{\rho r}(z_0))}{(\rho r)^{{\beta}}}&\leq c\rho^{\alpha-{\beta}}\frac{E(\nabla_vf;Q_r(z_0))}{r^{{\beta}}}+c\rho^{-(4n+2+\beta)}\dashint_{Q_r(z_0)}{|\nabla_vf|}\,dz\\
        &\quad+c\rho^{-(2n+1+{\beta})}\frac{E_2(G;Q_r(z_0))}{r^{{\beta}}}
    \end{aligned}
    \end{equation*}
    holds, where $c=c(n,\Lambda)$. Therefore, by taking a sufficiently small constant  $\rho=\rho(n,\Lambda,{\beta})$ such that $c\rho^{\alpha-{\beta}}\leq 1/4$,
    we derive 
    \begin{align*}
        M^{\#}_{{\beta},R}(|\nabla_vf|)(z_0)\leq \frac{M^{\#}_{{\beta},R}(|\nabla_vf|)(z_0)}{4}+c\left[M_{R}(|\nabla_vf|)(z_0)+M^{\#,2}_{{\beta},R}(G)(z_0)\right]
    \end{align*}
    for some constant $c=c(n,\Lambda,{\beta})$, which completes the proof.
\end{proof}
For the proof of \eqref{ineq1.hold} and \eqref{ineq2.hold}, we need the following lemma.
\begin{lemma}\label{lem.bdddiv}
    For $R>0$ and $z_0=(t_0,x_0,v_0) \in \mathbb{R}^{2n+1}$, let $\mu \in L^2(Q_R(z_0))$, $G \in L^2(Q_R(z_0),\mathbb{R}^n)$ and let $f\in H^1_{\mathrm{kin}}(Q_R(z_0))$ be a weak solution to 
     \begin{align*}
        (\partial_t+v\cdot \nabla_x)f-\divergence_v(a(t,x,v,\nabla_vf))=\mu-\divergence_v G\quad\text{in }Q_R(z_0),
    \end{align*}
    where $a(\cdot)$ satisfies the same assumptions as in Lemma \ref{lem.pointh2}.
    Then we have 
    \begin{equation}\label{est.bdddiv}
    \begin{aligned}
        &R^{\gamma-1}[f]_{C_{\mathrm{kin}}^{\gamma}(Q_{R/2}(z_0))}+[\nabla_vf]_{BMO(Q_{R/2}(z_0))}\\
        &\leq c\left[\left(\dashint_{Q_R(z_0)}|\nabla_vf|^2\,dz\right)^{\frac12}+[G]_{BMO(Q_R(z_0))}+R\| \mu\|_{L^\infty(Q_R(z_0))}\right],
    \end{aligned}
    \end{equation}
    where $\gamma\in(0,1)$ and $c=c(n,\Lambda,\gamma)$. In addition, if $\mu=0$, then we have 
    \begin{align}\label{est2.bdddiv}
        \norm{\nabla_vf}_{L^\infty(Q_{R/2}(z_0))}\leq c\left[\left(\dashint_{Q_R(z_0)}|\nabla_vf|^2\,dz\right)^{\frac12}+R^\beta[G]_{C_{\mathrm{kin}}^\beta(Q_R(z_0))}\right]
    \end{align}
    for some constant $c=c(n,\Lambda,\beta)$, where the constant $\beta$ is given in Lemma \ref{lem.pointh2}.
\end{lemma}
\begin{proof}
In view of Lemma \ref{lem.scale}, we may assume that $Q_R(z_0)=Q_1$.
Let us fix $0<\rho<r\leq R_1$, where $R_1\leq R/1000$ is determined later depending only on $n,\Lambda$ and $\beta$. We will prove that for any $z_1\in Q_{1/2}$,
\begin{equation}\label{ineq00.bdddiv}
\begin{aligned}
    &\int_{Q_\rho(z_1)}|\nabla_vf-(\nabla_vf)_{Q_\rho(z_1)}|^2\,dz\\
    &\leq c(\rho/r)^{4n+2+2\alpha}\int_{Q_r(z_1)}|\nabla_vf-(\nabla_vf)_{Q_r(z_1)}|^2\,dz\\
    &\quad+cr^{4n+2}\left[[G]^2_{BMO(Q_1)}+\|\mu\|^2_{L^\infty(Q_1)}+\int_{Q_1}|\nabla_vf|^2\,dz\right]
\end{aligned}
\end{equation}
holds, where $c=c(n,\Lambda,\beta)$ and the constant $\alpha$ is determined in Theorem \ref{thm.hol}.
We may assume $\rho\leq r/32$.
Let $w$ be a weak solution to \eqref{eq.ivp} with $Q_r(z_0)$ replaced by $Q_{r/4}(z_1)$ and let $g$ be a weak solution to \eqref{eq.comp3} with $Q_{r/4}(z_0)$ replaced by $Q_{r/16}(z_1)$.
We first observe 
\begin{equation}\label{ineq0.bdddiv}
\begin{aligned}
    \int_{Q_\rho(z_1)}|\nabla_vf|^2\,dz&\leq c\int_{Q_\rho(z_1)}|\nabla_v(f-w)|^2\,dz+c\int_{Q_{\rho}(z_1)}|\nabla_v(w-g)|^2\,dz\\
    &\quad+c\int_{Q_{\rho}(z_1)}|\nabla_vg|^2\,dz\coloneqq J_1+J_2+J_3.
\end{aligned}
\end{equation}

In view of \eqref{ineq1.decay2}, we note that
\begin{align*}
    \int_{\mathcal{R}Q_{r/8}(z_1)}|\nabla_v(f-w)|^2\,dz\leq c\int_{\mathcal{R}Q_{r/8}(z_1)}|G-(G)_{Q_r(z_1)}|^2\,dz+c\int_{\mathcal{R}Q_{r/8}(z_1)}|\mu(f-w)|\,dz,
\end{align*}
where $c=c(n,\Lambda)$. By employing Young's inequality and Poincar\'e's inequality together with the fact that $(f-w)(t,x,\cdot)\equiv 0$ on $B_{r/8}(v_1)$, we obtain
\begin{align*}
     \int_{\mathcal{R}Q_{r/8}(z_1)}|\nabla_v(f-w)|^2\,dz\leq c\int_{\mathcal{R}Q_{r/8}(z_1)}|G-(G)_{Q_r(z_1)}|^2\,dz+c\int_{\mathcal{R}Q_{r/8}(z_1)}|r\mu|^2\,dz.
\end{align*}
Therefore, we estimate $J_1$ as 
\begin{align}\label{ineq1.bdddiv}
    J_1\leq c\int_{Q_r(z_1)}|G-(G)_{Q_r(z_1)}|^2\,dz+c\int_{Q_r(z_1)}|r\mu|^2\,dz.
\end{align}
We now use \eqref{ineq1.comp3} and \eqref{ineq1.bdddiv} to see that 
\begin{equation}\label{ineq2.bdddiv}
\begin{aligned}
    J_2&\leq cr^{2\beta}\int_{\mathcal{R}Q_{r/8}(z_1)}|\nabla_vw|^2\,dz\\
    &\leq cr^{2\beta}\left(\int_{Q_r(z_1)}\left[|\nabla_vf|^2+|r\mu|^2\right]\,dz+\int_{Q_r(z_1)}|G-(G)_{Q_r(z_1)}|^2\,dz\right),
\end{aligned}
\end{equation}
where $c=c(n,\Lambda)$. By Theorem \ref{thm.hol}, we have
\begin{align*}
    J_3\leq c\rho^{4n+2}\|\nabla_vg\|^2_{L^\infty(Q_\rho(z_1))}&\leq c(\rho/r)^{4n+2}\int_{Q_{r/32}(z_1)}|\nabla_vg|^2\,dz.
\end{align*}
We further estimate $J_3$ as 
\begin{equation}\label{ineq21.bdddiv}
\begin{aligned}
    J_3&\leq c(\rho/r)^{4n+2}\left[\int_{\mathcal{R}Q_{r/16}(z_1)}|\nabla_v(g-w)|^2\,dz+\int_{\mathcal{R}Q_{r/16}(z_1)}|\nabla_v(w-f)|^2\,dz\right]\\
    &\quad+c(\rho/r)^{4n+2}\int_{Q_r(z_1)}|\nabla_vf|^2\,dz,
\end{aligned}
\end{equation}
where $c=c(n,\Lambda)$. Plugging \eqref{ineq21.bdddiv}, \eqref{ineq1.bdddiv} and \eqref{ineq2.bdddiv} along with the fact that $r\leq 1$ yields
\begin{align*}
    \int_{Q_\rho(z_1)}|\nabla_vf|^2\,dz&\leq c\left[\left(\frac{\rho}r\right)^{4n+2}+r^{2\beta}\right]\int_{Q_r(z_1)}|\nabla_vf|^2\,dz\\
    &\quad+cr^{4n+2-2\gamma}\left[[G]^2_{BMO(Q_{1})}+\|\mu\|_{L^\infty(Q_1)}^2\right],
\end{align*}
where $c=c(n,\Lambda)$ and $\gamma\in(0,1)$. By Lemma \ref{lem.tech2}, there is a small constant $R_1=R_1(n,\Lambda,\beta,\gamma)\in(0,1)$ such that 
\begin{equation}\label{ineq3.bdddiv}
\begin{aligned}
    \int_{Q_r(z_1)}|\nabla_vf|^2\,dz&\leq c\left(r^{4n+2-\gamma}\int_{Q_{R_1}(z_1)}|\nabla_vf|^2\,dz+r^{4n+2-\gamma}\mathcal{M}\right)\\
    &\leq c\left(r^{4n+2-\gamma}\int_{Q_{1}}|\nabla_vf|^2\,dz+r^{4n+2-\gamma}\mathcal{M}\right)
\end{aligned}
\end{equation}
holds for any $r\leq R_1$, where $\mathcal{M}\coloneqq [G]^2_{BMO(Q_{1})}+\|\mu\|_{L^\infty(Q_1)}^2$.
Therefore, by Lemma \ref{lem.spi}, 
\begin{align}\label{ineq31.bdddiv}
    \dashint_{Q_r(z_1)}|f-(f)_{Q_r(z_1)}|^2\,dz\leq cr^{2-\gamma}\left(\int_{Q_{1}}|\nabla_vf|^2\,dz+[G]^2_{BMO(Q_{1})}+\|\mu\|_{L^\infty(Q_1)}^2\right),
\end{align}
which yields $C_{\mathrm{kin}}^{\gamma}$ estimates of $f$.
We next observe  
\begin{align*}
    \int_{Q_\rho(z_1)}|\nabla_vf-(\nabla_vf)_{Q_\rho(z_1)}|^2\,dz
    & \leq c\int_{Q_\rho(z_1)}|\nabla_v(f-w)|^2\,dz+c\int_{Q_{\rho}(z_1)}|\nabla_v(w-g)|^2\,dz\\
    &\quad+c\int_{Q_{\rho}(z_1)}|\nabla_vg-(\nabla_vg)_{Q_\rho(z_1)}|^2\,dz\coloneqq L_1+L_2+L_3.
\end{align*}
We note that the terms $L_1$ and $L_2$ can be estimated as in \eqref{ineq1.bdddiv} and \eqref{ineq2.bdddiv}, respectively. In light of Theorem \ref{thm.hol}, we estimate $L_3$ as
\begin{align*}
    L_3&\leq c(\rho/r)^{4n+2+2\alpha}\int_{Q_r(z_1)}|\nabla_v g-(\nabla_vg)_{Q_r(z_1)}|^2\,dz\\
    &\leq L_1+L_2+c(\rho/r)^{4n+2+2\alpha}\int_{Q_r(z_1)}|\nabla_v f-(\nabla_vf)_{Q_r(z_1)}|^2\,dz
\end{align*}
for some constant $c=c(n,\Lambda)$. Therefore, we have 
\begin{equation}\label{ineq4.bdddiv}
\begin{aligned}
     &\int_{Q_\rho(z_1)}|\nabla_vf-(\nabla_vf)_{Q_\rho(z_1)}|^2\,dz\\
     &\leq c(\rho/r)^{4n+2+2\alpha}\int_{Q_r(z_1)}|\nabla_v f-(\nabla_vf)_{Q_r(z_1)}|^2\,dz\\
     &\quad+cr^{2\beta}\int_{Q_r(z_1)}|\nabla_vf|^2\,dz+r^{4n+2}\left([G]^2_{BMO(Q_1)}+\|\mu\|_{L^\infty(Q_1)}^2\right).
\end{aligned}
\end{equation}
Plugging \eqref{ineq3.bdddiv} with $\gamma=\beta$ into \eqref{ineq4.bdddiv} yields \eqref{ineq00.bdddiv}. Since 
\begin{equation*}
    \varphi(\rho)\coloneqq \int_{Q_\rho(z_1)}|\nabla_vf-(\nabla_vf)_{Q_\rho(z_1)}|^2\,dz
\end{equation*}
is non-decreasing and nonnegative,
we now apply Lemma \ref{lem.tech2} into \eqref{ineq0.bdddiv} to see that 
\begin{align*}
    \int_{Q_r(z_1)}|\nabla_vf-(\nabla_vf)_{Q_r(z_1)}|^2\,dz
    \leq cr^{4n+2}\left[\int_{Q_1}|\nabla_vf|^2\,dz+[G]^2_{BMO(Q_1)}+\|\mu\|_{L^\infty(Q_1)}^2\right],
\end{align*}
for some constant $c=c(n,\Lambda,\beta)$, whenever $r\leq R_1$, where the constant $R_1=R_1(n,\Lambda,\beta)$ is determined in \eqref{ineq3.bdddiv} with $\gamma=\beta$. This implies
\begin{align*}
    [\nabla_vf]_{BMO(Q_{1/2})}\leq c\left[\|\nabla_vf\|_{L^2(Q_1)}+[G]_{BMO(Q_1)}+\|\mu\|_{L^\infty(Q_1)}^2\right]
\end{align*}
for some constant $c=c(n,\Lambda,\beta)$. Using this and \eqref{ineq31.bdddiv}, we get \eqref{est.bdddiv}.

We next observe from \eqref{ineq1.bdddiv} that if $G\in C_{\mathrm{kin}}^{\beta}(Q_1)$, then we have 
\begin{equation}\label{ineq5.bdddiv}
\begin{aligned}
     &\int_{Q_\rho(z_1)}|\nabla_vf-(\nabla_vf)_{Q_\rho(z_1)}|^2\,dz\\
     &\leq c(\rho/r)^{4n+2+2\alpha}\int_{Q_r(z_1)}|\nabla_v f-(\nabla_vf)_{Q_r(z_1)}|^2\,dz\\
     &\quad+cr^{2\beta}\int_{Q_r(z_1)}|\nabla_vf|^2\,dz+r^{4n+2+2\beta}[G]^2_{C_{\mathrm{kin}}^\beta(Q_1)}.
\end{aligned}
\end{equation}
Plugging \eqref{ineq3.bdddiv} into \eqref{ineq5.bdddiv} yields 
\begin{equation}\label{ineq6.bdddiv}
\begin{aligned}
    \int_{Q_\rho(z_1)}|\nabla_vf-(\nabla_vf)_{Q_\rho(z_1)}|^2\,dz&\leq c(\rho/r)^{4n+2+2\alpha}\int_{Q_r(z_1)}|\nabla_vf-(\nabla_vf)_{Q_r(z_1)}|^2\,dz\\
    &\quad+cr^{4n+2+\beta}\left[[G]^2_{C_{\mathrm{kin}}^\beta(Q_1)}+\int_{Q_1}|\nabla_vf|^2\,dz\right]
\end{aligned}
\end{equation}
for some constant $c=c(n,\Lambda,\beta)$. We now use Lemma \ref{lem.tech2} to see that 
\begin{align*}
    \int_{Q_r(z_1)}|\nabla_vf-(\nabla_vf)_{Q_r(z_1)}|^2\,dz\leq cr^{4n+2+\beta}\left[\int_{Q_1}|\nabla_vf|^2\,dz+[G]^2_{C_{\mathrm{kin}}^\beta(Q_1)}\right],
\end{align*}
for some constant $c=c(n,\Lambda,\beta)$. This gives
\begin{align*}
    \|M^{\#}_{\beta/2,1/2}(\nabla_vf)(z)\|_{L^\infty(Q_{1/2})}\leq c\left[\int_{Q_1}|\nabla_vf|^2\,dz+[G]^2_{C_{\mathrm{kin}}^\beta(Q_1)}\right],
\end{align*} which implies $\nabla_vf\in L^\infty_{\mathrm{loc}}$ with the desired estimate by Corollary \ref{cor.ptmax}.
This completes the proof.
\end{proof}

We are now ready to prove Theorem \ref{thm.hold} and Theorem \ref{cor.lind}.
\begin{proof}[Proof of Theorem \ref{thm.hold}]
    By Lemma \ref{lem.pointh2}, we have the desired estimate \eqref{ineq1.hold}. In addition by Lemma \ref{lem.bdddiv}, we have \eqref{ineq2.hold}. It remains to prove \eqref{ineq3.hold}. Let us assume $G\in C_{\mathrm{kin}}^{{\beta}}(Q_{2R}(z_0))$ which yields $M^{\#,2}_{{\beta},R}(G)\in L^\infty_{\mathrm{loc}}(Q_{2R}(z_0))$. Then by Lemma \ref{lem.bdddiv}, we get $\nabla_vf\in L^\infty_{\mathrm{loc}}(Q_{2R}(z_0))$ which implies $M_{R}(\nabla_vf)\in L^\infty_{\mathrm{loc}}(Q_{2R}(z_0))$. Therefore, \eqref{ineq1.hold} implies $M^{\#}_{{\beta},R}(\nabla_v f)\in L^\infty_{\mathrm{loc}}(Q_{2R}(z_0))$. Together with Corollary \ref{cor.ptmax}, we arrive at $\nabla_vf\in C_{\mathrm{kin}}^{{\beta}}(Q_{R/8}(z_0))$. The desired implication \eqref{ineq3.hold} now follows in light of standard covering arguments. This completes the proof.
\end{proof}
\begin{proof}[Proof of Theorem \ref{cor.lind}.]
    Since we have \eqref{ineq1.lin}, by employing Theorem \ref{thm.hold} with $\alpha$ replaced by $1$, all the results given in Theorem \ref{thm.hold} directly follow for any ${\beta}\in(0,1)$. We are now going to prove \eqref{ineq.lind}. 
    
    Without loss of generality, we can assume that $R=1$ and $z_0=0$. We first observe that
    \begin{align*}
        (\partial_t+v\cdot\nabla_x)\delta_h^xf-\divergence_v(a(t,x,v)\delta_h^xf)&=\divergence_v((a(t,x+h,v)-a(t,x,v))\nabla_vf^x_h)\\
        &\quad-\divergence_v(\delta_h^x G) \quad \textnormal{weakly in } Q_{3/2},
    \end{align*}
    for $|h|$ small, where $f^x_h(t,x,v)\coloneqq f(t,x+h,v)$. We rewrite this as
    \begin{align*}
       (\partial_t+v\cdot\nabla_x)\overline{f}-\divergence_v(A(t,x,v)\nabla_v\overline{f})=\divergence_v(\overline{G}) \quad \textnormal{weakly in } Q_{3/2},
    \end{align*}
    where 
    \begin{equation*}
        \overline{f}\coloneqq\frac{\delta_h^xf}{|h|^{\beta/3}}\quad\text{and}\quad \overline{G}\coloneqq\frac{(A(t,x+h,v)-A(t,x,v))\nabla_vf^x_h}{|h|^{\beta/3}}-\frac{\delta_h^xG}{|h|^{\beta/3}}.
    \end{equation*}
    Since $\nabla_vf\in L^\infty(Q_{3/2})$ and $\frac{|A(t,x+h,v)-A(t,x,v)|}{|h|^{\beta/3}}+\frac{|\delta_h^xG|}{|h|^{\beta/3}}\in L^\infty(Q_{3/2})$,
    by Lemma \ref{lem.bdddiv}, we obtain
    \begin{align*}
        \overline{f}\in C_{\mathrm{kin}}^{1-\epsilon}(Q_{1})
    \end{align*}
    for any $\epsilon>0$. We now use \cite[Lemma A.1.2]{FerRos24} to get the desired result \eqref{ineq.lind}.
\end{proof}

\appendix

\section{Function spaces}\label{appen}
In this appendix, we will construct a regularized cylinder of $Q_r(z_0)$, and deduce several properties of functions $f\in H^1_{\mathrm{kin}}(Q_r(z_0))$.

Let us recall the cylinder $Q_{r,v_0}^{t,x}(t_0,x_0)$ which is defined in \eqref{defn.txcy}. We are going to find a smooth domain $W_{r,v_0}(t_0,x_0)$ such that 
\begin{align}\label{appen.fir}
    Q_{r,v_0}^{t,x}(t_0,x_0)\subset  W_{r,v_0}(t_0,x_0)\subset W_{5r/4,v_0}(t_0,x_0)\subset Q_{2r,v_0}^{t,x}(t_0,x_0).
\end{align}
To this end, we define a set
\begin{equation*}
    \overline{W}=\left\{(t,x)\in\bbR\times \bbR^n\,:\,|x|\leq 1+\sqrt{\frac{-1}{4\log\left(e^{-4}-e^{-1/|t|^2}\right)}}\text{ and }|t|\leq 1/2\right\}
\end{equation*}
to see that 
\begin{align*}
    \partial \overline{W}&=\{(t,x)\in\bbR^{n+1}\,:t=\pm1/2\text{ and } |x|\leq1\}\\
    &\quad\cup\left\{(t,x)\in\bbR^{n+1}\,:\, t=\pm\sqrt{\frac{-1}{\log\left(e^{-4}-e^{-1/(4(|x|-1)^2)}\right)}}\text{  and }1\leq|x|\leq\frac54\right\}.
\end{align*}
 Therefore, we observe that $\overline{W}\subset\bbR^{n+1}$ is a smooth domain and 
\begin{equation*}
    I_{1}(1/2)\times B^x_1\subset \overline{W}\subset I_{5/4}(1/2)\times B^x_{5/4},
\end{equation*}
where the ball $B_r^x$ is defined in \eqref{defn.xball}. We point out that the construction indeed uses a subset of the set $\{(t,x)\,:\, e^{-1/|t|^2}=e^{-4}-e^{-1/(|x|-1)^2}\}$. 

For any $r>0$ and $z_0\in\bbR^{2n+1}$, we now consider a mapping $F_{r,z_0}:\bbR^{n+1}\to\bbR^{n+1}$ defined by 
\begin{equation*}
    F_{r,z_0}(t,x)=(r^{2}(t-1/2)+t_0,r^3x+x_0+r^2(t-1/2)v_0),
\end{equation*}
which is a $C^{\infty}$-diffeomorphism and 
\begin{equation*}
    F_{r,z_0}(I_{1}(1/2)\times B_1^x)=Q^{t,x}_{r,v_0}(t_0,x_0).
\end{equation*}

We now define
\begin{equation}\label{defn.wrv0}
    W_{r,v_0}(t_0,x_0)\coloneqq F_{r,z_0}(\overline{W}),
\end{equation}
which is smooth in $\bbR^{n+1}$ and satisfies \eqref{appen.fir}.

Using this, we will prove that for any
\begin{align}\label{appen.sec}
    f\in H^1_{\mathrm{kin}}(Q_{2r}(z_0))\Longrightarrow f\in \mathcal{W}(\mathcal{R}Q_r(z_0)),
\end{align}
where $\mathcal{R}Q_r(z_0)= W_{r,v_0}(t_0,x_0)\times B_r(v_0)$ and the space $\mathcal{W}(\mathcal{R}Q_r(z_0))$ is the closure of $C^\infty(\overline{\mathcal{R}Q_r(z_0)})$ with respect to the norm of $H^1_{\mathrm{kin}}$. Indeed, the regular set $\mathcal{R}Q_r(z_0)$ and the space $\mathcal{W}(\mathcal{R}Q_r(z_0))$ are defined in \eqref{defn.regcy} and Section \ref{setup}, respectively. As in the proof of \cite[Lemma 3.10]{AveHou24} together with the fact that $W_{5r/4,v_0}(t_0,x_0)$ and $B_{5r/4}(v_0)$ are smooth domains in $\bbR^{n+1}$ and $\bbR^n$, respectively, we observe that there is a sequence of smooth function $\{f_k\}_k$ such that $f_k\in C^{\infty}(W_{5r/4,v_0}(t_0,x_0)\times B_{5r/4}(v_0))\cap C(\overline{W_{5r/4,v_0}(t_0,x_0)\times B_{5r/4}(v_0)})$ and
\begin{equation*}
    f_k\to f \quad\text{in }H^1_{\mathrm{kin}}(W_{5r/4,v_0}(t_0,x_0)\times B_{5r/4}(v_0)).
\end{equation*}
Thus, we have $f_k\in C^\infty(\overline{W_{r,v_0}(t_0,x_0)\times B_{r}(v_0)})$ and
\begin{equation*}
    f_k\to f \quad\text{in }H^1_{\mathrm{kin}}(W_{r,v_0}(t_0,x_0)\times B_{r}(v_0)).
\end{equation*}
Thus, \eqref{appen.sec} follows.

Using \eqref{appen.sec}, we are now able to prove that if $f\in H^1_{\mathrm{kin}}(W\times V)$ is a weak solution to \eqref{eq.main} with $\mu\in L^2(W\times V)$ and $G\in L^2(W\times V,\bbR^n)$, then 
\begin{align}\label{appen.thi}
    f\in L^\infty(I;L^2(U_x\times U_v)),
\end{align}
whenever $I\times U_x\times U_v\Subset W\times V$. To do this, we will prove that for any $Q_{2r}(z_0)\subset W\times V$,
\begin{equation}\label{appen.goal}
    \sup_{t\in I_{r}(t_0)}\|f^2(t,\cdot)\|_{L^1(B_{r}(x_0+(t-t_0)v_0)\times B_r(v_0))}\leq c\int_{Q_r(z_0)}|\nabla_vf|^2+|f|^2+|\mu|^2+|G|^2\,dz
\end{equation}
for some constant $c=c(n,\Lambda)$, where we choose $\epsilon_0\coloneqq\min\{\frac{r}{1000|v_0|},\frac{r}{1000}\}$. 
Let us choose a smooth function $\phi=\phi(t)\geq0$ and a cutoff function $\psi=\psi(z)\in C^\infty_c(Q_{3r/4}(z_0))$ with $\psi\equiv 1$ on $Q_{r/2}(z_0)$. Then we have 
    \begin{align*}
    &    \int_{W_{r,v_0}(t_0,x_0)}\left<(\partial_t+v\cdot\nabla_x )f,f\phi\psi\right>\,dx\,dt\\
    &=\lim_{k\to\infty}\int_{W_{r,v_0}(t_0,x_0)}\left<(\partial_t+v\cdot\nabla_x )f_k,f_k\phi\psi\right>\,dx\,dt\eqqcolon\lim_{k\to\infty}J_k,
    \end{align*}
    where $f_k\in C^1(W_{r,v_0}(t_0,x_0)\times B_r(v_0))\cap C(\overline{W_{r,v_0}(t_0,x_0)\times B_r(v_0)})$ satisfies 
    \begin{equation*}
        f_k\to f\in H_{\mathrm{kin}}^1(W_{r,v_0}(t_0,x_0)\times B_r(v_0)).
    \end{equation*}

    By standard properties of the dual operator $\skp{\cdot}{\cdot}_{H^{-1},H^1}$, we get 
    \begin{align*}
        J_k&=\int_{W_{r,v_0}(t_0,x_0)}\int_{B_r(v_0)}(\partial_t f+v\cdot\nabla_x f_k )f_k\phi\psi\,dz\\
        &=\frac12\int_{B_r(v_0)}\int_{W_{r,v_0}(t_0,x_0)}(\partial_t+v\cdot\nabla_x)(f_k^2\phi\psi)\,dz\\
        &\quad-\frac12\int_{B_r(v_0)}\int_{W_{r,v_0}(t_0,x_0)}f_k^2(\partial_t+v\cdot\nabla_x)(\phi\psi)\,dz\eqqcolon J_{k,1}+J_{k,2}.
    \end{align*}
    Since $\psi\in C^{\infty}_c(Q_{3r/4}(z_0))$, we have 
    \begin{align*}
        J_{k,1}&=\frac12\int_{B_r(v_0)}\int_{\partial W_{r,v_0}(t_0,x_0)}f_k^2\phi\psi (1,v)\cdot N_{t,x}\,dS_{t,x}\,dv\\
        &=\frac12\int_{B_r(v_0)}\int_{A}f_k^2\phi\psi [(1,v)\cdot N_{t,x}]\,dS_{t,x}\,dv\geq0,
    \end{align*}
    where we write $A\coloneqq\{(t_0,x)\,:\, |x|=(3/4)^3\}$.
    Therefore, we have 
    \begin{equation*}
        \lim_{k\to\infty}J_k\geq -\frac12\int_{B_r(v_0)}\int_{W_{r,v_0}(t_0,x_0)}f^2(\partial_t+v\cdot\nabla_x)(\phi\psi)\,dz\eqqcolon J.
    \end{equation*}
    In addition, by using the equation, we obtain
    \begin{align*}
       \overline{J}\coloneqq -\frac12\int_{B_r(v_0)}\int_{W_{r,v_0}(t_0,x_0)}f^2\partial_t\phi\psi\,dz\leq c\int_{Q_{r}(z_0)}|\nabla_v f|^2+|f|^2+|\mu|^2+|G|^2\,dz.
    \end{align*}
    Let us fix $t_1\in (t_0-(2\epsilon_0)^2,t_0)$ and choose $\epsilon<\frac12\min\{t_1-(t_0-(r/2)^2),t_0-t_1\}$. Then there is a smooth function $\phi(t)$ such that $\phi(t)\equiv 1$ on $t<t_1-\epsilon/2$ and $\phi(t)\equiv 0$ on $t>t_1+\epsilon$ and, $\phi'(t)\geq0$, $\phi'(t)=-\epsilon$ on $t_1-\epsilon/2<t<t_1+\epsilon/2$. In this setting, we have 
    \begin{align*}
        \overline{J}&\geq \int^{t_1}_{t_1-\epsilon/2}\int_{B_r(v_0)}\int_{B_{r}(x_0+v_0(t-t_0))}f^2/\epsilon\,dx\,dv\,dt
    \end{align*}
    whenever $\epsilon\leq \epsilon_0$. 
    By using Lemma \ref{lem.diff} which will be proved below, we obtain \eqref{appen.goal}. Since we have proved that for any $z_0\in W\times V$ such that $Q_{2r}(z_0)\Subset W\times V$, \eqref{appen.goal} holds, standard covering arguments yield \eqref{appen.thi}.

We end this appendix by mentioning a generalized Lebesgue differentiation theorem. 
\begin{lemma}\label{lem.diff}
    Let $f\in L^1(Q_R(z_0))$. Then we have for a.e.  $t\in I_R(t_0)$,
    \begin{align*}
        \lim_{\epsilon\to0}\dashint_{t-\epsilon}^{t}\|f(\tau,x,v)\|_{L^1(B_{R}(x_0+(t-t_0)v_0)\times B_R(v_0))}\,d\tau
        =\|f(t,x,v)\|_{L^1(B_{R}(x_0+(t-t_0)v_0)\times B_R(v_0))}.
    \end{align*}
\end{lemma}
\begin{proof}
    We may assume $f\in L^1(I_R(t_0);L^1(\bbR^n))$ by taking $f\equiv 0$ on $I_R(t_0)\times \bbR^{2n}\setminus Q_R(z_0)$. Then by \cite[Theorem 1.6]{Sho97} or \cite[Equation (3.8.4) in Corollary 2]{Hil48}, we have 
    \begin{align*}
        &\left\|\lim_{\epsilon\to0}\dashint_{t-\epsilon}^{t}f(\tau,x,v)-f(t,x,v)\,d\tau\right\|_{L^1(\bbR^{2n})}=0\quad\text{a.e. }t\in I_R(t_0).
    \end{align*}
    Therefore, we have 
    \begin{align*}
        &\left|\dashint_{t-\epsilon}^{t}\|f(\tau,x,v)\|_{L^1(\bbR^n)}\,d\tau-\|f(t,x,v)\|_{L^1(\bbR^n)}\right|\\
        &=\left|\dashint_{t-\epsilon}^{t}\|f(\tau,x,v)\|_{L^1(\bbR^n)}-\|f(t,x,v)\|_{L^1(\bbR^n)}\,d\tau\right|
        \leq \left\|\dashint_{t-\epsilon}^{t}f(\tau,x,v)-f(t,x,v)\,d\tau\right\|_{L^1(\bbR^{2n})}.
    \end{align*}
    Using this together with the fact that $f\equiv 0$ on $I_R(t_0)\times \bbR^{2n}\setminus Q_R(z_0)$, we obtain the desired estimate.
\end{proof}

\printbibliography

\end{document}